\documentclass[11pt]{amsart}
\usepackage{amsmath,amssymb,amsthm,amsfonts,setspace,hyperref,dsfont,yhmath}

\newcommand{\NN}{\mathbb{N}}
\newcommand{\RR}{\mathbb{R}}

\newcommand{\CC}{\mathbb{C}}

\newcommand{\ZZ}{\mathbb{Z}}

\newcommand{\norm}[1]{\lVert#1\rVert}
\newcommand{\abs}[1]{\lvert#1\rvert}

\newtheorem{theorem}{Theorem}[section]
\newtheorem{corollary}[theorem]{Corollary}

\newtheorem{lemma}[theorem]{Lemma}
\newtheorem{proposition}[theorem]{Proposition}

\newcommand{\comment}[1]{}
\theoremstyle{definition}
\newtheorem{definition}[theorem]{Definition}
\newtheorem{remark}[theorem]{Remark}

\newtheorem{case}{Case}

\numberwithin{equation}{section}

\DeclareMathOperator{\diag}{diag}

\makeatletter
\renewcommand*\env@matrix[1][*\c@MaxMatrixCols c]{%
  \hskip -\arraycolsep
  \let\@ifnextchar\new@ifnextchar
  \array{#1}}
\makeatother
\begin{document}
\title[Cohomological equation and cocycle rigidity ]
{Cohomological equation and cocycle rigidity of parabolic actions in $SL(n,\RR)$}


\thanks{{\em Key words and phrases:} Higher rank Abelian group actions, cocycle rigidity, induced unitary representation, Mackey theory}

\author[]{ Zhenqi Jenny Wang }
\address{Department of Mathematics\\
        Yale University\\
        New Haven, CT 06520,   USA}
\email{Zhenqi.Wang@yale.edu}

\begin{abstract}For any unitary representation $(\pi,\mathcal{H})$ of $G=SL(n,\RR)$, $n\geq 3$ without non-trivial $G$-invariant vectors,
we study smooth solutions of the cohomological
equation $\mathfrak{u}f=g$ where $\mathfrak{u}$ is a vector in the root space of $\mathfrak{sl}(n,\RR)$ and
$g$ is a given vector in $\mathcal{H}$. We characterize the obstructions to solving the cohomological
equation, construct
smooth solutions of the cohomological equation and obtain tame Sobolev estimates for $f$.

We also study common solutions to (the infinitesimal version of) the cocycle equation $\mathfrak{u}h=\mathfrak{v}g$, where
$\mathfrak{u}$ and $\mathfrak{v}$ are commutative vectors in different root spaces of $\mathfrak{sl}(n,\RR)$ and $g$ and $h$ are given vectors in $\mathcal{H}$.
We give precisely the condition under which the cocycle equation has common solutions: $(*)$ if $\mathfrak{u}$ and $\mathfrak{v}$ embed in $\mathfrak{sl}(2,\RR)\times \RR$, then the common solution exists. Otherwise, we show counter examples in each
$SL(n,\RR)$, $n\geq 3$. As an application, we obtain smooth cocycle rigidity for higher rank parabolic actions over $SL(n,\RR)/\Gamma$, $n\geq 4$ if the Lie algebra of the acting parabolic subgroup contains a pair $\mathfrak{u}$ and $\mathfrak{v}$ satisfying property $(*)$ and prove that the cocycle rigidity fails otherwise. Especially, the
cocycle rigidity always fails for $SL(3,\RR)$.

The main new ingredient in the proof is making use of unitary duals of various subgroup in $SL(n,\RR)$ isomorphic to $SL(2,\RR)\ltimes\RR^2$ or $(SL(2,\RR)\ltimes\RR^2)\ltimes\RR^3$ obtained by Mackey theory.
\end{abstract}

\maketitle

\section{Introduction}
\subsection{Various algebraic actions} We define $\ZZ^k\times \RR^\ell$, $k+\ell\geq 1$ algebraic actions as follows. Let $G$ be a
connected Lie group, $A\subseteq G$ a closed abelian subgroup which
is isomorphic to $\ZZ^k\times \RR^\ell$, $M$ a compact subgroup of
the centralizer $Z(A)$ of $A$, and $\Gamma$ a cocompact torsion free lattice in
$G$. Then $A$ acts by left translation on the compact space
$\mathcal{M}=M\backslash G/\Gamma$. The three specific
types of examples discussed below correspond to:
\begin{itemize}
  \item  for the symmetric space examples take $G$ a semisimple Lie group of the non-compact
type and $A$ a subgroup of a maximal $\RR$-split Cartan subgroup in $G$

\medskip
  \item  for the twisted symmetric space examples take $G=H\ltimes _\rho\RR^m$ or $G=H\ltimes_\rho N$, a semidirect
product of a reductive Lie group $H$ with semisimple factor of the non-compact
type with $\RR^m$ or a simply connected nilpotent group $N$. In this case, $A$ is a subgroup of a maximal $\RR$-split Cartan subgroup in $H$

\medskip
  \item for the parabolic action examples, take $G$ a semisimple Lie group of the non-compact
type and $A$ a subgroup of a maximal abelian unipotent subgroup in $G$
\end{itemize}
In the past two decades various rigidity phenomena for (partially) hyperbolic actions have been well understood. Significant progresses have been made in
the case of cocycle rigidity of (partially) hyperbolic algebraic actions (see \cite{Damjanovic1}, \cite{Kononenko}, \cite{Spatzier1}, \cite{Spatzier2} and \cite{Zhenqi}) for symmetric and twisted symmetric space examples. This is in contrast to the rank-one situation, where Livsic showed that
there is an infinite-dimensional space of obstructions to solving the cohomology
equation for a hyperbolic action by $\RR$ or $\ZZ$. In the higher rank cases, it was showed in above mentioned papers that smooth cocycles over algebraic Anosov (partially hyperbolic) abelian actions are cohomologically constant via smooth transfer
functions. The key ingredient in proofs of \cite{Spatzier1}, \cite{Spatzier2} and \cite{Zhenqi} is the exponential decay of matrix coefficients for the split Cartan action and the main observation in \cite{Damjanovic1}, \cite{Kononenko} is the exponential decay rate along stable and unstable foliations of periodic cycle functionals.
For all these results, the stable and
unstable foliations of the space play a central role.

The classical horocycle flow is the flow on $PSL(2,\RR)/\Gamma$ given by left translation of the one parameter subgroup generated by $\mathcal{U}=\begin{pmatrix}0 & t \\
0 & 0 \end{pmatrix}\in \mathfrak{sl}(2,\RR)$. Horocycle flow, or more generally algebraic parabolic action, possesses very different dynamical behavior with complete
absence of hyperbolicity. In contrast to the hyperbolic cases mentioned above, for parabolic actions most orbits grow ``polynomially"
and matrix coefficients decay  ``polynomially". To handle this  problem, G. Flaminio and L. Forni used representation theory as an essential tool in \cite{Forni} to study the cohomological equation. They characterized the obstructions to solving the cohomological
equation for classical horocycle flows on quotients of $PSL(2,\RR)$ and showed that the space of obstructions to the equation $\mathcal{U}g=f$ (where $f,\,g$ are in a
unitary representation space of $PSL(2,\RR)$ with a spectral
gap) is of infinite countable dimension; and if $f$ is a smooth vector, then $g$ is a smooth vector. (In fact, G. Flaminio and L. Forni showed that there is finite loss of regularity (of Sobolev norms) between $f$ and $g$.)

The approach of \cite{Forni} was further employed
in \cite{Mieczkowski}, \cite{Mieczkowski1} and \cite{Ramirez} to obtain smooth
cocycle rigidity for some models of higher rank parabolic actions. In \cite{Mieczkowski} and \cite{Mieczkowski1}
Mieczkowski considered actions by subgroup
\begin{align*}
U_1=\bigl\{\begin{pmatrix}1 & c \\
0 & 1 \end{pmatrix}\bigl|c\in\CC\bigl\}\qquad\text{ on }SL(2,\CC)/\Gamma
\end{align*}
and by subgroup
\begin{align*}
U_2=\bigl\{\begin{pmatrix}1 & r \\
0 & 1 \end{pmatrix}\times \begin{pmatrix}1 & s \\
0 & 1 \end{pmatrix}\bigl|r,\,s\in\RR\bigl\}\qquad \text{ on }SL(2,\RR)\times SL(2,\RR)/\Gamma.
\end{align*}
Again, the solutions to the
cocycle equation come with some loss of Sobolev order. In \cite{Ramirez} Ramirez
replaced $SL(2,\RR)\times SL(2,\RR)$ with any noncompact simple Lie group with finite center. But this comes with a price:
smoothness of transfer functions follows from the general elliptic PDE result \cite{Spatzier2}, which means that the solution of the cohomological equation  loses at least half of regularity. Along lines similar to the proof given in \cite{Forni}, the results of Mieczkowski and Ramirez are achieved in every non-trivial irreducible component of unitary representation of $SL(2,\RR)\times SL(2,\RR)$ and $SL(2,\CC)$.

The natural difficulty in studying cohomological equations and obtaining cocycle rigidity in higher rank simple (semisimple) Lie groups is related to the complexity of the representation theory tool.
In particular, in above mentioned cases, the solution to the cohomological equation was established when representations of the group satisfy some special properties (there exists an orthogonal basis in each non-trivial irreducible component of $SL(2,\RR)$ or $SL(2,\RR)\times SL(2,\RR)$), or when the unitary dual of the group is not hard to deal with (for the case of $SL(2,\CC)$). In general, the unitary dual of many higher rank almost-simple algebraic
groups is not completely classified. Even when the classification is known, it is too complicated to apply. For example, the method of \cite{Forni} fails when the group is $SL(3,\RR)$ even though the unitary dual of $SL(3,\RR)$ is available from
the literature \cite{Vogen}.

In this paper we characterize the obstructions to solving the cohomological
equation, construct
smooth solutions of the cohomological equation and obtain tame Sobolev estimates for the solution, i.e, there is finite loss of regularity (with respect to Sobolev norms) between the coboundary and the solution.
We also give a precise description of the condition under which the cocycle equation has common solutions.
As an application, we prove the smooth cocycle rigidity for higher rank parabolic actions over $SL(n,\RR)$, $n\geq 4$.
To prove these results we introduce new ingredients in application of representation theory to higher rank simple Lie groups: studying unitary dual of
subgroups $SL(2,\RR)\ltimes\RR^2$ and $(SL(2,\RR)\ltimes\RR^2)\ltimes\RR^3$ in $SL(n,\RR)$ instead of that of $SL(n,\RR)$ itself. We use Mackey theory to study these representations and carry out explicit calculations in each irreducible component that may appear in restricted non-trivial representation of $SL(n,\RR)$. The global property of the solution comes from the fact that there are enough many semidirect product groups containing the one-parameter root subgroup that determines the cohomological equation.
These results are of independent interest and have wide applicability.

Though it is the first time that the semidirect product group plays
central role in studying cohomological equations and rigidity phenomena in dynamical systems, it has many important applications in other
area of mathematics. The pair $SL(2,\RR)\ltimes\RR^2$ has relative Kazhdan's property $(T)$.  One of the first application of this property
was the resolution of the
Ruziewicz problem for $\RR^n$ when $n\geq 3$ which is due to G. A. Margulis (see \cite{Margulis3}). R. Howe on the other
hand used the property $(T)$ of the pair $SL(2,\RR)\ltimes\RR^2$ to show the Kazhdan's property $(T)$ of $SL(n,\RR)$, $n\geq 3$ (see \cite{tan}).
The semidirect products play also an important
role in the paper of Hee oh \cite{oh} where she gave an explicit calculation of Kazhdan's constants and obtained sharper upperbound for matrix coefficients.

\section{Background, definition, and statement of results}
\subsection{Preliminaries on cocycles} Let $\alpha:A\times
\mathcal{M}\rightarrow \mathcal{M}$ be an action of a topological group $A$ on a (compact)
manifold $\mathcal{M}$ by diffeomorphisms. For a topological group
$Y$ a $Y$-valued {\em cocycle} (or {\em an one-cocycle}) over
$\alpha$ is a continuous function $\beta : A\times \mathcal{M}\rightarrow Y$
satisfying:
\begin{align}
\beta(ab, x) = \beta(a, \alpha(b, x))\beta(b, x)
\end{align}
 for any $a, b \in A$. A
homomorphism $s: A\rightarrow Y$ satisfies the cocycle identity by setting $s(a, x)=s(a)$,
and is called a \emph{constant cocycle}. A cocycle is
{\em cohomologous to a constant cocycle} if there exists a homomorphism $s : A\rightarrow Y$ and a
continuous transfer map $H : \mathcal{M}\rightarrow Y$ such that for all $a
\in A $
\begin{align}\label{for:6}
 \beta(a, x) = H(\alpha(a, x))s(a)H(x)^{-1}
\end{align}
\eqref{for:6} is called the cohomology equation.

In particular, a cocycle is a {\em coboundary} if it is cohomologous
to the trivial cocycle $\pi(a) = id_Y$, $a \in A$, i.e. if for all
$a \in A$ the following equation holds:
\begin{align}
 \beta(a, x) = H(\alpha(a, x))H(x)^{-1}.
\end{align}
For more detailed information on cocycles adapted to the present setting
see \cite{Damjanovic1} and \cite{Katok}. 

This paper will be only concerned  with smooth $\CC^k$-valued cocycles over
algebraic parabolic actions on smooth manifolds. By taking component functions we may always assume that $\beta$ is valued on $\CC$. Further, by taking real and imaginary parts, we can extend the results for real valued cocycles as well.
Specifically, $A$ is a subgroup of a maximal abelian unipotent subgroup in $G=SL(n,\RR)$, $n\geq 3$ and the space $X=G/\Gamma$, where $\Gamma\subset G$ is a torsion free lattice. A cocycle is called smooth if the map $\beta:\,A\rightarrow C^\infty(L^2(G/\Gamma))$ is smooth. We can also define $\beta$ to be of class $C^r$.  We
also note that if the cocycle $\beta$ is cohomologous to a constant cocycle,
then the constant cocycle is given by $s(a)=\int_{G/\Gamma}\beta(a,x)dx.$

For a $C^1$ cocycle $\beta$,  we can define the \emph{infinitesimal generator} of the cocycle $\beta$ by
\begin{align*}
\omega(\nu)=\frac{d}{dt}\beta(\exp t\nu)\bigl|_{t=0}
\end{align*}
The cocycle identity and abelianess of $A$ imply that $\omega$ is a closed $1$-form on
the $A$-orbits in $X$. We can also recover $\beta$ from $\omega$ by
\begin{align*}
    \beta(\exp X)=\int_{0}^1\omega(X)\cdot \exp tXdt
\end{align*}
Thus, we can restrict our attention to infinitesimal version of the cohomology equation
$\omega=\eta-dH$, where $\eta$ is another infinitesimal generator of a smooth cocycle and $H$ is the transfer function. Therefore 
a cocycle $\beta$ is cohomologically trivial if the associated $1$-form $\omega$ is exact and the problem of finding which
cocycle is cohomologically trivial boils down to the problem of determing
which closed $1$-form on the orbit foliation is exact. In fact, this point of view is
the most useful for our purposes.

In what follows, $C$ will denote any constant that depends only
  on the given  group $G$. $C_{x,y,z,\cdots}$ will denote any constant that in addition to the
above depends also on parameters $x, y, z,\cdots$.

\subsection{Statement of the results}\label{sec:25} In $\mathfrak{sl}(n,\RR)$, let $\mathfrak{u}_{i,j}$, $i\neq j$ be the elementary $n\times n$ matrix with only one nonzero entry
equal to one, namely, that in the row $i$ and the column $j$ and let $U_{i,j}$ be the one-parameter subgroup generated by $\mathfrak{u}_{i,j}$. For $1\leq i\neq j\leq n$, set
\begin{align*}
E_{i,j}&=\{\mathfrak{u}_{k,\ell}: \mathfrak{u}_{k,\ell}\neq \mathfrak{u}_{i,j},\, [\mathfrak{u}_{k,\ell},\,\mathfrak{u}_{i,j}]=0\text{ and }[\mathfrak{u}_{k,\ell},\,\mathfrak{u}_{j,i}]\neq 0\}\quad\text{ and}\\
\bar{E}_{i,j}&=\{\mathfrak{u}_{k,\ell}: [\mathfrak{u}_{k,\ell},\,\mathfrak{u}_{i,j}]=0\text{ and }[\mathfrak{u}_{k,\ell},\,\mathfrak{u}_{j,i}]= 0\}.
\end{align*}
In fact, $\bar{E}_{i,j}$ consists of all $\mathfrak{u}_{k,\ell}$ such that the subgroups $U_{k,\ell}\times U_{i,j}$ imbed in subgroups of $SL(n,\RR)$ isomorphic to $SL(2,\RR)\times \RR$.

Suppose $(\pi,\,\mathcal{H})$ is a unitary representation  of $G=SL(n,\RR)$, $n\geq 4$ without non-trivial $G$-fixed vectors. Since $U_{i,j}$ is a closed subgroup of $G$, we have a direct integral decomposition
\begin{align*}
    \pi\mid_{\exp(t\mathfrak{u}_{i,j})}=\int_{\widehat{\RR}}\chi(t)du(\chi),\qquad \forall i\neq j.
\end{align*}
where $u$ is a regular Borel measure and
\begin{align*}
    v=\int_{\widehat{\RR}}v_{i,j,\chi}du(\chi),\qquad \forall\, v\in \mathcal{H}.
\end{align*}
Define $D_{i,j}(v,u)(\chi):=\norm{v_{i,j,\chi}}$. If $u$ is the Lebesgue measure, we just write $D_{i,j}(v)(\chi)$.

Our first two results characterize the obstructions to solving the cohomological
equation and obtain Sobolev estimates for the solution. The next theorem  shows that the $U_{i,j}$-invariant distributions are the only obstructions to solving
the cohomology equation for a given vector $g\in \mathcal{H}^\infty$. This result is similar to the rank one cases (see \cite{Forni} and \cite{Mieczkowski}).
\begin{theorem}\label{th:6}
For any unitary representation $(\pi,\,\mathcal{H})$ of $G=SL(n,\RR)$, $n\geq 3$ without non-trivial $G$-fixed vectors and all $g\in \mathcal{H}^\infty$,
\begin{enumerate}
  \item \label{for:33}if the cohomological equation $\mathfrak{u}_{i,j}f=g$, has a solution $f\in \mathcal{H}$, then $f\in \mathcal{H}$ and satisfies
the Sobolev estimate
\begin{align*}
    \norm{f}_s\leq C_s\norm{g}_{s+7}\qquad \forall s\geq 0,
\end{align*}
  \item \label{for:35}if $\mathcal{D}(g)=0$ for all $U_{i,j}$-invariant distributions, then the cohomological equation $\mathfrak{u}_{i,j}f=g$, has a solution $f\in \mathcal{H}^\infty$.

\end{enumerate}
\end{theorem}
It turns out that in higher rank cases, we have a more concrete way to describe the obstructions: they are exactly the spectral space of the one-parameter subgroup $U_{i,j}$ at $0$. Using above nations, we prove:
\begin{theorem}\label{th:9}
\begin{enumerate}
\item \label{for:94} $u$ is absolutely continuous with respect to the Lebesgue measure $d\chi$. Then we can assume
\begin{align*}
    \pi\mid_{\exp(t\mathfrak{u}_{i,j})}=\int_{\widehat{\RR}}\chi(t)d\chi.
\end{align*}

\noindent Further, for any $g\in \mathcal{H}^2$

  \item \label{for:95}if $\mathfrak{u}_{k,\ell}\in E_{i,j}$, then $D_{i,j}(\mathfrak{u}_{k,\ell}^2g)(\chi)$ is almost a continuous function on $\widehat{\RR}$, that is, there exists a continuous function $v$ on $\widehat{\RR}$ such that
  $D_{i,j}(\mathfrak{u}_{k,\ell}^2g)(\chi)=v(\chi)$ for almost all $\chi\in\widehat{\RR}$,
  \medskip
  \item \label{for:96}if $g\in \mathcal{H}^\infty$ and the cohomological equation $\mathfrak{u}_{i,j}f=g$, has a solution $f\in \mathcal{H}$, then
  \begin{align*}
    \lim_{\chi\rightarrow 0}D_{i,j}(\mathfrak{u}_{k,\ell}^2g)(\chi)=0
  \end{align*}
for any $\mathfrak{u}_{k,\ell}\in E_{i,j}$,
\medskip
  \item \label{for:97}if $g\in \mathcal{H}^\infty$ and there exists a pair $\mathfrak{u}_{m,n}$ and $\mathfrak{u}_{m_1,n_1}$ in $E_{i,j}$ with $\mathfrak{u}_{m,n}\in \bar{E}_{m_1,n_1}$ such that
  \begin{align*}
    \lim_{\chi\rightarrow 0}D_{i,j}(\mathfrak{u}_{m,n}^2g)(\chi)=0\quad\text{ and }\quad \lim_{\chi\rightarrow 0}D_{i,j}(\mathfrak{u}_{m_1,n_1}^2g)(\chi)=0,
  \end{align*}
   then the cohomological equation $\mathfrak{u}_{i,j}f=g$, has a solution $f\in \mathcal{H}$.
\end{enumerate}

\end{theorem}
The next three theorems state precisely the conditions under which the (infinitesimal version of) cocycle equation has a common solution.
\begin{theorem}\label{th:8}Suppose $(\pi,\,\mathcal{H})$ is a unitary representation of $G=SL(n,\RR)$, $n\geq 3$ without $G$-fixed vectors and $f,\,g\in \mathcal{H}^\infty$
and satisfy the cocycle equation $\mathfrak{u}_{i,j}f=\mathfrak{u}_{k,\ell}g$, where $[\mathfrak{u}_{i,j}, \mathfrak{u}_{k,\ell}]=0$. Then we have:
\begin{enumerate}
\item (strong cocycle rigidity)\label{for:79}  if $\mathfrak{u}_{i,j} \in \bar{E}_{k,\ell}$, then the cocycle equation has a common solution $h\in \mathcal{H}^\infty$, that is, $\mathfrak{u}_{k,\ell}h=f$ and $\mathfrak{u}_{i,j}h=g$; and $h$ satisfies
the Sobolev estimate
\begin{align*}
    \norm{h}_s\leq C_s\max\{\norm{g}_{s+7}, \,\norm{f}_{s+7}\},\qquad \forall\,s>0.
\end{align*}
\smallskip
  \item (weak cocycle rigidity) \label{for:26} if $\mathfrak{u}_{i,j} \in E_{k,\ell}$ and there exists $p\in\mathcal{H}$ and $\mathfrak{u}_{m,l}\in (E_{k,\ell}\bigcap E_{i,j})\bigcup \mathfrak{u}_{k,\ell}$ such that $g=\mathfrak{u}_{m,l}p$, then the cocycle equation has a common solution $h\in \mathcal{H}^\infty$, that is, $\mathfrak{u}_{k,\ell}h=f$ and $\mathfrak{u}_{i,j}h=g$; and $h$ satisfies
the Sobolev estimate
\begin{align*}
    \norm{h}_s\leq C_s\max\{\norm{g}_{s+7}, \,\norm{f}_{s+7}\},\qquad \forall\,s>0.
\end{align*}
\end{enumerate}
\end{theorem}
Moreover, it turns out that for $G=SL(n,\RR)$, $n\geq 3$, the condition in \eqref{for:79} of above theorem  is in fact the sufficient and necessary condition to guarantee the cocycle rigidity, more precisely, there exist uncountably many irreducible unitary representations of $G$ such that the cocycle rigidity fails if $\mathfrak{u}_{i,j} \in E_{k,\ell}$.

Let $P$ be the maximal parabolic subgroup of $G$ which stabilizes the line $e_1=(\RR,0,\cdots,0)^\tau\in\RR^n$, where $\tau$ is the transpose map. Then $P$ has the form $\begin{pmatrix}a & v \\
0 & A\\
 \end{pmatrix}$, where $v^\tau\in\RR^{n-1}$, $a\in\RR\backslash 0$ and $A\in GL(n-1,\RR)$. For any $t\in\RR$, $\lambda_t^{\pm}$ is the unitary character of $P$ defined by
\begin{align}\label{for:112}
\lambda_t\begin{pmatrix}a & v \\
0 & A\\
 \end{pmatrix}=\varepsilon^{\pm}(a)\abs{a}^{t\sqrt{-1}}
\end{align}
with $\varepsilon^{+}(a)=1$ and $\varepsilon^{-}(a)=\text{sgn}(a)$.

\begin{theorem}\label{th:7}
For any $t\in\RR$, in the unitary representation  $\text{Ind}_P^G(\lambda_t^{\delta})$ $\delta=\pm$,  for each $E_{k,\ell}$ and each $\mathfrak{u}_{i,j} \in E_{k,\ell}$ there exist smooth vectors $f,\,g$ of $\text{Ind}_P^G(\lambda_t^{\delta})$ such that they satisfy the  cocycle equation $\mathfrak{u}_{i,j}f=\mathfrak{u}_{k,\ell}g$, while  neither $\mathfrak{u}_{k,\ell}\omega=f$  nor $\mathfrak{u}_{i,j}\omega=g$ have a solution in the attached Hilbert space of $\text{Ind}_P^G(\lambda_t^{\delta})$.
\end{theorem}
By the theory of theta series, there exists a arithmetic lattice $\Gamma$ such that for some $t\in\RR$ $\text{Ind}_P^G(\lambda_t^{\delta})$ occurs as a subrepresentation of $L^2(G/\Gamma)$. Moreover, every arithmetic lattice in $G$ is commensurable with one of the lattices stated above$^{1}$. \footnotetext[1]{The comment was made by R. Howe and the proof will come in a separated paper.} Since all lattices in $G$
are arithmetic \cite{Margulis}, the earlier statement  can be made much stronger: for any lattice $\Gamma$ of $G$, there is a finite index subgroup $\Gamma_1\subset \Gamma$ such that $\text{Ind}_P^G(\lambda_t^{\delta})$ occurs as a subrepresentation of $L^2(G/\Gamma_1)$. Then we have:
\begin{theorem}\label{th:11}
Let $U\subset SL(n,\RR)$, $n\geq 3$ be a rank-$2$
abelian subgroup generated by $U_{i,j}$ and $U_{k,\ell}$ where $\mathfrak{u}_{i,j} \in E_{k,\ell}$. Then the cocycle rigidity fails for the $U$-action
on $SL(n,\RR)/\Gamma$, where $\Gamma$ is a lattice in $SL(n,\RR)$. Especially, since in $SL(3,\RR)$ non of the higher rank unipotent subgroups can be embedded in $SL(2,\RR)\times \RR$, the cocycle rigidity fails for any abelian parabolic actions.
\end{theorem}

As an application of Theorem \ref{th:8} we have
\begin{theorem}\label{th:10}
Let $U\subset SL(n,\RR)$, $n\geq 3$ be a closed abelian subgroup generated by root subgroups. If $U$ contains a subgroup

a rank-$2$
abelian subgroup generated by $U_{i,j}$ and $U_{k,\ell}$ where $\mathfrak{u}_{i,j} \in \bar{E}_{k,\ell}$ and let $V\subset SL(n,\RR)$ be an abelian unipotent
subgroup containing $U$. Then a smooth $\CC^k$-valued cocycle over the $V$-action
on $SL(n,\RR)/\Gamma$, where $\Gamma$ is a lattice in $SL(n,\RR)$, is smoothly cohomologous to a constant cocycle.
\end{theorem}

The paper is organized as follows: after recalling Meckey theory and basic notations in Section \ref{sec:15}, we give explicit calculations of some representations of
$SL(2,\RR)\ltimes\RR^2$ and $(SL(2,\RR)\ltimes\RR^2)\ltimes\RR^3$; and give a detailed description of group action for $\text{Ind}_P^G(\lambda_t^{\pm})$ in Section \ref{sec:20}; we give the proof of Theorem \ref{th:7} in Section \ref{sec:21};
we study the cohomological equation on $SL(2,\RR)\ltimes\RR^2$, construct smooth solutions and give Sobolev estimates of the solutions; based on the
conclusions for $SL(2,\RR)\ltimes\RR^2$, we prove Theorem \ref{th:6} and weak version of cocycle rigidity for $SL(2,\RR)\ltimes\RR^2$ in Section \ref{sec:16};
we study strong and weak version of cocycle rigidity on $(SL(2,\RR)\ltimes\RR^2)\ltimes\RR^3$ in Section \ref{sec:9}; we use conclusions in Section \ref{sec:9} to prove  Theorem \ref{th:8} in Section \ref{sec:22}; we study dual representation of $SL(2,\RR)\ltimes\RR^2$ and then  prove Theorem \ref{th:9} in Section \ref{sec:23}.
At the end of this paper we prove Theorem \ref{th:10}.

\noindent{\bf Acknowledgements.} I would like to thank Roger Howe for many valuable comments which improved the paper significantly. I would also like to thank Anatole Katok and Giovanni Forni for their helpful comments. Thanks also are due to Livio Flaminio for suggesting a
method of obtaining tame estimates in the centralizer direction.

\section{Preliminaries on unitary representation theory}\label{sec:15}

\subsection{Direct integrals of unitary representations}
\label{sec:5}Let $(Z,\mu)$ be a measure space, where $\mu$ is a $\sigma$-finite positive measure on
$Z$. \emph{A field of Hilbert spaces over $Z$} is a family $(\mathcal{H}(z))_{z\in Z}$, where $\mathcal{H}(z)$ is a
Hilbert space for each $z\in Z$. Elements of the vector space
$\prod_{z\in Z}\mathcal{H}(z)$ are called \emph{vector fields over $Z$}.

A sequence $(x_n)_{n\in\NN}$ of vector fields over $Z$ is called \emph{a fundamental
family of measurable vector fields} if the following properties are satisfied:
\begin{enumerate}
  \item for any $m,\,n\in\NN$, the function $z\rightarrow \langle x_n(z), x_m(z)\rangle$ is measurable;

  \medskip
  \item for every $z\in Z$, the linear span of $\{x_n(z) : n\in \NN\}$ is dense in $\mathcal{H}(z)$.
\end{enumerate}
Fix a fundamental family of measurable vector fields. A vector field $x\in \prod_{z\in Z}\mathcal{H}(z)$ is said to be \emph{a measurable vector field} if all the functions
\begin{align*}
z\rightarrow \langle x(z), x_n(z)\rangle, \qquad n\in\NN
\end{align*}
are measurable. In the sequel, we identify two measurable vector fields which are equal $\mu$-almost everywhere. A measurable vector field $x$ is a \emph{square-integrable vector
field } if
\begin{align*}
    \int_Z\norm{x(z)}d\mu(z)<\infty.
\end{align*}
The set $\mathcal{H}$ of all square-integrable vector fields is a Hilbert space for the inner product
\begin{align*}
    \langle x(z),y(z)\rangle d\mu(z),\qquad x,\,y\in\mathcal{H}.
\end{align*}
We write
\begin{align*}
\mathcal{H}=\int_Z\mathcal{H}(z)d\mu(z)
\end{align*}
and call $\mathcal{H}$ the \emph{direct integral} of the field $(\mathcal{H}(z))_{z\in Z}$ of Hilbert spaces over $Z$. If $\mathcal{H}(z)=\mathcal{K}$ for all $z\in Z$ where $\mathcal{K}$ is a fixed Hilbert
space, we can choose a fundamental family
of measurable vector fields such that the measurable vector fields are the
measurable mappings $Z\rightarrow\mathcal{K}$, with respect to the Borel structure on $\mathcal{K}$
given by the weak topology. However, this is actually the same as the Borel structure defined by the norm
topology \cite[Chapter 2.3]{Zimmer}. Then
\begin{align*}
\int_Z\mathcal{H}(z)d\mu(z)=L^2(Z,\mathcal{K})
\end{align*}
the Hilbert space of all square-integrable measurable mappings $Z\rightarrow\mathcal{K}$.

For every $z\in Z$, let $T(z)$ be a unitary operator on $\mathcal{H}(z)$. We
say that $(T(z))_{z\in Z}$ is \emph{a measurable field of unitary operators} on $Z$ if all the functions
\begin{align*}
z\rightarrow \langle T(z)x(z), y(z)\rangle, \qquad x, y\in \mathcal{H},
\end{align*}
are measurable. In this case, we write
\begin{align*}
 T=\int_ZT(z)d\mu(z).
\end{align*}
\subsection{Unitary dual of $S=SL(2,\RR)$}\label{sec:1} We list the conclusions in \cite{tan}. We choose as generators for $\mathfrak{sl}(2,\RR)$ the elements
\begin{align}\label{for:4}
X=\begin{pmatrix}
  1 & 0 \\
  0 & -1
\end{pmatrix},\quad U=\begin{pmatrix}
  0 & 1 \\
  0 & 0
\end{pmatrix},\quad V=\begin{pmatrix}
  0 & 0 \\
  1 & 0
\end{pmatrix}.
\end{align}
The \emph{Casimir} operator is then given by
\begin{align*}
\Box:= -X^2-2(UV+VU),
\end{align*}
which generates the center of the enveloping algebra of $\mathfrak{sl}(2,\RR)$. The Casimir operator $\Box$
acts as a constant $u\in\RR$ on each irreducible unitary representation space  and its value classifies them into four classes.
For \emph{Casimir parameter} $u$ of $SL(2,\RR)$, let $\nu=\sqrt{1-u}$ be a representation parameter. Then all the irreducible unitary representations of $SL(2,\RR)$
must be equivalent to one the following:
\begin{itemize}
  \item principal series representations $\pi_\nu^{\pm}$, $u\geq 1$ so that
$\nu=i\RR$,
\medskip
  \item complementary series representations $\pi_\nu$, $0 <u< 1$, so that $0 < \nu< 1$,
  \medskip
  \item discrete series representations $\pi_\nu$ and $\pi_{-\nu}$, $u=-n^2+n$, $n\geq 1$, so that $\nu=2n-1$,
  \medskip
  \item the trivial representation, $u=0$.
\end{itemize}
Any unitary representation $(\pi,\mathcal{H})$ of $SL(2,\RR)$ is decomposed into a direct integral (see \cite{Forni} and \cite{Mautner})
\begin{align}\label{for:1}
\mathcal{H}=\int_{\oplus}\mathcal{H}_ud\mu(u)
\end{align}
with respect to a positive Stieltjes measure $d\mu(u)$ over the spectrum $\sigma(\Box)$. The
Casimir operator acts as the constant $u\in \sigma(\Box)$ on every Hilbert space $\mathcal{H}_u$. The
representations induced on $\mathcal{H}_u$ do not need to be irreducible. In fact, $\mathcal{H}_u$ is in general
the direct sum of an (at most countable) number of unitary representations equal
to the spectral multiplicity of $u\in \sigma(\Box)$. We say that \emph{$\pi$ has a spectral gap (of $u_0$)} if $u_0>0$ and $\mu((0,u_0])=0$.

\subsection{Introduction to Mackey representation theory}\label{sec:24} The problem of determining the complete set of equivalence classes of unitary irreducible
representations of a general class of semi-direct product groups has been solved by Mackey \cite{Mac}.
These results are summarized in this section with explicit application to groups
$SL(2,\RR)\ltimes\RR^2$ and $(SL(2,\RR)\ltimes\RR^2)\ltimes\RR^3$ to facilitate the study of cohomological equation and cocycle rigidity
that follows. Let $S$ be a locally compact group with a closed subgroup $H$. Let $\pi$ be a unitary
representations of $H$ on a Hilbert space $\mathcal{H}$. Suppose $S/H$ carries a $S$-invariant $\sigma$ finite measure $\mu$. Choose a Borel map $\Lambda:S/H\rightarrow S$ such that $p\circ \Lambda=Id$, where $p:S\rightarrow S/H$ is the natural projection. The representation $\pi$ on $H$ induces a representation $\pi_1$ on $S$ as:
\begin{align}\label{for:80}
   ( \pi_1(g)f)(\gamma)=\pi\bigl(\Lambda(\gamma)^{-1}s\Lambda(s^{-1}\gamma)\bigl)f(g^{-1}\gamma)
\end{align}
where $s\in S$, $\gamma\in S/H$ and $f\in L^2(S/H,\mathcal{H},\mu)$. More precisely, if $g^{-1}\Lambda(\gamma)$ decomposes as
\begin{align*}
  g^{-1}\Lambda(\gamma)=\bigl(g^{-1}\Lambda(\gamma)\bigl)_\Lambda \bigl(g^{-1}\Lambda(\gamma)\bigl)_H
\end{align*}
where $\bigl(g^{-1}\Lambda(\gamma)\bigl)_\Lambda\in \Lambda(S/H)$ and $\bigl(g^{-1}\Lambda(\gamma)\bigl)_H\in H$, then \eqref{for:80} has the expression
\begin{align*}
   ( \pi_1(g)f)(\gamma)=\pi(\bigl(g^{-1}\Lambda(\gamma)\bigl)_H^{-1})f(\bigl(g^{-1}\Lambda(\gamma)\bigl)_\Lambda).
\end{align*}
The representation $\pi_1$ is unitary and is called \emph{the representation of the group $S$ induced from $\pi$ in the sense of Mackey} and is denoted by Ind$_{H}^S(\pi)$. For the cases of interest to us, the groups are very well behaved and satisfy the requisite
properties.

We list some of the identifications which are commonly used in the theory of unitarily induced representations (see Proposition 5.1.3.2, 5.1.3.5 of \cite{warner}, \cite[p. 123]{Mac1} and Proposition E.2.1 of \cite{Valette}).
\begin{proposition}\label{po:3}
\begin{enumerate}
  \item \label{for:102}(Induction by stages) Let $H$ and $K$ be closed subgroups
of $S$ with $K\subset H$, and let $\tau$ be a unitary representation of $K$. Then $\text{Ind}_H^S(\text{Ind}_K^H(\tau))$
is unitarily equivalent to $\text{Ind}_K^S(\tau)$,
\medskip
  \item \label{for:103}Suppose $\int_Z \tau_zd\mu(z)$ is a unitary representation of $H$, then $\text{Ind}_H^S(\int_Z \tau_zd\mu(z))$ is unitarily equivalent to
  $\int_Z\text{Ind}_H^S( \tau_z)d\mu(z)$,
  \medskip
  \item \label{for:104}Let $\sigma_1$ and $\sigma_2$ be equivalent representations of $H$. Then $\text{Ind}_H^S(\sigma_1)$ and
$\text{Ind}_H^S(\sigma_2)$ are unitarily equivalent.
\end{enumerate}

\end{proposition}

\begin{theorem}\label{th:1}(Mackey theorem, see \cite[Ex 7.3.4]{Zimmer}, \cite[III.4.7]{Margulis}) Let $S$ be a locally compact group and $\mathcal{N}$ be an abelian closed
normal subgroup of $S$. We define the natural action of $S$ on the group of characters $\widehat{\mathcal{N}}$ of the group $\mathcal{N}$ by setting
\begin{align*}
    (s\chi)(\mathfrak{n}):=\chi(s^{-1}\mathfrak{n}s),\qquad s\in S,\,\chi\in \widehat{\mathcal{N}}, \,\mathfrak{n}\in \mathcal{N}.
\end{align*}
Assume that every orbit $S\cdot \chi$, $\chi\in \widehat{\mathcal{N}}$ is locally closed in $\widehat{\mathcal{N}}$. Then for
any irreducible unitary representation $\pi$ of $S$, there is a point $\chi_0\in \widehat{\mathcal{N}}$ with $S_{\chi_0}$
its stabilizer in $S$, a measure $\mu$ on $\widehat{\mathcal{N}}$ and an irreducible unitary representation  $\sigma$ of $S_{\chi_0}$ such that
\begin{enumerate}
  \item $\pi=\text{Ind}_{S_{\chi_0}}^S(\sigma)$,
  \item $\sigma\mid_{\mathcal{N}}=(\dim)\chi_0$,
  \item $\pi(x)=\int_{\widehat{\mathcal{N}}}\chi(x)d\mu(\chi)$, for any $x\in \mathcal{N}$; and $\mu$ is ergodicly supported on the orbit $S\cdot \chi_0$.
\end{enumerate}
\end{theorem}

\subsection{Weak containment and tempered representations }\label{sec:11}
In terms of
representations of the $\ast$-algebra of $S$ (see \cite{Valette} and \cite{Ey}), for two unitary representations $\rho_1$ and $\rho_2$ of $S$,  we say that $\rho_1$ is \emph{weakly contained} in $\rho_2$
if
\begin{align*}
    \norm{\rho_1(f)}\leq \norm{\rho_2(f)},\qquad \forall f\in L^1(S).
\end{align*}
We write for this $\rho_1\prec \rho_2$.

A unitary representation $\rho$ is said to be \emph{tempered } if $\rho$ is weakly contained in the regular representation of $S$. If $S$ is semisimple, then it is well-known that every tempered representation of $S$ has a spectral gap. For example, if $S=SL(2,\RR)$, then the discrete series and principal series representations are  tempered, while the complementary series representations are not (see \cite{tan}). The following follows from (the proof) of \cite[Lemma 14]{Bachir} and \cite[Lemma 6.2]{howe}:
\begin{theorem}\label{th:3}
Let $Z$ be a standard Borel space and $\mu$ a  positive measure on $Z$. Let $S$ be a separable locally compact group and $\pi$ a representation of $S$, and
\begin{align*}
    \pi=\int_Z\pi_xd\mu(x)
\end{align*}
a direct integral decomposition of $\pi$ with respect to a measurable field $z\rightarrow\pi_z$ of representations of $\pi$. Then
\begin{enumerate}
  \item $\pi_z$ is weakly contained in $\pi$ for almost all $z\in Z$;
  \item $\pi$ is tempered if and only if $\pi_z$ for almost all $z\in Z$.
\end{enumerate}
\end{theorem}
Even though it is assumed that $\mu$ is bounded in \cite[Lemma 14]{Bachir}, the proof works for unbounded case as well without any change. On the other hand,  since $Z$ is standard, we can always assume that the measure is bounded, upon passing to one which is finite
and has the same support. This  changes $\pi$ but not the set of irreducible representations weakly contained
in $\pi$.

\subsection{Sobolev spaces and elliptic regularity theorem}\label{sec:17} Let $\pi$ be a unitary representation of a Lie group $G$ with Lie algebra $\mathfrak{g}$ on a
Hilbert space $\mathcal{H}=\mathcal{H}(\pi)$.
\begin{definition}
For $k\in\NN$, $\mathcal{H}^k(\pi)$ consists of all $v\in\mathcal{H}(\pi)$ such that the
$\mathcal{H}$-valued function $g\rightarrow \pi(g)v$ is of class $C^k$ ($\mathcal{H}^0=\mathcal{H}$). For $X\in\mathfrak{g}$, $d\pi(X)$ denotes the infinitesimal generator of the
one-parameter group of operators $t\rightarrow \pi(\exp tX)$, which acts on $\mathcal{H}$ as an essentially skew-adjoint operator. For any $v\in\mathcal{H}$, we also write $Xv:=d\pi(X)v$.
\end{definition}
We shall call $\mathcal{H}^k=\mathcal{H}^k(\pi)$ the space of $k$-times differentiable vectors for $\pi$ or the \emph{Sobolev space} of order $k$. The
following basic properties of these spaces can be found, e.g., in \cite{Nelson} and \cite{Goodman}:
\begin{enumerate}
  \item $\mathcal{H}^k=\bigcap_{m\leq k}D(d\pi(Y_{j_1})\cdots d\pi(Y_{j_m}))$, where $\{Y_j\}$ is a basis for $\mathfrak{g}$, and $D(T)$
denotes the domain of an operator on $\mathcal{H}$.
  \item $\mathcal{H}^k$ is a Hilbert space, relative to the inner product
  \begin{align*}
    \langle v_1,\,v_2\rangle_{G,k}:&=\sum_{1\leq m\leq k}\langle Y_{j_1}\cdots Y_{j_m}v_1,\,Y_{j_1}\cdots Y_{j_m}v_2\rangle+\langle v_1,\,v_2\rangle
  \end{align*}
  \item The spaces $\mathcal{H}^k$ coincide with the completion of the
subspace $\mathcal{H}^\infty\subset\mathcal{H}$ of \emph{infinitely differentiable} vectors with respect to the norm
\begin{align*}
    \norm{v}_{G,k}=\bigl\{\norm{v}^2+\sum_{1\leq m\leq k}\norm{Y_{j_1}\cdots Y_{j_m}v}^2\bigl\}^{\frac{1}{2}}.
  \end{align*}
induced by the inner product in $(2)$. The subspace $\mathcal{H}^\infty$
coincides with the intersection of the spaces $\mathcal{H}^k$ for all $k\geq 0$.

\item $\mathcal{H}^{-k}$, defined as the Hilbert space duals of
the spaces $\mathcal{H}^{k}$, are subspaces of the space $\mathcal{E}(\mathcal{H})$ of distributions, defined as the
dual space of $\mathcal{H}^\infty$.
  \end{enumerate}
We write $\norm{v}_{k}:=\norm{v}_{G,k}$ and $ \langle v_1,\,v_2\rangle_{k}:= \langle v_1,\,v_2\rangle_{G,k}$ if there is no confusion. Otherwise,
we use subscripts to emphasize that the regularity is measured with respect to $G$.

If $G=\RR^n$ and $\mathcal{H}=L^2(\RR^n)$, the square integrable functions on $\RR^n$, then $\mathcal{H}^k$ is the space consisting of all functions on $\RR^n$ whose first $s$ weak derivatives are functions in $L^2(\RR^n)$. In this case, we use the notation $W^k(\RR^n)$ instead of $\mathcal{H}^k$ to avoid confusion. For any open set $\mathcal{O}\subset\RR^n$, $\norm{\cdot}_{(C^r,\mathcal{O})}$ stands for $C^r$ norm for functions having continuous derivatives up to order $r$ on $\mathcal{O}$. We also write $\norm{\cdot}_{C^r}$ if there is no confusion.

We list the well-known elliptic regularity theorem which will be frequently
used in this paper (see \cite[Chapter I, Corollary 6.5 and 6.6]{Robinson}):
\begin{theorem}\label{th:4}
Fix a basis $\{Y_j\}$ for $\mathfrak{g}$ and set $L_{2m}=\sum Y_j^{2m}$, $m\in\NN$. Then
\begin{align*}
    \norm{v}_{2m}\leq C_m(\norm{L_{2m}v}+\norm{v}),
\end{align*}
where $C_m$ is a constant only dependent on $m$ and $\{Y_j\}$.
\end{theorem}
Suppose $\Gamma$ is an
irreducible torsion-free cocompact lattice in $G$. Denote by $\Upsilon$ the regular representation of $G$ on $\mathcal{H}(\Upsilon)=L^2(G/\Gamma)$. Then we have the following subelliptic regularity theorem (see \cite{Spatzier2}):
\begin{theorem}\label{th:5}
Fix $\{Y_j\}$ in $\mathfrak{g}$ such that commutators of $Y_j$ of length at most $r$ span $\mathfrak{g}$. Also set $L_{2m}=\sum Y_j^{2m}$, $m\in\NN$. Suppose $f\in\mathcal{H}(\Upsilon)$ or a distribution on $G/\Gamma$. If $L_{2m}f\in \mathcal{H}(\Upsilon)$ for any $m\in\NN$, then $f\in \mathcal{H}(\Upsilon)$ and satisfies
\begin{align*}
\norm{f}_{\frac{2m}{r}-1}\leq
C_m(\norm{L_{2m}f}+\norm{f}),
\end{align*}
where $C_m$ is a constant only dependent on $m$ and $\{Y_j\}$.
\end{theorem}
\begin{remark}
The elliptic regularity theorem is a general property, while the subelliptic regularity theorem can't be applied without adopting extra
assumption. For example, if $G/\Gamma$ is non-compact then the above theorem fails.
\end{remark}
\subsection{Direct decompositions of Sobolev space}\label{sec:3}
For any Lie group $G$ of type $I$ and its unitary representation $\rho$, there is a decomposition of $\rho$ into a direct integral
\begin{align}\label{for:66}
 \rho=\int_Z\rho_zd\mu(z)
\end{align}
of irreducible unitary representations for some measure space $(Z,\mu)$ (we refer to
\cite[Chapter 2.3]{Zimmer} or \cite{Margulis} for more detailed account for the direct integral theory). All the operators in the enveloping algebra are decomposable with respect to the direct integral decomposition \eqref{for:66}. Hence there exists for all $s\in\RR$ an induced direct
decomposition of the Sobolev spaces:
\begin{align}\label{for:67}
\mathcal{H}^s=\int_Z\mathcal{H}_z^sd\mu(z)
\end{align}
with respect to the measure $d\mu(z)$.

The existence of the direct integral decompositions
\eqref{for:66}, \eqref{for:67} allows us to reduce our analysis of the
cohomological equation to irreducible unitary representations. This point of view is
essential for our purposes.
\subsection{The Fourier transform}
Let $\mathcal{N}$ be a locally compact abelian group with a Haar measure $d\mathfrak{n}$ and denote by $\widehat{\mathcal{N}}$ its dual group.  The Fourier transform of $L^1(\mathcal{N})$ is obtained by restriction:
\begin{align*}
    \widehat{f}(\chi)=\int_{\mathcal{N}}f(\mathfrak{n})\overline{\chi(\mathfrak{n})}d\mathfrak{n},\qquad f\in L^1(\mathcal{N}),
\end{align*}
the bar denoting complex conjugation. In particular, $\widehat{f}$ belongs to $C_0(\widehat{\mathcal{N}})$ for all $f\in L^1(\mathcal{N})$, where $C_0(\mathcal{N})$ is the space of complex-valued continuous functions vanishing at infinity \cite[pg. 93]{Folland1}. The space of functions $\mathcal{S}(\mathcal{N})$,
known as the \emph{Schwartz-Bruhat space} of $\mathcal{N}$ (rapidly decreasing functions on $\mathcal{N}$), is defined such that it has the property: the Fourier transform induces a topological isomorphism
$\mathcal{S}(\mathcal{N})\cong \mathcal{S}(\widehat{\mathcal{N}})$.
\begin{theorem} \label{th:1}For a suitable normalization of the dual Haar measure $d\widehat{\mathfrak{n}}$ on $\widehat{\mathcal{N}}$, we have:
\begin{enumerate}
  \item  The Fourier transform $f\rightarrow \widehat{f}$ from $L^1(\mathcal{N})\bigcap L^2(\mathcal{N})$ to $L^2(\widehat{\mathcal{N}})$ extends to an isometry from $L^2(\mathcal{N})$ onto $L^2(\widehat{\mathcal{N}})$.
      \medskip
  \item  If $f\in L^1(\mathcal{N})$ and $\widehat{f}\in L^1(\widehat{\mathcal{N}})$, then
for almost every $\mathfrak{n}\in \mathcal{N}$, $f(\mathfrak{n})=\int_{\widehat{\mathcal{N}}}\chi(\mathfrak{n})\widehat{f}(\chi)d\chi$.
\medskip

\item  Every $\mathfrak{n}\in\mathcal{N}$ defines a unitary character $\eta(\mathfrak{n})$ on $\widehat{\mathcal{N}}$ by the formula
$\eta(\mathfrak{n})(\chi)=\chi(\mathfrak{n})$ for any $\chi\in \widehat{\mathcal{N}}$. The canonical group homomorphism
$\eta:\mathcal{N}\rightarrow \widehat{\widehat{\mathcal{N}}} $ is an isomorphism of topological groups.
\end{enumerate}
\end{theorem}
$(1)$, $(2)$ and $(3)$ in above theorem are called Plancherel's Theorem, Fourier Inversion Theorem and Pontrjagin's Duality Theorem respectively. From Plancherel's Theorem, we see that Fourier Inversion Theorem extends to $L^2(\mathcal{N})$. We can and will always identify $\widehat{\widehat{\mathcal{N}}}$ with $\mathcal{N}$ and will take the normalized dual Haar measure $d\widehat{\mathfrak{n}}$ on $\widehat{\mathcal{N}}$ (relative to $d\mathfrak{n}$ on $\mathcal{N}$).
\subsection{Group algebra of locally compact groups} \label{sec:12}Let
 $S$ be a locally compact group, with a left invariant Haar measure $ds$. The \emph{convolution } $f_1\ast f_2$ of two functions $f_1,\,f_2\in L^1(S)$ is defined by
 \begin{align*}
f_1\ast f_2(h)=\int_Sf_1(s)f_2(s^{-1}h)ds.
 \end{align*}
 The group convolution algebra $L^1(S)$, equipped with the involution $f\rightarrow f^\ast$, where
 \begin{align*}
    f^\ast(s)=\delta_S(s^{-1})\overline{f}^\vee(s),\qquad \forall s\in S,
 \end{align*}
$\delta_S$ denoting the modular function of group $S$ and $^\vee$ denoting reflection ($f^\vee(s)=f(s^{-1})$ for all $s\in S$), is a Banach $^\ast$-algebra.

 Let $\pi$ be a unitary representations  of $S$ on a Hilbert space $\mathcal{H}$ with inner product $\langle\,,\,\rangle$. The representation of $\pi$ extends to a $^\ast$-representation of $L^1(S)$: for any $f_1,\,f_2\in L^1(S)$
 \begin{align}\label{for:85}
 \pi(f_1\ast f_2)=\pi(f_1)\pi(f_2)\quad\text{ and }\quad\pi(f^\ast)=\pi(f)^\ast
 \end{align}
where $\pi(f)^\ast$ denotes the adjoint operator of $\pi(f^\ast)$ and $\pi(f)$ is the operator on $\mathcal{H}$ for which
\begin{align*}
\langle \pi(f)v,w\rangle=\int_{S}f(s)\langle \pi(s)v,w\rangle ds, \qquad \forall v,\,w\in \mathcal{H}
\end{align*}
for any $f\in L^1(S)$.

In particular, for the left regular representation $\Upsilon$, $\Upsilon(f)$ is the operator of left convolution by $f$ on $L^2(S)$: $\Upsilon(f)g=f\ast g$ for any $g\in L^2(S)$.

Let $\mathcal{N}$ be an abelian closed subgroup of $S$. For $\xi,\,\eta\in \mathcal{H}$, consider
the corresponding matrix coefficient of $\pi\mid_\mathcal{N}$:
\begin{align*}
\phi_{\xi,\eta}(\mathfrak{n})=\langle \pi(\mathfrak{n})\xi,\,\eta\rangle,\quad\text{ for any }\mathfrak{n}\in \mathcal{N}.
\end{align*}
There exists a regular Borel measure $\mu$ on $\widehat{\mathcal{N}}$, called \emph{the associated measure of $\pi$ (with respect to $\widehat{\mathcal{N}}$)}, such that $\xi=\int_{\widehat{\mathcal{N}}}\xi_\chi d\mu(\chi)$, and
\begin{align}\label{for:84}
    \phi_{\xi,\eta}(\mathfrak{n})&=\int_{\widehat{\mathcal{N}}}\chi(\mathfrak{n})\langle \xi_\chi,\,\eta_\chi\rangle d\mu(\chi).
\end{align}
The representation $\pi\mid_\mathcal{N}$ extends to a $^\ast$-representation on $\mathcal{S}(\widehat{\mathcal{N}})$: for any $f\in \mathcal{S}(\widehat{\mathcal{N}})$, $\widehat{\pi}(f)$ is the operator on $\mathcal{H}$ for which
\begin{align*}
\bigl\langle \widehat{\pi}(f)\xi,\eta\bigl\rangle=\int_{\mathcal{N}}\bigl\langle\widehat{f}(\mathfrak{n})\pi(\mathfrak{n})\xi,\eta\bigl\rangle d\mathfrak{n}, \qquad \forall \xi,\,\eta\in \mathcal{H}.
\end{align*}
Then we have
\begin{align}\label{for:89}
\bigl\langle \widehat{\pi}(f)\xi,\eta\bigl\rangle&=\int_{\mathcal{N}}\bigl\langle\widehat{f}(\mathfrak{n})\pi(\mathfrak{n})\xi,\eta\bigl\rangle d\mathfrak{n}
\overset{\text{(1)}}{=}\int_{\mathcal{N}}\int_{\widehat{\mathcal{N}}}\widehat{f}(\mathfrak{n})\chi(\mathfrak{n})\langle \xi_\chi,\,\eta_\chi \rangle d\mu(\chi) d\mathfrak{n}\notag\\
&=\int_{\widehat{\mathcal{N}}}\langle \xi_\chi,\,\eta_\chi \rangle\int_{\mathcal{N}}\widehat{f}(\mathfrak{n})\chi(\mathfrak{n}) d\mathfrak{n}d\mu(\chi)
\overset{\text{(2)}}{=}\int_{\widehat{\mathcal{N}}}f(\chi)\langle \xi_\chi,\,\eta_\chi \rangle d\mu(\chi).
\end{align}
$(1)$ follows from \eqref{for:84} and $(2)$ holds by using Fourier Inversion Theorem. Then we have
\begin{align*}
\norm{\widehat{\pi}(f)}\leq \norm{f}_\infty,\qquad \forall f\in \mathcal{S}(\widehat{\mathcal{N}}),
\end{align*}
which allows us to extend $\widehat{\pi}$ from $\mathcal{S}(\widehat{\mathcal{N}})$ to $L^\infty(\widehat{\mathcal{N}})$  by taking strong limits of operators and pointwise monotone increasing limits of non-negative functions (see \cite{Lang} for a detailed treatment). Hence $\widehat{\pi}$ is a homomorphism of $L^\infty(\widehat{\mathcal{N}})$ to bounded operators on $\mathcal{H}$.
\begin{lemma}\label{le:10}
Suppose $(\pi,\mathcal{H})$ is isomorphic to another unitary representation $(\pi_1,\mathcal{H}_1)$ and the isomorphism is $\mathcal{I}$, then
\begin{enumerate}
\item $\mathcal{I}\bigl(\widehat{\pi}(f)(v)\bigl)=\widehat{\pi}(f)(\mathcal{I}v)$ for any $f\in L^\infty(\widehat{\mathcal{N}})$ and $v\in \mathcal{H}$;
\medskip
\item if the associated measures of $\pi$ and $\pi_1$ are Lebesgue measures and $\xi=\int \xi_\chi d\chi\in\mathcal{H}$ and $\mathcal{I}(\xi)=\int \xi'_\chi d\chi\in\mathcal{H}_1$, then $\norm{\xi_\chi}=\norm{\xi'_\chi}$ for almost all $\chi$.
\end{enumerate}
\end{lemma}
\begin{proof}
$(1)$ is clear from the definition. We just need to show $(2)$. For any Borel set $B\subset \widehat{\mathcal{N}}$, let $X_B$ denote the characteristic function of $B$. Use $L_\xi$ (resp. $L_{\mathcal{I}(\xi)}$) to denote the set of Lebesgue points of the function: $\chi\rightarrow \norm{\xi_\chi}$ (resp. $\chi\rightarrow \norm{\xi'_\chi}$).

Let $B(\lambda,r)$ denote the ball centered at $\lambda$ with radius $r$. Then by using  \eqref{for:89} and  Lebesgue differentiation theorem (see Theorem 7.7 of \cite{Rudin}) for any $\lambda\in L_\xi\bigcap L_{\mathcal{I}(\xi)}$ we have
\begin{align*}
    \norm{\xi_\lambda}^2&=\lim_{r\rightarrow 0}\frac{1}{\mu(B_r)}\int_{B(\lambda,r)}\norm{\xi_y}^2d\mu(y)=\lim_{r\rightarrow 0}\frac{1}{\mu(B_r)}\norm{\widehat{\pi}(X_{B(\lambda,r)})(\xi)}^2\\
    &\overset{\text{(*)}}{=}\lim_{r\rightarrow 0}\frac{1}{\mu(B_r)}\norm{\widehat{\pi}(X_{B(\lambda,r)})(\mathcal{I}\xi)}^2=\lim_{r\rightarrow 0}\frac{1}{\mu(B_r)}\int_{B(\lambda,r)}\norm{\xi'_y}^2d\mu(y)\\
    &=\norm{\xi'_\chi}^2.
\end{align*}
$(*)$ holds since it is a special case of $(1)$. Then we finish the proof.
\end{proof}
\begin{remark}\label{re:3}
For any unitarily equivalent representations $(\pi,\mathcal{H})$ and $(\pi_1,\mathcal{H}_1)$ over $\mathcal{N}$, the associated measures are
absolutely continuous with respect to each other (see \cite[Proposition 2.3.3]{Zimmer}). Hence if one of the associated measures is the Lebesgue measure, so is the other up to an isomorphism.
\end{remark}
\section{Explicit calculations based on Mackey theory}\label{sec:20}
\subsection{Dual action of  $SL(n,\RR)$ on $\RR^n$ for regular representation} Recall notations in Section \ref{sec:24}. Let $H=SL(n,\RR)$, $n\geq 2$, $\mathcal{N}= \RR^n$ and $S=H\ltimes\RR^n$. The action of $H$ on $\RR^n$ is
given by usual matrix multiplication. The group composition law is
\begin{align*}
(g_1,v_1)(g_2,v_2)=(g_1g_2,g_2^{-1}v_1+v_2).
\end{align*}
The dual group $\widehat{\RR^n}$ of $\RR^n$ can be identified with
$\RR^n$ as follows. Fix a unitary character $\zeta$ of the additive group of $\RR$ distinct from the unit character. The mapping
\begin{align*}
\RR^n\rightarrow \widehat{\RR^n},\qquad v\rightarrow \zeta_v
\end{align*}
is a topological group isomorphism, where $\zeta_v(x)$ is defined by $e^{xv\sqrt{-1}}$ (see. \cite[Ch II-5, Theorem 3]{weil}). Under this
identification, the dual action of $H$ on $\widehat{\RR^n}$ corresponds
to the standard adjoint $H$ action $(\rho(g)^{-1})^\tau$ on $\RR^n$.
Therefore the actions of $H$ in $\widehat{\RR^n}$ are algebraic and hence the $H$-orbits on $\widehat{\RR^n}$ are locally
closed \cite{Zimmer}. There are only two orbits of  $S$ acting on $\RR^n$, namely the origin and its
complement. If $\pi$ is an irreducible unitary representation of $S$ such that
$\pi\mid_{\RR^n}(x)=\int_{\widehat{\RR^n}}\chi(x) d\mu(\chi)$ with $\mu$ supported on the origin then $\pi\mid_{\RR^n}$ is trivial, and hence $\pi$
factors to a representation of $H$. If $\mu$ is supported on the complement, then there is no non-trivial $\RR^n$-vectors, and hence $\pi$ is an induced representation.

\subsection{Unitary dual of $S=SL(2,\RR)\ltimes \RR^2$\label{sec:2} of no non-trivial $\RR^2$-fixed vectors}Write $S$ in the form
$\begin{pmatrix}[cc|c]
  a & b & v_1\\
  c & d & v_2
\end{pmatrix}$, where $\begin{pmatrix}
  a & b \\
  c & d
\end{pmatrix}\in SL(2,\RR)$ and $\begin{pmatrix}
  v_1 \\
  v_2
\end{pmatrix}\in \RR^2$. The discussion in previous part shows that for the vector $\begin{pmatrix}
  0 \\
  1
\end{pmatrix}$, its stabilizer is isomorphic to the Heisenberg group
\begin{align*}
N=\Bigl\{\begin{pmatrix}[cc|c]
  1 & x & v_1\\
  0 & 1 & v_2
\end{pmatrix}:\,x,\,v_1,\,v_2\in\RR\Bigl\}.
\end{align*}
Since $SL(2,\RR)\ltimes\RR^2/N$ is isomorphic to $\RR^2\backslash(0,0)$, we choose a Borel section $\Lambda:SL(2,\RR)\ltimes\RR^2/N\rightarrow SL(2,\RR)\ltimes\RR^2$ given by $\Lambda(x,y)=\begin{pmatrix}x & 0 \\
y & x^{-1} \\
 \end{pmatrix}$.  The
action of the group on the cosets is
\begin{align*}
g^{-1}\Lambda(x,y)&=\Lambda(dx-by,ay-cx)\begin{pmatrix}
  1 & \frac{-bx^{-1}}{dx-by} \\
  0 & 1
\end{pmatrix}
\end{align*}
where $g=\begin{pmatrix}
  a & b \\
  c & d
\end{pmatrix}$. The action of the group on the section $\Lambda$ is:
\begin{align*}
    &\Lambda(x,y)^{-1}(g,v)\Lambda\bigl((g,v)^{-1}(x,y)\bigl)\\
    &=\bigg(\begin{pmatrix}1 & \frac{bx^{-1}}{dx-by} \\
0 & 1\\
 \end{pmatrix},\begin{pmatrix}v_1(dx-by)^{-1} \\
v_2(dx-by)-v_1(ay-cx)\\
 \end{pmatrix}\bigg)
\end{align*}
where $v=\begin{pmatrix}
  v_1 \\
  v_2
\end{pmatrix}$. Since the  irreducible representations of $S$ with $\mu$ supported on the orbit of $\begin{pmatrix}
  0 \\
  1
  \end{pmatrix}$ are induced from irreducible representations on $N$ which is isomorphic to $\RR$, by using Theorem \ref{th:1} we have
\begin{lemma}\label{le:1}
The irreducible representations of $SL(2,\RR)\ltimes \RR^2$ without non-trivial $\RR^2$-fixed vectors are parameterized by $t\in\RR$ and the group action is defined by
\begin{gather*}
\rho_t: SL(2,\RR)\ltimes\RR^2\rightarrow \mathcal{B}(\mathcal{H}_t)\\
\rho_{t}(v)f(x,y)=e^{(v_2x-v_1y)\sqrt{-1}}f(x,y),\\
\rho_{t}(g)f(x,y)=e^{\frac{bt\sqrt{-1}}{x(dx-by)}}f(dx-by,-cx+ay);
 \end{gather*}
 and
 \begin{align*}
  \norm{f}_{\mathcal{H}_t}=\norm{f}_{L^2(\RR^2)},
 \end{align*}
 where $(g,v)=\Big(\begin{pmatrix}a & b \\
c & d\\
 \end{pmatrix},\begin{pmatrix}v_1 \\
v_2\\
 \end{pmatrix}\Big)\in SL(2,\RR)\ltimes \RR^2$.

 We choose a basis for the Lie algebra of $\mathfrak{sl}(2,\RR)$ as in \eqref{for:4}
 and a basis for $\RR^2$ to be $Y_1=\begin{pmatrix}1 \\
0\end{pmatrix}$ and $Y_2=\begin{pmatrix}0 \\
1\end{pmatrix}$. Then we get
\begin{gather}\label{for:107}
 X=-x\partial_x+y\partial_y,\quad U=t_0x^{-2}\sqrt{-1}-y\partial_x,\quad V=-x\partial_y\notag\\
 Y_1=-y\sqrt{-1},\quad Y_2=x\sqrt{-1}.
 \end{gather}

\end{lemma}
\subsection{Unitary dual of $S=(SL(2,\RR)\ltimes \RR^2)\ltimes\RR^3$ of no non-trivial $\RR^2$-fixed vectors}\label{sec:10}
We consider the group $(SL(2,\RR)\ltimes \RR^2)\ltimes\RR^3$ which can be expressed in the form $\begin{pmatrix}[ccc|c]
  a & b & u_1& v_1\\
c & d & u_2& v_2\\
0 & 0 & 1 & v_3\\
 \end{pmatrix}$, where $\begin{pmatrix}
  a & b\\
c &  d\\
 \end{pmatrix}\in SL(2,\RR)$, $\begin{pmatrix}u_1 \\
u_2\\
\end{pmatrix}\in\RR^2$ and $\begin{pmatrix}v_1 \\
v_2\\
v_3\\
 \end{pmatrix}\in\RR^3$. Let $L=\begin{pmatrix}v_1 \\
v_2\\
0\\
 \end{pmatrix}\in\RR^3$ be the rank two subgroup of $\RR^3$. Note that $H=SL(2,\RR)\ltimes \RR^2$ and its action on $\RR^3$ is the restriction of the standard representation of $SL(3,\RR)$ on $\RR^3$. We choose a basis for the Lie algebra of $\mathfrak{sl}(2,\RR)\ltimes \RR^2$ to be
\begin{align*}
X&=\begin{pmatrix}1 & 0 & 0\\
0 & -1& 0\\
0 & 0 & 0\\
 \end{pmatrix}&\quad&
 U_1=\begin{pmatrix}0 & 1 & 0\\
0 & 0& 0\\
0 & 0 & 0\\
 \end{pmatrix}&\quad &
 U_2=\begin{pmatrix}0 & 0& 1 \\
0 & 0& 0\\
0 & 0& 0\\
 \end{pmatrix}&\\
 U_3&=\begin{pmatrix}0 & 0 & 0 \\
0 & 0& 1\\
0 & 0& 0\\
 \end{pmatrix}
 &\quad & V_1=\begin{pmatrix}0 & 0 & 0 \\
1 & 0& 0\\
0 & 0& 0\\
 \end{pmatrix};
 \end{align*}
 and a basis for $\RR^3$ to be $Y_1=\begin{pmatrix}1 \\
0\\
0\\
 \end{pmatrix}$, $Y_2=\begin{pmatrix}0 \\
1\\
0\\
 \end{pmatrix}$ and $Y_3=\begin{pmatrix}0 \\
0\\
1\\
 \end{pmatrix}$.

 Next, we will give a detailed description of irreducible representations of $S$ \emph{\textbf{without no non-trivial $L$-fixed vectors}.}

For any $h=\begin{pmatrix}
  a & b & u_1\\
c & d & u_2\\
0 & 0 & 1 \\
 \end{pmatrix}\in H$ and $v=\begin{pmatrix}v_1 \\
v_2\\
v_3\\
 \end{pmatrix}\in\RR^3$, the action of $h^\tau$ on $v$ is:
\begin{align*}
 h^\tau v=(av_1+cv_2,\,bv_1+dv_2,\,u_1v_1+u_2v_2+v_3)^\tau.
\end{align*}
This allows us to completely determine the orbits and the corresponding representations. The orbits fall into two classes:
\begin{itemize}
  \item if $(v_1,v_2,v_3)\neq (0,0,v_3)$, then the orbit is just the whole space except the origin,

\medskip
  \item if $(v_1,v_2,v_3)= (0,0,v_3)$, then the orbit is a single point $(0,0,v_3)^\tau\in\RR^3$ and its stabilizer is $S$.
\end{itemize}
For the second class, the corresponding representations are trivial on $L$. Then we just need to focus on the first class. Using \eqref{for:81} we see that the stabilizer $N$ of $\begin{pmatrix}1\\
 0\\
0 \\
\end{pmatrix}$ in $H$ is $\begin{pmatrix}1 & 0& 0\\
 c & 1& u_2\\
0 & 0& 1\\
\end{pmatrix}$, where $(c,u_2)^\tau\in\RR^2$. Since $H/N$ is isomorphic to $\RR^3\backslash(0,0,0)^\tau$, a Borel section is given by $\Lambda(x,y,z)=\begin{pmatrix}x^{-1} & y& zx^{-1}\\
0 & x & 0\\
0 & 0 & 1\\
 \end{pmatrix}$. For $g=\begin{pmatrix}a & b & u_1\\
c & d & u_2\\
0 & 0 & 1\\
 \end{pmatrix}$, the
action of the group on the cosets is
\begin{align*}
g^{-1}\Lambda(x,y,z)&=\Lambda(D,E,FD)\begin{pmatrix}
  1 & 0 & 0 \\
 -\frac{c}{xD} & 1 & \frac{cu_1-au_2-zx^{-1}c}{D}\\
  0 & 0 & 1
\end{pmatrix}
\end{align*}
where
\begin{align}\label{for:13}
D&= -cy+xa,\qquad E=yd-bx\notag\\
F&=(azd-adu_1x-cybu_2+au_2yd+cu_1bx-czb)D^{-1}.
\end{align}
Then the action of the group on the section $\Lambda$ is:
\begin{align*}
    &\Lambda(x,y,z)^{-1}(g,v)\Lambda\bigl((g,v)^{-1}(x,y,z)\bigl)=(P,V)
\end{align*}
where $v=(v_1,v_2,v_3)^\tau\in\RR^3$,
\begin{align*}
P&=\begin{pmatrix}
  1 & 0 & 0 \\
 \frac{c}{xD} & 1 & \frac{zx^{-1}c-cu_1+au_2}{D}\\
  0 & 0 & 1
\end{pmatrix}\quad\text{and}\quad V=\begin{pmatrix}Dv_1-Ev_2-DFv_3 \\
D^{-1}v_2\\
v_3
 \end{pmatrix}.
\end{align*}
Let
\begin{align}\label{for:3}
    p_1=c(xD)^{-1}\qquad\text{and}\qquad p_2=z(xD)^{-1}c-cu_1D^{-1}+au_2D^{-1}.
\end{align}
Note that the  irreducible representations of $S$ with $\mu$ supported on the orbit of $\begin{pmatrix}
  1 \\
  0\\
  0
\end{pmatrix}$ are induced from irreducible representations on $N$ which is isomorphic to $\RR^2$, then by using Theorem \ref{th:1} we have
\begin{lemma}\label{le:2}
All the irreducible representations of $(SL(2,\RR)\ltimes \RR^2)\ltimes\RR^3$ without non-trivial $L$-fixed vectors are induced representations and parameterized by $t,\,r\in\RR^2$ and the group action is defined by
\begin{gather*}
\Pi_{(t,r)}: (SL(2,\RR)\ltimes \RR^2)\ltimes\RR^3\rightarrow \mathcal{B}(\mathbb{H}_{(t,r)})\\
\Pi_{(t,r)}(v)f(x,y,z)=e^{(xv_1-yv_2-zv_3)\sqrt{-1}}f(x,y,z),\\
\Pi_{(t,r)}(g)f(x,y,z)=e^{(p_1r+p_2t)\sqrt{-1}}f(D,E,FD);
 \end{gather*}
 and
 \begin{align*}
 \norm{f}_{\mathbb{H}_{(t,r)}}=\norm{f}_{L^2(\RR^3)},
 \end{align*}
  where $g=\begin{pmatrix}a & b & u_1\\
c & d & u_2\\
0 & 0 & 1\\
 \end{pmatrix}$ and $v=(v_1,\,v_2,\,v_3)^\tau\in\RR^3$. Here $D,\,E,\,F$ are defined in \eqref{for:13} and $p_1,\,p_2$ are defined in \eqref{for:3}.

Computing derived representations, we get
 \begin{gather}\label{for:7}
 X=x\partial_x-y\partial_y,\qquad U_1=-x\partial_y,\qquad U_2=-x\partial_z,\notag\\
  U_3=y\partial_z+\sqrt{-1}tx^{-1},\notag\\
  V_1=-y\partial_x+\sqrt{-1}(r+tz)x^{-2},\notag\\
  Y_1=x\sqrt{-1},\qquad Y_2=-y\sqrt{-1},\qquad Y_3=-z \sqrt{-1}.
 \end{gather}
\end{lemma}
\begin{remark}\label{re:5}
From the relation
\begin{align*}
  \Pi_{t,r}\bigl(\begin{pmatrix}1 & 0 & c\\
0 & 1 & 0\\
0 & 0 & 1\\
 \end{pmatrix}\bigl)f(x,y,z)=f(x,y,z-cx),\qquad \forall\, c\in\RR,
\end{align*}
we see that the only vector in $\mathbb{H}_{(t,r)}$ fixed by the one-parameter subgroup $\begin{pmatrix}1 & 0 & c\\
0 & 1 & 0\\
0 & 0 & 1\\
 \end{pmatrix}$ is zero, which implies that $\Pi_{t,r}\mid_{SL(2,\RR)\ltimes \RR^2}$ has no non-trivial $\RR^2$-invariant vectors.
\end{remark}
Next, we will give a detailed description of $\text{Ind}_{N}^{G}(1)$, where $G=SL(n,\RR)$, $n\geq 2$ and $N$ is the stabilizer of the vector $(1,0,\cdots,0)^\tau\in \RR^n$ in $G$, which has the form $\begin{pmatrix}1 & v \\
0 & A\\
 \end{pmatrix}$, where $v^\tau\in\RR^{n-1}$ and $A\in SL(n-1,\RR)$. Let $P$ the maximal parabolic subgroup of $G$ which stabilizes the
line $(\RR,0,\cdots,0)^\tau$.
\subsection{Decomposition of $\text{Ind}_{N}^{G}$ into a direct integral}\label{sec:4}
 At first, we calculate $\text{Ind}_N^{P}(1)$. Note that $P/N$ is isomorphic to $\RR\backslash 0$. Choose a section given by
$\Lambda(\delta,x)=\text{diag}\bigl(\text{sgn}(\delta) e^x,\text{sgn}(\delta) e^{-x},1,\cdots,1\bigl)$, where $\delta=\pm 1$. By Mackey theory we see that the group action is defined by
\begin{gather*}
\gamma: P\rightarrow \mathcal{B}(\mathcal{H}_\gamma)\\
\gamma(\begin{pmatrix}a & v \\
0 & A\\
 \end{pmatrix})f(\delta,x)=f(\text{sgn}(a)\delta,\,x-\log\abs{a}),
\end{gather*}
and
\begin{align*}
\norm{f}_{\mathcal{H}_\gamma}=\norm{f(1,\cdot)}_{L^2(\RR)}+\norm{f(-1,\cdot)}_{L^2(\RR)},
\end{align*}
where $a\in \RR\backslash 0$, $A\in GL(n-1,\RR)$ and $v^\tau\in\RR^{n-1}$. Let $\mathcal{H}'_\gamma=L^2(\RR)\times L^2(\RR)$ with norm
 \begin{align*}
\norm{(f,g)}_{\mathcal{H}'_\gamma}=\norm{f-g}_{L^2(\RR)}+\norm{f+g}_{L^2(\RR)}.
 \end{align*}
 The group action is defined by
 \begin{gather*}
\gamma': P\rightarrow \mathcal{B}(\mathcal{H}'_\gamma)\\
\gamma'(\begin{pmatrix}a & v \\
0 & A\\
 \end{pmatrix})(f(x),\,g(x))=\bigl(f(x-\log\abs{a}),\,\text{sgn}(a)g(x-\log\abs{a})\bigl).
\end{gather*}
 Then the map $\mathcal{H}_\gamma\xrightarrow{F} \mathcal{H}'_\gamma$:
 \begin{align*}
 f(\delta,x)\rightarrow\bigl(f(1,x)+f(-1,x),\,f(1,x)-f(-1,x)\bigl)
 \end{align*}
 is a unitary equivalence over $P$. Reformulating terms by using Fourier Inversion Theorem, we have
 \begin{align*}
   (f(x),\,g(x))=\bigl(\frac{1}{\sqrt{2\pi}}\int_{\RR}\hat{f}(t)e^{tx\sqrt{-1}}dt,\,\frac{1}{\sqrt{2\pi}}\int_{\RR}\hat{g}(t)e^{tx\sqrt{-1}}dt\bigl)
   \end{align*}
   and
   \begin{align*}
    &\gamma'(\begin{pmatrix}a & v \\
0 & A\\
 \end{pmatrix})(f(x),\,g(x))\\
 &=\bigl(\frac{1}{\sqrt{2\pi}}\int_{\RR}\hat{f}(t)\abs{a}^{-t\sqrt{-1}}e^{tx\sqrt{-1}}dt,\,\frac{1}{\sqrt{2\pi}}\int_{\RR}\text{sgn}(a)
 \hat{g}(t)\abs{a}^{-t\sqrt{-1}}e^{tx\sqrt{-1}}dt\bigl)
 \end{align*}
 where $\hat{h}(t)=\frac{1}{\sqrt{2\pi}}\int_{\RR}h(x)e^{-tx\sqrt{-1}}dx$ for any $h\in L^2(\RR)$.

Hence we see that $(\gamma',\,\mathcal{H}'_\gamma)$ is unitarily equivalent to $\int_{\RR} \lambda_{-t}^{+}\oplus \lambda_{-t}^{-}dt$ where
$\lambda_t^{\pm}$ is defined in \eqref{for:112} of Section \ref{sec:25}.

\eqref{for:102} of Proposition \ref{po:3} shows that $\text{Ind}_N^{SL(n,\RR)}(1)$ is unitarily equivalent to $\text{Ind}_P^{SL(n,\RR)}(\text{Ind}_N^{P}(1))$. By \eqref{for:104} of Proposition \ref{po:3} and earlier arguments we find that $\text{Ind}_P^{SL(n,\RR)}(\text{Ind}_N^{P}(1))$ is unitarily equivalent to $\text{Ind}_P^{SL(n,\RR)}(\int_{\RR} \lambda_{-t}^{+}\oplus \lambda_{-t}^{-}dt)$. Finally,
\eqref{for:103} of Proposition \ref{po:3} shows that $\text{Ind}_P^{SL(n,\RR)}(\int_{\RR} \lambda_{-t}^{+}\oplus \lambda_{-t}^{-}dt)$ is unitarily equivalent to
$\int_{\RR}\text{Ind}_P^{SL(n,\RR)}( \lambda_{-t}^{+}\oplus \lambda_{-t}^{-})dt$.
\begin{remark}\label{re:6}
When $n=2$ recall the well known fact that for any $t\in\RR$ the principal series representation $\pi_{t\sqrt{-1}}^{\pm}$ of $SL(2,\RR)$ is equivalent to $\text{Ind}_P^{SL(2,\RR)} (\lambda_t^{\pm})$. Then the earlier discussion shows that only principal series representations appear in $\rho_0\mid_{SL(2,\RR)}=\text{Ind}_N^{SL(2,\RR)}(1)$.
\end{remark}
Before proceeding further with the proof of Theorem \ref{th:7}, we list some important properties of representation of semidirect product $SL(n,\RR)\ltimes\RR^n$ without non-trivial $\RR^n$-invariant vectors (see \cite{oh}, \cite{zhenqi1} and \cite{Zimmer}) which will be frequently
used in this paper:
\begin{proposition}\label{cor:1}
For any unitary representation $\pi$ of $SL(n,\RR)\ltimes\RR^n$ without non-trivial $\RR^n$-fixed vectors, where $SL(n,\RR)$ acts on $\RR^n$ as the standard representation, $\pi$ contains no non-trivial $SL(n,\RR)$-fixed vectors either.
\end{proposition}
The proposition is a special case of Lemma 7.4 in \cite{zhenqi1}, which follows from Mackey¡¯s theory and Borel density theorem (see \cite[Theorem 3.2.5]{Zimmer}).

Recall the following direct consequence of the well known Howe-Moore theorem
on vanishing of the matrix coefficients at infinity \cite{howe-moore}: if $G$ is a simple Lie group with finite center and $\rho$ is a unitary representation of $G$ without a non-zero $G$-invariant vector and $M$ is a closed non-compact subgroup of $G$, then $\rho$ has
no $M$-invariant vector.

Since $SL(n,\RR)$, $n\geq 3$ has Kazhdan's property $(T)$ (see \cite{tan} and \cite{Margulis}), by above proposition and Howe-Moore, we see that if $n\geq 3$ $\pi\mid_{SL(n,\RR)}$ has a spectral gap, that is,  $\pi\mid_{SL(n,\RR)}$ is outside a fixed neighborhood of the trivial representation of $SL(n,\RR)$ in the Fell topology. When $n=2$, $SL(2,\RR)$ fails to have property $(T)$, but the following result (see \cite{oh} and \cite{Zimmer}) shows that $\pi\mid_{SL(2,\RR)}$ also behaviors in a similar way:
\begin{proposition}\label{cor:2}
We assume notations in Proposition \ref{cor:1}. If $n=2$, then $\pi\mid_{SL(2,\RR)}$ is  tempered.
\end{proposition}
\begin{remark}\label{re:2}From arguments in Section \ref{sec:11}, we see that $\pi\mid_{SL(2,\RR)}$ only contains the principal series and discrete series of $SL(2,\RR)$. If the attached space of $\pi$ is $\mathcal{H}$ and $\mathcal{H}$ is decomposed into a direct integral as described in \eqref{for:1} of Section \ref{sec:1}
\begin{align*}
\mathcal{H}=\int_{\oplus}\mathcal{H}_ud\mu(u).
\end{align*}
then above discussion shows that $\mu(0,1)=0$.
\end{remark}
We end this section by a standard result about cocycle equation:
\begin{lemma}\label{le:14}
Suppose $(\pi,\mathcal{H})$ is a unitary representation for a Lie group $G$ with Lie algebra $\mathfrak{g}$ and $\mathfrak{u}_1,\,\mathfrak{u}_2\in \mathfrak{g}$. Suppose there is no non-trivial $\mathfrak{u}_2$-invariant vectors (we call $v\in\mathcal{H}$ a $\mathfrak{u}_2$-invariant vector if $\mathfrak{u}_2v=0$). If
$f,\,g\in \mathcal{H}$ satisfies the cocycle equation $\mathfrak{u}_1f=\mathfrak{u}_2g$ and the equations $\mathfrak{u}_1h=g$ has a solution $h\in \mathcal{H}^2$, then $h$ also solves the equation $\mathfrak{u}_1h=g$.
\end{lemma}
\begin{proof}
From $\mathfrak{u}_1h=g$ we have
\begin{align*}
\mathfrak{u}_1\mathfrak{u}_2h=\mathfrak{u}_2(\mathfrak{u}_1h)=\mathfrak{u}_2 g=\mathfrak{u}_1f,
\end{align*}
which implies that $\mathfrak{u}_2h=f$ since there is no non-trivial $\mathfrak{u}_2$-invariant vectors.
\end{proof}
\section{Proof of Theorem \ref{th:7}}\label{sec:21}
At the beginning, we give a detailed description of the group action of $\text{Ind}_P^{G}( \lambda_{t}^{\pm})$. We recall the formula for the induced representation $\text{Ind}_P^{G}( \lambda_{t}^{\pm})$ (see \cite{Kn} and \cite{warner}). Consider the Langlands
decomposition of $P$: $P=MA_PN$. Denote by $\bar{N}$ the unipotent radical of the opposite parabolic subgroup to $P$ with the common Levi subgroup $MA_P$.
Note that $\bar{N}P$ is a dense open submanifold of $G$ whose complement has zero Haar measure. If $g\in\bar{N}P$ decomposes under
the decomposition $\bar{N}P$, we denote by $P(g)$ the $P$-component of $g$. If $g$
decomposes under $\bar{N}MA_PN$ as
\begin{align*}
    g=\bar{n}(g)m(g)\exp a(g)n(g),
\end{align*}
then the action is given by
\begin{align*}
\text{Ind}_P^{G}( \lambda_{t}^{\pm})(g)f(x)=e^{-\delta_0(a(g^{-1}x))}\lambda_{t}^{\pm}(P(g^{-1}x)^{-1})f(\bar{n}(g^{-1}x))
\end{align*}
for any $f\in L^2(\bar{N},\,dx)$ and $x\in \bar{N}$, where $\delta_0$ is the half sum of positive $N$-roots. We write an element of $\bar{N}$ as
$x=(1,x_1,\cdots,x_{n-1})^\tau$. We choose a basis for the Lie algebra of $\mathfrak{sl}(n,\RR)$ to be $\mathfrak{u}_{i,j}$, $1\leq i\neq j\leq n$ and $X_i$, $1\leq i\leq n-1$, where $\mathfrak{u}_{i,j}$ are defined in Section \ref{sec:25} and  $X_{i}=\diag(0,\cdots,\underset{i}{1},\underset{i+1}{-1},\cdots,0)$. Let $g_{i,j}^t=\exp(t\mathfrak{u}_{i,j})$ and $h_i^t=\exp(tX_i)$, $t\in\RR$. The realization of the representation $\text{Ind}_P^{G}( \lambda_{t}^{\pm})$
on $L^2(\bar{N},\,dx)$ can be formulated as follows:
\begin{align*}
&\text{Ind}_P^{G}( \lambda_{t}^{\pm})(h_{i}^s)f(x_1,\cdots,x_{n-1})\\
&=\left\{\begin{aligned} &e^{sn/2}e^{ts\sqrt{-1}}f(e^{2s}x_1,e^{s}x_{2}\cdots,e^{s}x_{n-1}),&\quad &i=1\\
&f(x_1,\cdots,e^{-s}x_{i-1},e^{s}x_{i}\cdots,x_{n-1}),&\quad& i\geq 2;
\end{aligned}
 \right.
\end{align*}
and has the following expressions
\begin{align*}
&\text{Ind}_P^{G}( \lambda_{t}^{\pm})(g_{i,j}^s)f(x_1,\cdots,x_n)\\
    &=\left\{\begin{aligned} &\abs{1-x_{j-1}s}^{-n/2-t\sqrt{-1}}\varepsilon^{\pm}(1-x_{j-1}s)\\
    &\cdot f(\frac{x_1}{1-x_{j-1}s},\cdots,\frac{x_{n-1}}{1-x_{j-1}s}),&\quad &i=1,\,\,j\geq 2,\\
&f(x_1,\cdots,x_{i-1}-sx_{j-1},\cdots,x_{n-1}),&\quad &i\geq 2,\,\,j\neq 1,\\
&f(x_1,\cdots,x_{i-1}-s,\cdots,x_{n-1}),&\quad &i\geq 2,\,\,j= 1.
\end{aligned}
 \right.
\end{align*}
Since the one-parameter subgroups $g_{i,j}^t$ and $h_\ell^t$ generate $SL(n,\RR)$, the actions of these subgroups determine the group action of $SL(n,\RR)$. Computing derived representations, we get
\begin{align}\label{for:9}
X_i=&\left\{\begin{aligned} &(\frac{n}{2}+t\sqrt{-1})+2x_1\partial_{x_{1}}+\sum_{k=2}^{n-1}x_k\partial_{x_{k}},&\qquad \qquad &i=1,\\
&-x_{i-1}\partial_{x_{i-1}}+x_{i}\partial_{x_{i}},&\qquad \qquad &i\geq 2,
\end{aligned}
 \right.
\end{align}
and
\begin{align}\label{for:105}
\mathfrak{u}_{i,j}=&\left\{\begin{aligned} &(\frac{n}{2}+t\sqrt{-1})x_{j-1}+\sum_{k=1}^{n-1}x_{j-1}x_k\partial_{x_{k}},&\qquad \qquad &i=1,\,\,j\geq 2,\\
&-x_{j-1}\partial_{x_{i-1}},&\qquad \qquad &i\geq 2, \,\,j\neq 1,\\
&-\partial_{x_{i-1}},&\qquad \qquad &i\geq 2, \,\,j= 1.
\end{aligned}
 \right.
\end{align}
We are now in a position to proceed with the proof of Theorem \ref{sec:21}.
 Noting that the Weyl group is the symmetric group $S_n$ which operates simply transitive on the set of Weyl chambers, we may assume that one element in the pair $\mathfrak{u}_{i,j}$ and $\mathfrak{u}_{k,\ell}$ is $\mathfrak{u}_{2,1}$. By assumption, the other one is $\mathfrak{u}_{2,j}$ or $\mathfrak{u}_{j,1}$, $j\geq 3$. Let $h(x_1,\cdots,x_{n-1})=p(x_1)p(x_2)\cdots p(x_{n-1})$ where $p$ is a smooth function with compact support over $\RR$ satisfying the following two conditions: $p\geq 0$ and $p=1$ on $[-1,1]$.
\begin{case} The pair is $\mathfrak{u}_{2,1}$ and $\mathfrak{u}_{2,j}$, $j\geq 3$.

 Let $g=h$ and $f=g\cdot x_{j-1}$. From relations in \eqref{for:9} and \eqref{for:105}, it is easy to check that $f$ and $g$ are smooth vectors for $\text{Ind}_P^{G}( \lambda_{t}^{\pm})$. Using relations in \eqref{for:105} we have
\begin{align}\label{for:78}
 \mathfrak{u}_{2,1}f=-\partial_{x_1}f=-x_{j-1}\partial_{x_{1}}g=\mathfrak{u}_{2,j}g.
\end{align}
If $\mathfrak{u}_{2,1}\omega=g$ where $\omega\in L^2(\RR^{n-1})$, then $-\partial_{x_1}\omega=g$.
Taking fourier transformation on factor $x_1$, we have
\begin{align}
-\sqrt{-1}\hat{\omega}_\xi(\xi,x_2,\cdots,x_{n-1})\cdot \xi&=\hat{p}(\xi)p(x_2)\cdots p(x_{n-1})
\end{align}
where
\begin{align*}
\hat{p}(\xi)&=\frac{1}{\sqrt{2\pi}}\int p(y) e^{-y\xi\sqrt{-1}}dy,\qquad \text{ and }\\
\hat{\varpi}_\xi(\xi,x_2\cdots,x_{n-1})&=\frac{1}{\sqrt{2\pi}}\int \omega(x_1,\cdots,x_{n-1}) e^{-x_1\xi\sqrt{-1}}dx_1.
\end{align*}
Then $\hat{\omega}_\xi\in L^2(\RR^n)$ and $\hat{p}$ is a continuous function. Then we have
\begin{align*}
\hat{\omega}_\xi=\bigl(\sqrt{-1}\hat{p}(\xi)\cdot \xi^{-1}\bigl)p(x_2)\cdots p(x_{n-1}).
\end{align*}
Since  $\hat{p}(0)=\frac{1}{\sqrt{2\pi}}\int p(y)dy>0$ by assumption, $\hat{\omega}_\xi\notin L^2(\RR^n)$.  Then we get a contradiction. Note that the equation $\mathfrak{u}_{2,j}\omega=f$ is also equivalent to the earlier equation $-\partial_{x_1}\omega=g$. Thus we proved the claim for the pair  $\mathfrak{u}_{2,1}$ and $\mathfrak{u}_{2,j}$, $j\geq 3$.
\end{case}
\begin{case}The pair is $\mathfrak{u}_{2,1}$ and $\mathfrak{u}_{j,1}$, $j\geq 3$.

Under the permutation $(1,j)$ the pair changes to $\mathfrak{u}_{2,j}$ and $\mathfrak{u}_{1,j}$. Then we can consider the pair $\mathfrak{u}_{2,j}$ and $\mathfrak{u}_{1,j}$ instead. Let
\begin{align*}
 g&=(\frac{n}{2}+t\sqrt{-1})h-h+\sum_{k=1}^{n-1}x_k\partial_{x_{k}}h\quad \text{ and }\quad f=-\partial_{x_{1}}h.
\end{align*}
Obviously, $f$ and $g$ are smooth vectors for $\text{Ind}_P^{G}( \lambda_{t}^{\pm})$. Using relations in \eqref{for:105} we have
\begin{align*}
 \mathfrak{u}_{1,j}f&=-(\frac{n}{2}+t\sqrt{-1})x_{j-1}\partial_{x_{1}}h-\sum_{k=1}^{n-1}x_{j-1}x_k\partial_{x_{k}}\partial_{x_{1}}h\\
 &=-x_{j-1}\partial_{x_{1}}\Bigl((\frac{n}{2}+t\sqrt{-1})h-h+\sum_{k=1}^{n-1}x_k\partial_{x_{k}}h\Bigl)\\
 &=-x_{j-1}\partial_{x_{1}}g=\mathfrak{u}_{2,j}g.
\end{align*}
If $\mathfrak{u}_{2,j}\omega=f$ where $\omega\in L^2(\RR^{n-1})$, then
\begin{align*}
x_{j-1}\partial_{x_{1}}\omega-\partial_{x_{1}}h=0.
\end{align*}
 Let $A_\epsilon=\{x\in\RR:\abs{x}\geq \epsilon\}$ and $1_{A_\epsilon}(x_{j-1})$ be the indicator function for the set $A_\epsilon$. Multiplying  $1_{A_\epsilon}(x_{j-1})$ to each side the above equation we get
\begin{align*}
0&=(x_{j-1}\partial_{x_{1}}\omega-\partial_{x_{1}}h)1_{A_\epsilon}(x_{j-1})\\
&=x_{j-1}\partial_{x_{1}}\bigl(\omega\cdot 1_{A_\epsilon}(x_{j-1})-h \cdot x_{j-1}^{-1} 1_{A_\epsilon}(x_{j-1})\bigl)\\
&=-\mathfrak{u}_{2,j}\bigl(\omega\cdot 1_{A_\epsilon}(x_{j-1})-h \cdot x_{j-1}^{-1} 1_{A_\epsilon}(x_{j-1})\bigl)
\end{align*}
for any $\epsilon>0$. Since $\omega\cdot 1_{A_\epsilon}(x_{j-1})$ and $h \cdot x_{j-1}^{-1} 1_{A_\epsilon}(x_{j-1})$ are both in $L^2(\RR^{n-1})$, by Howe-Moore,
\begin{align*}
\omega\cdot 1_{A_\epsilon}(x_{j-1})-h \cdot x_{j-1}^{-1} 1_{A_\epsilon}(x_{j-1})=0.
\end{align*}
Thus $\omega+h\cdot x_{j-1}^{-1}=0$ follows immediately  from the arbitrariness of $\epsilon$, which means that $h\cdot x_{j-1}^{-1}\in L^2(\RR^{n-1})$. Then we get a contradiction. If $\mathfrak{u}_{1,j}\omega=g$ where $\omega\in L^2(\RR^{n-1})$, then Theorem \ref{th:6} (the proof is in the next section) shows that $\omega$ is a smooth vector for $\text{Ind}_P^{G}( \lambda_{t}^{\pm})$. It follows from Howe-moore and Lemma \ref{le:14} that $\omega$ also solves the equation $\mathfrak{u}_{2,j}\omega=f$, which contradicts the assumption of $f$ by the earlier argument. Hence we proved the claim for the pair  $\mathfrak{u}_{2,1}$ and $\mathfrak{u}_{j,1}$.
\end{case}

\section{Coboundary for the unipotent flow of $SL(n,\RR)$, $n\geq 3$} \label{sec:16}In this section, $G$ always denotes a Lie group $G$ with finite center, $\mathfrak{g}$ denotes its Lie algebra and $(\pi,\,\mathcal{H})$ denotes a non-trivial unitary representation of $G$.  We recall that for a flow $\psi_t$ on $G$ we say that $F\in \mathcal{H}$ is a coboundary for the
flow  if there is a solution
$f\in \mathcal{H}$ to the cohomological equation
\begin{align*}
    \frac{d}{dt}f\circ\psi_t\,\big|_{t=0}=F.
\end{align*}
In this section we will study the solution of the cohomological equation for various types of Lie group $G$.

\subsection{Coboundary for the horocycle
flow of $G=SL(2,\RR)$} Recall notations in Section \ref{sec:1} and \ref{sec:17}. For the classical horocycle flow defined by the $\mathfrak{sl}(2,\RR)$-matrix $U=\begin{pmatrix}
  0 & 1 \\
  0 & 0
\end{pmatrix}$, there is a classification of the obstructions to the solution of the cohomological
equation established by Flaminio and Forni \cite{Forni}. That is, for any $F\in \mathcal{H}^\infty$, we know precisely the condition under which the equation $Uf=F$ has a solution $f$. Let
\begin{align*}
    \mathcal{E}_U(\mathcal{H})=\{\mathcal{D}\in \mathcal{E}(\mathcal{H}): \mathcal{L}_U\mathcal{D}=0 \}\quad\text{and}\quad\mathcal{H}_U^{-k}=\{\mathcal{D}\in \mathcal{H}^{-k}: \mathcal{L}_U\mathcal{D}=0 \}.
\end{align*}

\begin{theorem}\label{th:2} Suppose $\pi$ has a spectral gap of $u_0$. For all $F\in \mathcal{H}^s$, if $\mathcal{D}(g)=0$ for all $\mathcal{D}\in \mathcal{H}_U^{-k}$ and $t<s-1$, then the equation $Uf=F$ has a solution $f\in \mathcal{H}^t$, which satisfies the Sobolev estimates $\norm{f}_t\leq C_{t,s,u_0}\norm{F}_s$.
\end{theorem}
\begin{remark}\label{re:1}
In fact, the above theorem applies to any irreducible
unitarizable representations of $\mathfrak{sl}(2,\RR)$; that is, those representations that arise
as the derivatives of irreducible unitary representations of some Lie group whose
Lie algebra is $\mathfrak{sl}(2,\RR)$. In fact,
all such representations can be realized from irreducible unitary representations
of some finite cover of $SL(2,\RR)$. In turn, all of these are unitarily equivalent to
irreducible representations of $SL(2,\RR)$ itself \cite{tan}.
\end{remark}
\subsection{Coboundary for unipotent flows in any Lie group $G$}\label{sec:6} We present two technical results in this part, which are suggested by L. Flaminio. Lemma \ref{le:10} and the ``centralizer trick" in Proposition \ref{le:6} will pay a key role in next section.
\begin{lemma}\label{le:10}
Suppose $G$ is a simple Lie group and $\pi$ contains no non-trivial $G$-invariant vectors. Also suppose
$\{\exp(tY)\}_{t\in\RR}$ is a non-compact subgroup for some $Y\in \mathfrak{g}$. For any $v_1,\,v_2\in\mathcal{H}$, if $\langle v_1,\,\mathfrak{u}h\rangle=\langle v_2,\,Y\mathfrak{u}h\rangle$ for any $h\in\mathcal{H}^\infty$, then $v_1=-Yv_2$.
\end{lemma}
\begin{proof}
Thanks to Howe-Moore, we see that $\pi$ has
no non-trivial $Y$-invariant vectors. Since the orthogonal complement of $Y$-coboundary
are the $Y$-invariant vectors, which by earlier discussion are zero, we see that $v_1=-Yv_2$.
\end{proof}

Suppose $\mathfrak{u}\in \mathfrak{g}$ is a nilpotent element.  The
Jacobson-Morosov theorem asserts the existence of an element $\mathfrak{u}'\in \mathfrak{g}$
such that $\{\mathfrak{u}, \mathfrak{u}', [\mathfrak{u},\mathfrak{u}']\}$ span a three-dimensional Lie algebra $\mathfrak{g}_\mathfrak{u}$ isomorphic to $\mathfrak{sl}(2,\RR)$. Set $G_\mathfrak{u}$ to be the connected subgroup in $G$ with Lie algebra spanned by $\{\mathfrak{u}, \mathfrak{u}', [\mathfrak{u},\mathfrak{u}']\}$. Since $G$ has finite center, $G_\mathfrak{u}$ is
isomorphic to a finite cover of $PSL(2,\RR)$. We have the following result which can be viewed as an extension of Theorem \ref{th:2}.
\begin{proposition}\label{le:6}
Suppose there is a spectral gap of $u_0$ for $(\pi\mid_{G_\mathfrak{u}},\,\mathcal{H})$. Suppose $g\in \mathcal{H}^\infty$ and $\mathcal{D}(g)=0$ for all $\mathcal{D}\in\mathcal{E}_{U}(\mathcal{H})$, where $U=\{\exp(t\mathfrak{u})\}_{t\in\RR}$.  Fix a norm $\abs{\,\cdot\,}$ on $\mathfrak{g}$. Set
\begin{align*}
\mathfrak{N}_{\mathfrak{u}}=\{Y\in \mathfrak{g}:\abs{Y}\leq 1\text{ and }[Y,\mathfrak{u}]=a\mathfrak{u},\text{ for some  constant }a\in\RR\}.
\end{align*}Then the cohomological equation $\mathfrak{u}f=g$ has a solution $f\in \mathcal{H}$ which satisfies
the Sobolev estimate
\begin{align}\label{for:25}
 \norm{Y^mf}_{G_\mathfrak{u},t}\leq C_{u_0, m,s,t}\norm{g}_{s},\qquad \forall\,Y\in \mathfrak{N}_\mathfrak{u}
\end{align}
if $t+m<s-1$.
\end{proposition}
\begin{proof} As a direct consequence of Theorem \ref{th:2} and Remark \ref{re:1} we see that the cohomological equation $\mathfrak{u}f=g$ has a solution $f\in \mathcal{H}$ with estimates
\begin{align}\label{for:24}
 \norm{f}_{G_\mathfrak{u},t}\leq C_{s,t,u_0}\norm{g}_{s}
\end{align}
if $t<s-1$. As a first step to get the Sobolev estimates along $\mathfrak{N}_\mathfrak{u}$, we prove the following fact:

\noindent\emph{Fact $(*)$: if $\mathcal{D}\in \mathcal{E}_{U}(\mathcal{H})$ then $Y\mathcal{D}\in \mathcal{E}_{U}(\mathcal{H})$ for any $Y\in\mathfrak{N}_{\mathfrak{u}}$.}

By definition $Y\mathcal{D}(h)=-\mathcal{D}(Yh)$ for any $h\in \mathcal{H}^\infty$. Then
\begin{align*}
(\mathfrak{u}Y\mathcal{D})(h)&=\mathcal{D}(Y\mathfrak{u}h)=\mathcal{D}(\mathfrak{u}Yh)+a\mathcal{D}(\mathfrak{u}h)\\
&=-(\mathfrak{u}\mathcal{D})(Yh)-a(\mathfrak{u}\mathcal{D})(h)=0,
\end{align*}
which proves Fact $(*)$.

For any $Y\in\mathfrak{N}_{\mathfrak{u}}$, from Fact $(*)$ we see that $\mathcal{D}(Yg)=0$ for any $\mathcal{D}\in \mathcal{E}_{U}(\mathcal{H})$.
Then Theorem \ref{th:2} and Remark \ref{re:1} imply that the equation $\mathfrak{u}f_1=Yg$ has a solution $f_1\in \mathcal{H}$ with sobolev estimates
\begin{align}\label{for:22}
\norm{f_1}_{G_\mathfrak{u},t}\leq C_{s,t,u_0}\norm{Yg}_{s}\leq C_{s,t,u_0}\norm{g}_{s+1}
\end{align}
if $t<s-1$. On the other hand, for any $h\in\mathcal{H}^\infty$ we have
\begin{align*}
    -\langle f_1,\mathfrak{u}h\rangle&= \langle \mathfrak{u}f_1,h\rangle=\langle Yg,\,h\rangle=-\langle g,\,Yh\rangle=-\langle \mathfrak{u}f,\,Yh\rangle\\
    &=\langle f,\,\mathfrak{u}Yh\rangle=\langle f,\,(Y\mathfrak{u}-a\mathfrak{u})h\rangle.
\end{align*}
This shows that $Yf=f_1-af$ by Lemma \ref{le:10}. From \eqref{for:24} and \eqref{for:22} we have
\begin{align*}
\norm{Yf}_{G_\mathfrak{u},t}=\norm{f_1-af}_{G_\mathfrak{u},t}\leq C_{s,t,u_0}\norm{g}_{s+1}
\end{align*}
if $t<s-1$. Then we just proved \eqref{for:25} when $m=1$. By induction suppose \eqref{for:25} holds when $m\leq k$. Next we will prove the case when  $m=k+1$.  By induction for any $j\geq 1$ we have
\begin{align}\label{for:77}
Y^j\mathfrak{u}=\mathfrak{u}Y^j+p_{j-1}(Y)\mathfrak{u}
\end{align}
where $p_{j-1}$ is a polynomial of degree $j-1$ with coefficients determined by commutator relations in $\mathfrak{g}$.

Fact $(*)$ shows that $\mathcal{D}\bigl(Y^{k+1}g-p_{k}(Y)g\bigl)=0$ for all $\mathcal{D}\in\mathcal{E}_{U}(\mathcal{H})$. Then it follows from  Theorem \ref{th:2} and Remark \ref{re:1} that the equation
\begin{align*}
 \mathfrak{u}f_{k+1}=Y^{k+1}g-p_{k}(Y)g
\end{align*}
has a solution $f_{k+1}\in \mathcal{H}$ with Sobolev estimates
\begin{align}\label{for:19}
\norm{f_{k+1}}_{G_\mathfrak{u},t}\leq C_{s,t,u_0}\norm{Y^{k+1}g-p_{k}(Y)g}_{s}\leq C_{s,t,u_0}\norm{g}_{s+k+1}
\end{align}
if $t<s-1$. On the other hand, for any $h\in\mathcal{H}^\infty$ we have
\begin{align}\label{for:106}
    -\langle f_{k+1},\mathfrak{u}h\rangle&= \langle \mathfrak{u}f_{k+1},\,h\rangle=\langle Y^{k+1}g-p_{k}(Y)g,\,h\rangle\notag\\
    &=\langle g,\,(-1)^{k+1}Y^{k+1}h-p'_{k}(Y)h\rangle\notag\\
    &=\langle \mathfrak{u}f,\,(-1)^{k+1}Y^{k+1}h-p'_{k}(Y)h\rangle\notag\\
    &=\langle f,\,(-1)^{k+1}\mathfrak{u}Y^{k+1}h-\mathfrak{u}p'_{k}(Y)h\rangle,
    \end{align}
where $p'_{k}$ is the adjoint polynomial of $p_k$. Keeping using relation \eqref{for:77} we see that there exists a polynomial $p''_{k}$ of degree
$k$ such that
\begin{align*}
(-1)^{k+1}\mathfrak{u}Y^{k+1}-\mathfrak{u}p'_{k}(Y)=(-1)^{k+1}Y^{k+1}\mathfrak{u}-p''_{k}(Y)\mathfrak{u}.
\end{align*}
Substituting the above relation into \eqref{for:106} we have
\begin{align*}
    -\langle f_{k+1},\mathfrak{u}h\rangle&=\langle f,\,(-1)^{k+1}Y^{k+1}\mathfrak{u}h-p''_{k}(Y)\mathfrak{u}h\rangle\\
    &=-\langle Y^{k}f,\,Y\mathfrak{u}h\rangle-\langle p'''_{k}(Y)f,\,\mathfrak{u}h\rangle,
\end{align*}
where $p'''_{k}$ is the adjoint polynomial of $p''_{k}$. This shows that
\begin{align*}
 Y^{k+1}f=-f_{k+1}+p'''_{k}(Y)f
\end{align*}
by Lemma \ref{le:10}. From \eqref{for:24} and \eqref{for:19} we have
\begin{align*}
\norm{Y^{k+1}f}_{G_\mathfrak{u},t}=\norm{f_{k+1}+p'''_{k}(Y)f}_{G_\mathfrak{u},t}\leq C_{s,t}\norm{g}_{s+k+1}
\end{align*}
if $t<s-1$. Then we proved the case when $m=k+1$ and thus finish the proof.
\end{proof}

\subsection{Coboundary for the unipotent flow in irreducible component of $G=SL(2,\RR)\ltimes\RR^2$}
 In this section we take notations in Section \ref{sec:2}. Let $G'$ denote the subgroup $\begin{pmatrix}[cc|c]
  a & 0 & v_1\\
  c & a^{-1} & v_2
\end{pmatrix}$, where $a\in \RR^+$ and $c,\,v_1,\,v_2\in\RR$. The Lie algebra of $G'$ is generated by $X$, $V$, $Y_1$ and $Y_2$. Our contention is:
\begin{theorem}\label{po:2}
 For any irreducible component $(\rho_{t},\,\mathcal{H}_{t})$ of $SL(2,\RR)\ltimes \RR^2$,
 \begin{enumerate}
 \item \label{for:61} if the cohomological equation $Vf=g$ has a solution $f\in \mathcal{H}_{t}^\infty$, then $f$ satisfies
 \begin{align*}
  \norm{f}_{s}\leq C_{s}\norm{g}_{s+6},\qquad \forall\,s\geq 0.
\end{align*}

\bigskip

 \item \label{for:62} suppose $t\neq 0$.  If $g,\,UY_2g\in (\mathcal{H}_t)_{G'}^\infty$ and $\int_{-\infty}^\infty g(x,y)dy=0$ for almost all $x\in\RR$, then the cohomological equation $Vf=g$ has a solution $f\in (\mathcal{H}_t)_{G'}^\infty$.

  \medskip
     \item \label{for:23} suppose $t\neq 0$. If $g\in\mathcal{H}_{t}^\infty$ and $\int_{-\infty}^\infty g(x,y)dy=0$ for almost all $x\in\RR$, then the cohomological equation $Vf=g$ has a solution $f\in \mathcal{H}_{t}$ such that $UY_2f\in(\mathcal{H}_t)_{G'}^\infty$.

\medskip

     \item \label{for:38} suppose $t\neq 0$. If $g\in\mathcal{H}_{t}^\infty$ and $\int_{-\infty}^\infty g(x,y)dy=0$ for almost all $x\in\RR$, then the cohomological equation $Vf=g$ has a solution $f\in \mathcal{H}_{t}$ such that $Uf\in(\mathcal{H}_t)_{G'}^\infty$.

\medskip
     \item \label{for:63} suppose $t\neq 0$. If $g\in\mathcal{H}_{t}^\infty$ and $\int_{-\infty}^\infty g(x,y)dy=0$ for almost all $x\in\RR$, then for any $n\in\NN$, the cohomological equation $Vf=g$ has a solution $f\in \mathcal{H}_{t}$ such that $U^jf,\,UY_2U^jf\in(\mathcal{H}_t)_{G'}^\infty$ for any $0\leq j\leq n$.
\medskip
   \item \label{for:55} suppose $t\neq 0$. If $g\in \mathcal{H}_{t}^\infty$ and $\int_{-\infty}^\infty g(x,y)dy=0$ for almost all $x\in\RR$, then the cohomological equation $Vf=g$ has a solution $f\in \mathcal{H}_{t}^\infty$ satisfying
\begin{align*}
  \norm{f}_{s}\leq C_{s}\norm{g}_{s+6},\qquad \forall\,s\geq 0.
\end{align*}
\medskip
\item \label{for:53} if $g\in \mathcal{H}_{t}^\infty$ and $\int_{-\infty}^\infty g(x,y)dy=0$, then the cohomological equation $Vf=Y_2g$ has a solution $f\in \mathcal{H}_{t}^\infty$ satisfying
\begin{align*}
  \norm{f}_s\leq C_{t}\norm{g}_{s+7},\qquad \forall\,s\geq 0.
\end{align*}

\medskip

\item \label{for:64}if $g\in \mathcal{H}_{t}^\infty$ and the cohomological equation $Vf=g$  has a solution $f\in \mathcal{H}_{t}$, then  $f\in \mathcal{H}_{t}^\infty$ and satisfies
\begin{align*}
  \norm{f}_{s}\leq C_{s}\norm{g}_{s+6},\qquad \forall\,s\geq 0.
\end{align*}
 \end{enumerate}
\end{theorem}
 The subsequent discussion will be devoted to the proof of
this theorem.
\begin{definition}\label{de:1}
For any function $f(x,y)$ on $\RR^2$ and any $x\in\RR$, we associate a function $f_x$ defined on $\RR$ by $f_x(y)=f(x,y)$. Then for any function $f(x_1,\cdots,x_n)$ on
$\RR^n$ and $(x_{k_1},\cdots,x_{k_m})\in\RR^m$, $f_{x_{k_1},\cdots,x_{k_m}}$ is an obviously defined function on $\RR^{n-m}$.
\end{definition}
The following lemma gives the necessary condition under which there exists a solution to the cohomological equation $Vf=g$ in each irreducible component $(\rho_{t},\,\mathcal{H}_t)$:
\begin{lemma}\label{le:9}
 Suppose $g\in\mathcal{H}_t$ and $Y_1g\in \mathcal{H}_t$. Then
 \begin{enumerate}
   \item \label{for:108}$\int_{-\infty}^\infty |g(x,y)|dy<\infty$ for almost all $x\in\RR$.

   \medskip
   \item \label{for:109}if the cohomological equation $Vf=g$ has a solution $f\in \mathcal{H}_t$, then $\int_{-\infty}^\infty g(x,y)dy=0$ for almost all $x\in\RR$.

   \medskip
   \item \label{for:110} if $Y_1^2g\in \mathcal{H}_t$ and $\int_{-\infty}^\infty g(x,y)dy=0$, then $f(x,y)=\int_{0}^\infty g(x,t+y)dt$ is an element in $\mathcal{H}_t$ with the estimate
       \begin{align*}
        \norm{f}\leq 2(\norm{g}+\norm{Y_1g}+\norm{Y_1^2g}).
       \end{align*}

 \end{enumerate}
 \end{lemma}
\begin{proof}
\textbf{\emph{Proof of \eqref{for:108}}} For any $h(x,y)\in L^2(\RR^2)$ denote by $\Omega_h\subset\RR$ a full measure set such that $h_x\in L^2(\RR)$ for any $x\in \Omega_h$. For any $x\in \Omega_g\bigcap \Omega_{Y_1g}$ we have
\begin{align}\label{for:59}
 &\int_{\RR}|g(x,y)|dy\notag\\
 &=\int_{\abs{y}\leq 1}|g(x,y)|dy+\int_{\abs{y}>1}|g(x,y)|dy\notag\\
 &\overset{\text{(1)}}{\leq} \bigl(\int_{\abs{y}\leq 1}|g(x,y)|^2dy\bigl)^{\frac{1}{2}}+\int_{\abs{y}>1}|g(x,y)y\cdot y^{-1}|dy\notag\\
 &\overset{\text{(2)}}{\leq} \bigl(\int_{\RR}|g(x,y)|^2dy\bigl)^{\frac{1}{2}}+\bigl(\int_{\abs{y}> 1}\bigl|Y_1g(x,y)\bigl|^2dy\bigl)^{\frac{1}{2}}\cdot \bigl(\int_{\abs{y}> 1}y^{-2}dy\bigl)^{\frac{1}{2}}\notag\\
 &\leq \bigl(\int_{\RR}|g(x,y)|^2dy\bigl)^{\frac{1}{2}}+\bigl(\int_{\abs{y}> 1}\bigl|Y_1g(x,y)\bigl|^2dy\bigl)^{\frac{1}{2}}\notag\\
 &\leq \norm{g_x}_{L^2(\RR)}+\norm{(Y_1g)_x}_{L^2(\RR)},
\end{align}
where $\norm{\cdot}_{L^2(\RR)}$ is the standard $L^2$ norm for functions over $\RR$. Note that $(1)$ and $(2)$ follow from Cauchy-Schwarz inequality. The above calculations show that $g_x\in L^1(\RR)$ for any $x\in \Omega_g\bigcap \Omega_{Y_1g}$, which proves the claim.

\medskip
\textbf{\emph{Proof of \eqref{for:109}}} By relations in \eqref{for:107} of Lemma \ref{le:1}, the equation $Vf=g$ has the expression
\begin{align*}
 -x\partial_yf=g.
\end{align*}
Taking fourier transformation on factor $y$, we have
\begin{align}\label{for:2}
-\hat{f}_\xi(x,\xi)\cdot x\xi\sqrt{-1}=\hat{g}_\xi(x,\xi)
\end{align}
where
\begin{align*}
\hat{k}_\xi(x,\xi)&=\frac{1}{\sqrt{2\pi}}\int k(x,y) e^{-y\xi\sqrt{-1}}dy,\qquad \text{ where }k=f\text{ or }g.
\end{align*}
The  earlier result shows that for any $x\in \Omega_g\bigcap \Omega_{Y_1g}$, $(\hat{g}_\xi)_x$ are continuous functions and $\hat{g}_\xi(x,0)=\int g(x,y)dy$. From \eqref{for:2} we see that $(\hat{g}_\xi)_x\cdot\xi^{-1}\in L^2(\RR)$ for any $x\in \Omega_{\hat{f}}\backslash 0$, which implies that
\begin{align*}
\int g(x,y)dy=\hat{g}_\xi(x,0)=0
\end{align*}
for any $x\in (\Omega_{\hat{g}}\bigcap \Omega_{\hat{f}})\backslash 0$.

\medskip
\textbf{\emph{Proof of \eqref{for:110}}} The earlier discussion shows that $f$ is measurable. Note that
 \begin{align*}
|f(x,y)|^2\leq2\bigl(\int_{0}^1 |g(x,t+y)|dt\bigl)^2+2\bigl(\int_{1}^\infty |g(x,t+y)|dt\bigl)^2.
 \end{align*}
When $y\geq 0$, we will get estimates for the above two terms respectively. We have
 \begin{align}\label{for:82}
&\int_{y\geq 0}\bigl(\int_{0}^1 |g(x,t+y)|dt\bigl)^2dxdy\notag\\
&\overset{\text{(1)}}{\leq}\int_{0\leq y< 1}\bigl(\int_{y}^{y+1} |g(x,t)|^2dt\bigl)dxdy+\int_{y\geq 1}\bigl(\int_{y}^{y+1} |g(x,t)|^2dt\bigl)dxdy\notag\\
&\leq\int_{\RR}\int_{0}^{2} |g(x,t)|^2dtdx+\int_{y>1}\bigl(\int_{y}^{y+1} |g(x,t)t|^2\cdot t^{-2}dt\bigl)dxdy\notag\\
&\leq\norm{g}^2+\int_{y\geq 1}\bigl(\int_{y}^{y+1} |(Y_1g)(x,t)|^2dtdx\bigl)\cdot y^{-2}dy\notag\\
&\leq\norm{g}^2+\norm{Y_1g}^2\cdot\int_{y\geq 1} y^{-2}dy\notag\\
&=\norm{g}^2+\norm{Y_1g}^2,
 \end{align}
 and
  \begin{align}\label{for:83}
    &\int_{y\geq 0}\bigl(\int_{1}^\infty |g(x,t+y)|dt\bigl)^2dxdy\notag\\
    &=\int_{y\geq 0}\bigl(\int_{1+y}^\infty |g(x,t)|dt\bigl)^2dxdy\notag\\
    &\leq\int_{y\geq 0}\bigl(\int_{1+y}^\infty |g(x,t)t^2\cdot t^{-2}|dt\bigl)^2dxdy\notag\\
    &\overset{\text{(2)}}{\leq}\int_{y\geq 0}\bigl(\int_{1+y}^\infty \bigl|Y_1^2g(x,t)\bigl|^2dt\cdot \int_{1+y}^\infty \bigl|t^2\bigl|^{-2}dt\bigl)dxdy\notag\\
    &\leq\int_{\RR^2}\bigl|Y_1^2g(x,t)\bigl|^2dtdx\cdot \int_{y\geq 0}\int_{1+y}^\infty \bigl|t|^{-4}dtdy\notag\\
    &\leq\frac{1}{6}\norm{Y_1^2g}^2.
    \end{align}
Note that $(1)$ and $(2)$ follow from Cauchy-Schwarz inequality.    This shows that
  \begin{align*}
 \int_{y\geq 0}| f(x,y)|^2dxdy\leq 2(\norm{g}^2+\norm{Y_1g}^2+\frac{1}{6}\norm{Y_1^2g}^2).
  \end{align*}
Since $f(x,y)=\int_{-\infty}^0 g(x,t+y)dt$ by assumption, we also have
  \begin{align*}
|f(x,y)|^2\leq2\bigl(\int_{-1}^0 |g(x,t+y)|dt\bigl)^2+2\bigl(\int_{-\infty}^{-1} |g(x,t+y)|dt\bigl)^2.
 \end{align*}
In exactly the same manner as before we find that
\begin{align*}
    \int_{y<0}| f(x,y)|^2dxdy\leq 2(\norm{g}^2+\norm{Y_1g}^2+\frac{1}{6}\norm{Y_1^2g}^2).
 \end{align*}
Hence we have
\begin{align*}
\bigl(\int| f(x,y)|^2dxdy\bigl)^{\frac{1}{2}}\leq 2(\norm{g}+\norm{Y_1g}+\frac{1}{6}\norm{Y_1^2g}),
\end{align*}
which proves the claim.
\end{proof}

 The crucial step in proving Theorem \ref{po:2} is:
 \begin{lemma}\label{le:8}
 For any irreducible component $(\rho_{t},\,\mathcal{H}_{t})$ of $SL(2,\RR)\ltimes \RR^2$, if $g\in(\mathcal{H}_t)_{G'}^\infty$ and $\int_{-\infty}^\infty g(x,y)dy=0$ for almost all $x\in\RR$, then
 \begin{enumerate}
   \item  the cohomological equation $Vf=Y_2g$ has a solution $f\in \mathcal{H}_t$ with the estimate
   $\norm{f}\leq 2\norm{g}_2$.

   \medskip
   \item further, $f$ is in $(\mathcal{H}_t)^\infty_{G'}$ and satisfies the Sobolev estimates
 \begin{align}\label{for:34}
  \norm{f}_{G',r}\leq C_{r}\norm{g}_{G',r+3},\qquad \forall\,r>0.
\end{align}
 \end{enumerate}
 \end{lemma}
 \begin{proof}
 The proof is divided in two parts: in the first part, we construct explicitly a solution in $\mathcal{H}$ and then give Sobolev estimates of the solution in the second part.

\emph{\textbf{Part I: Construction of the solution}} Recall notations in Section \ref{sec:17}. Let $\Omega=\{(x,y)\in\RR^2:x\neq 0\}$, $\Omega_{a,b}=\{(x,y)\in\RR^2:\abs{x}\geq a,\,y\geq b\}$ and $\Omega^y_{a}=\{y\in\RR:y\geq a\}$. Using relations in \eqref{for:107} of  Lemma \ref{le:1} we have
$x^2\partial_x=Y_2X\sqrt{-1}-Y_1V\sqrt{-1}$ and then
\begin{align}\label{for:21}
    &\partial_x\partial_x\circ Y_2^4+\partial_y\partial_y\circ Y_2^4\notag\\
    &=-(Y_2X-Y_1V)^2-12Y_2^2+6Y_2(Y_2X-Y_1V)-V^2Y_2^2.
\end{align}
For the vector fields $\partial_x$ and $\partial_y$ $\partial_x\partial_x+\partial_y\partial_y$ is the Laplace operator. By elliptic regularity theorem we find that $Y_2^4g\in W^2(\RR^2)$ with the Sobolev estimate
\begin{align*}
    \norm{Y_2^4g}_{W^2(\RR^2)}\leq C\norm{\partial_x\partial_x(Y_2^4g)+\partial_y\partial_y(Y_2^4g)}+C\norm{Y_2^4g}\leq C\norm{g}_4.
\end{align*}
Further, Sobolev imbedding theorem implies that
\begin{align}\label{for:27}
    \norm{g\cdot x^4}_{C^0}=\norm{Y_2^4g}_{C^0}\leq C\norm{Y_2^4g}_{W^2(\RR^2)}\leq C\norm{g}_{G',4}.
\end{align}
 Note that $Y_1=-y\sqrt{-1}$ and $V=-x\partial_y$.  In \eqref{for:27}, by substituting $g$ with $Y_1^2g$ and $Vg$ respectively,  we get
\begin{align}\label{for:28}
\norm{g\cdot y^2x^4}_{C^0}&=\norm{Y_1^2g\cdot x^4}_{C^0}\leq C\norm{Y_1^2g}_{4}\leq C\norm{g}_{G',6},\qquad\text{ and }\notag\\
\norm{\partial_yg\cdot x^5}_{C^0}&=\norm{Vg\cdot x^4}_{C^0}\leq C\norm{Vg}_4\leq C\norm{g}_{G',5}.
\end{align}
Let  $f(x,y)=-\int_{0}^\infty g(x,t+y)dt$. Lemma \ref{le:9} shows that $f\in \mathcal{H}_t$ with the estimate $\norm{f}\leq 2\norm{g}_2$.
For any $x\neq 0$ we have
\begin{align*}
\partial_yf(x,y)&=-\frac{\partial}{\partial y}\int_{0}^\infty g(x,t+y)dt=-\int_{0}^\infty \frac{\partial}{\partial y}g(x,t+y)dt\\
&=-\int_{0}^\infty \frac{\partial}{\partial t}g(x,t+y)dt=g(x,y).
\end{align*}
Of course, to justify differentiation under the integral sign, we must prove that $\int_{0}^\infty \frac{\partial}{\partial t}g(x,t+y)dt$ is a uniformly convergent integral. From above, however, we note that
\begin{align*}
\bigl|\int_{r}^\infty \frac{\partial}{\partial t}g(x,t+y)dt\bigl|=|g(x,r+y)|.
\end{align*}
 So by \eqref{for:28} for any $a>0$ and $b\in\RR$, we can always make $|g_x(r+y)|$ uniformly small on the set $\Omega_{a,b}$  by choosing $r$ large enough. Therefore $\partial_yf(x,y)=g(x,y)$ on $\Omega\times\RR$. Recall relations in \eqref{for:107}. Hence we showed that
\begin{align}\label{for:43}
V(f\sqrt{-1})=Y_2g.
\end{align}
Hence  we proved the first claim.

\medskip
\noindent \emph{\textbf{Part II: Sobolev estimates of the solution on $G'$}}.

\noindent\textbf{Sobolev estimates along $Y_2$}. Note that for any $n\in\NN$,
\begin{align*}
\int_{-\infty}^\infty (Y_2^ng)(x,y)dy=0
\end{align*}
for almost all $x\in\RR$ and
\begin{align*}
 (Y_2^nf)(x,y)=-\int_{0}^\infty (Y_2^ng)(x,t+y)dt.
\end{align*}
Then it follows from \eqref{for:110} of Lemma \ref{le:9} that
\begin{align}\label{for:58}
    \norm{Y_2^nf}\leq C\norm{Y_2^ng}_2\leq C\norm{g}_{G',n+2},\qquad \forall\,n\in\NN.
 \end{align}

 \medskip
\noindent\textbf{Sobolev estimates along $V$}. Using \eqref{for:43} we see that
 \begin{align}\label{for:46}
    \norm{V^nf}=\norm{V^{n-1}Y_2g}\leq \norm{g}_{G',n}\qquad \forall\,n\in\NN.
 \end{align}

 \medskip
\noindent\textbf{Sobolev estimates along $Y_1$}. By relations in \eqref{for:107}, we see that
\begin{align*}
(Y_1^nf)(x,y)=-\int_{0}^\infty g(x,t+y)y^ndt.
\end{align*}
From the relation
 \begin{align*}
  &|(Y_1^nf)(x,y)|^2\leq2\bigl(\int_{0}^1 |g(x,t+y)y^n|dt\bigl)^2+2\bigl(\int_{1}^\infty |g(x,t+y)y^n|dt\bigl)^2,
\end{align*}
it suffices to get estimates for the above two terms respectively. For any $n\in\NN$ we have
\begin{align*}
&\int_{y\geq 0}\bigl(\int_{0}^1 \bigl|g(x,t+y)y^n\bigl|dt\bigl)^2dxdy\\
&\leq\int_{y\geq 0}\bigl(\int_{0}^1 \bigl|g(x,t+y)(t+y)^n\bigl|dt\bigl)^2dxdy\\
&=\int_{y\geq 0}\bigl(\int_{0}^1 \bigl|(Y_1^ng)(x,t+y)\bigl|dt\bigl)^2dxdy
\end{align*}
and
\begin{align*}
    &\int_{y\geq 0}\bigl(\int_{1}^\infty \bigl|g(x,t+y)y^n\bigl|dt\bigl)^2dxdy\\
    &\leq\int_{y\geq 0}\bigl(\int_{1}^\infty \bigl|g(x,t+y)(t+y)^n\bigl|dt\bigl)^2dxdy\\
    &=\int_{y\geq 0}\bigl(\int_{1}^\infty \bigl|(Y_1^ng)(x,t+y)\bigl|dt\bigl)^2dxdy.
    \end{align*}
In \eqref{for:82} and \eqref{for:83} substituting $g$ with $Y_1^ng$ we get
\begin{align}\label{for:29}
 \int_{y\geq 0}\bigl(\int_{0}^1 \bigl|g(x,t+y)y^n\bigl|dt\bigl)^2dxdy&\leq \norm{Y_1^ng}^2+\norm{Y_1^{n+1}g}^2\qquad\text{ and}\notag\\
 \int_{y\geq 0}\bigl(\int_{1}^\infty \bigl|g(x,t+y)y^n\bigl|dt\bigl)^2dxdy&\leq \frac{1}{6}\norm{Y_1^{n+2}g}^2,
\end{align}
which gives
\begin{align*}
\int_{y\geq 0}| Y_1^nf(x,y)|^2dxdy\leq\norm{Y_1^ng}^2+\norm{Y_1^{n+1}g}^2+ \frac{1}{6}\norm{Y_1^{n+2}g}.
\end{align*}
By assumption, we also have $(Y_1^nf)(x,y)=\int_{-\infty}^0 g(x,t+y)y^ndt$ and the relation
\begin{align*}
  &|(Y_1^nf)(x,y)|^2\leq2\bigl(\int_{-1}^{0} |g(x,t+y)y^n|dt\bigl)^2+2\bigl(\int_{-\infty}^{-1} |g(x,t+y)y^n|dt\bigl)^2.
\end{align*}
In exactly the same manner as before we find that
\begin{align}\label{for:71}
\int_{y< 0}| Y_1^nf(x,y)|^2dxdy\leq\norm{Y_1^ng}^2+\norm{Y_1^{n+1}g}^2+ \frac{1}{6}\norm{Y_1^{n+2}g}.
\end{align}
For any $n\in\NN$ the above discussion gives
\begin{align}\label{for:31}
    \norm{Y_1^nf}\leq 2(\norm{Y_1^ng}+\norm{Y_1^{n+1}g}+\norm{Y_1^{n+2}g})\leq 2\norm{g}_{G',n+2}.
 \end{align}

 \medskip
\noindent\textbf{Sobolev estimates along $X$}. Using relations in \eqref{for:107} we get
\begin{align*}
x^2y^3\partial_x&=Y_1^3Y_2X-Y_1^4V.
\end{align*}
In \eqref{for:27}, by substituting $g$ with $Y_1^3Y_2Xg-Y_1^4Vg$,  we get
\begin{align}\label{for:48}
\norm{\partial_xg\cdot x^6y^3}_{C^0}&=\norm{(Y_1^3Y_2Xg-Y_1^4Vg)x^4}_{C^0}\notag\\
&\leq C\norm{Y_1^3Y_2Xg-Y_1^4Vg}_{G',4}\notag\\
&\leq C\norm{g}_{G',9}.
\end{align}
Since $g\in \mathcal{H}^\infty$, \eqref{for:108} of Lemma \ref{le:9}
\eqref{for:48} shows that for any $x\neq 0$ shows that $\int (Xg)(x,y)dy<\infty$ for almost all $x\in\RR$. Further, for almost all $x\in\RR$, formally we have
\begin{align}\label{for:30}
    \int (Xg)(x,y)dy&=\int \bigl(-\partial_xg\cdot x+\partial_yg\cdot y\bigl)dy\notag\\
    &\overset{\text{(1)}}{=}-\int \partial_xg\cdot xdy+\int \partial_yg\cdot ydy\notag\\
    &\overset{\text{(2)}}{=}-\int \partial_xg\cdot xdy-\int gdy\notag\\
    &=-\int \partial_xg\cdot xdy\overset{\text{(3)}}{=}-\bigl(\partial_x\int gdy\bigl)x=0.
 \end{align}
From \eqref{for:48}, we find that for any $x\in\RR\backslash 0$ $(\partial_xg)_x$ and $(\partial_xg\cdot y)_x$ are both in $L^2(\RR)$. In \eqref{for:59}, substituting $g$ with $\partial_xg$, we see that $\int (\partial_xg\cdot x)dy<\infty$ for any $x\in\RR\backslash 0$, which gives $(1)$. \eqref{for:28} shows that for any $x\neq 0$, $g\cdot y\rightarrow 0$ as $y\rightarrow 0$, which justifies $(2)$. Finally, to justify differentiation under the integral sign, we must prove that for any $x\neq 0$, $\int_{-\infty}^\infty \frac{\partial}{\partial x}g(x,y)dy$ is a uniformly convergent integral in a small neighborhood of $x$. From \eqref{for:48} we  see that $\partial_xg\cdot y^3$ is uniformly bounded on the set $\Omega_{a}=\{(x,y)\in\RR^2,\,\abs{x}> a\}$, $a>0$.  For any $r_1,\,r_2\geq r$ we have
 \begin{align*}
&\Bigl|\int_{r_1}^\infty \frac{\partial}{\partial x}g(x,y)dy\Bigl|+\Bigl|\int_{-\infty}^{-r_2} \frac{\partial}{\partial x}g(x,y)dy\Bigl|\\
&\leq \Bigl|\int_{r_1}^\infty \frac{\partial}{\partial x}g(x,y)y^3\cdot y^{-3}dy\Bigl|+\Bigl|\int_{-\infty}^{-r_2}\frac{\partial}{\partial x}g(x,y)y^3\cdot y^{-3}dy\Bigl|\\
&\leq 2\bigl\|\frac{\partial}{\partial x}g(x,y)y^3\bigl\|_{(C^0,\Omega_a)}\bigl(\int_{r}^\infty y^{-3}dy\bigl)\\
&=\frac{2}{3r^3a^6}\bigl\|\frac{\partial}{\partial x}g(x,y)y^2\bigl\|_{C^0},
\end{align*}
which gives $(3)$. Then we just showed that:

\noindent $(*)$ if $\int gdy=0$ for almost all $x\in\RR$ then $\int (Xg)dy=0$ for almost all $x\in\RR$.

Then inductively, we see that for any $n\in\NN$, $\int (X^ng)dy=0$ for almost all $x\in\RR$. Next, we will prove by induction that for any $n\in\NN$, $X^nf\in \mathcal{H}_t$ and is a solution of the equation
\begin{align}\label{for:32}
    V(X^nf)=Y_2q_{n}(X)g,
\end{align}
where $q_{n}$ is a polynomial of degree $n$.

$(*)$ and the conclusion in earlier part  show that the equation
\begin{align*}
Vh=Y_2(Xg+g)
\end{align*}
has a solution $h\in \mathcal{H}_t$ with estimates
\begin{align*}
    \norm{h}\leq 2\norm{Xg+g}_2\leq C\norm{g}_3.
 \end{align*}
Using the commutator relations
\begin{align*}
XY_2-Y_2X=-Y_2,\quad XV-VX=-2V, \quad\text{and}\quad VY_2=Y_2V,
\end{align*}
for any $\omega\in \mathcal{H}_t^\infty$ we have
\begin{align*}
    \langle h,V\omega\rangle&=-\langle Vh,\,\omega\rangle=-\langle Y_2(Xg+g),\,\omega\rangle=-\langle (XY_2+2Y_2)g,\,\omega\rangle\\
    &=-\langle Y_2g,\,(2-X)\omega\rangle=-\langle Vf,\,(2-X)\omega\rangle=\langle f,\,V(2-X)\omega\rangle\\
    &=-\langle f,\,XV\omega\rangle
    \end{align*}
Proposition \ref{cor:1} shows that there is no non-trivial $SL(2,\RR)$-invariant vectors in $\mathcal{H}_t$. Applying Lemma \ref{le:10} to the above relation we get
    $Xf=h$. Thus the estimate of $h$ gives $\norm{Xf}\leq C\norm{g}_{G',3}$. Then we proved the case when $n=1$. Suppose \eqref{for:32} holds when $n\leq k$.
Earlier arguments show that
\begin{align*}
\int \bigl(q_{k}(X)(Xg)+g\bigl)dy=0\qquad\text{ for almost all }x\in\RR.
\end{align*}
Along the proof line of the case of $n=1$, we can show that $X^{k+1}f$ is in $\mathcal{H}_t$ and satisfies the equation
\begin{align*}
    V(X^{k+1}f)=Y_2\bigl(q_{k}(X)(Xg)+g\bigl)=Y_2q_{k+1}(X)g
\end{align*}
with the estimate
\begin{align*}
    \norm{X^{k+1}f}\leq C\norm{q_{k}(X)(Xg)+g}_{2}\leq C_k\norm{g}_{G',k+3}.
 \end{align*}
Hence we proved the case when $n\leq k+1$ and thus obtained
\begin{align}\label{for:111}
    \norm{X^nf}\leq C_n\norm{g}_{G',n+2},\qquad \forall \,n\in\NN.
 \end{align}
 Then \eqref{for:34} follows directly from estimates in \eqref{for:58}, \eqref{for:46}, \eqref{for:31} and \eqref{for:111} and Theorem \ref{th:4}.
\end{proof}
\begin{remark}
 Without the condition ``$\int gdy=0$ for almost all $x\in\RR$" we can just as well define $f(x,y)=-\int_{0}^\infty g(x,t+y)dt$ and show that
\begin{align*}
 V(f\sqrt{-1})=Y_2g,\qquad\text{ on }\Omega\times\RR.
\end{align*}
Being able to write $f(x,y)=\int_{-\infty}^0 g(x,t+y)dt$, is what allows us to explore the $L^2(\RR^2)$ and $(\mathcal{H}_t)_{G'}^k$ properties of $f$.

We obtained Sobolev estimates long $V$, $Y_1$ and $Y_2$ by explicit calculation. To get smoothness of along $X$, we make use of the ``centralizer trick" that first appeared in Proposition \ref{le:6}, which will play the key role in next part to get smoothness along $U$.
\end{remark}
We are now in a position to proceed with the proof of Proposition \ref{po:2}.
\subsection{Proof of Theorem \ref{po:2}} \emph{\textbf{Proof of \eqref{for:61}}}. If $f$ is smooth, then immediately we see that $\mathcal{D}(g)=0$ for all $\mathcal{D}\in\mathcal{E}_{S}(\mathcal{H})$, where $S=\{\exp(tV)\}_{t\in\RR}$. By using Proposition \ref{le:6} we get the estimates
\begin{align}\label{for:20}
    \norm{Y_2^mf}_{SL(2,\RR),r}<C_{t,s}\norm{g}_s,\qquad \forall \,r<s-m-1.
\end{align}
Note that the constants $C_{t,s}$ are independent of the parameter $t$ since all $\rho_t\mid_{SL(2,\RR)}$ are outside a fixed neighborhood the trivial representation in the sense of Fell topology  by Remark \ref{re:2}.

\eqref{for:109} of Lemma \ref{le:9} shows that $\int gdy=0$ for almost all $x\in\RR$. Moreover, by arguments in the first part of the proof of Lemma \ref{le:8}, we find that
\begin{align*}
f(x,y)\cdot xy^n=&\left\{\begin{aligned} &\int_{0}^\infty g(x,t+y)\cdot y^ndt,&\qquad & y\geq 0\\
&-\int_{\infty}^0 g(x,t+y)\cdot y^ndt,&\qquad &y<0.
 \end{aligned}
 \right.
\end{align*}
for any $n\in\NN$. From \eqref{for:29} and \eqref{for:71} we see that
\begin{align*}
\norm{Y_1^nY_2f}=\norm{f(x,y)\cdot xy^n}\leq C\norm{g}_{n+2},\qquad \forall\,n\in\NN.
\end{align*}
From above relation and \eqref{for:20}, by using Theorem \ref{th:4} we see that  $Y_2f$ satisfies the estimates
\begin{align}\label{for:65}
    \norm{Y_2f}_s\leq C_{s}\norm{g}_{s+3},\qquad \forall\,s>0.
\end{align}
Using the commutator relation
\begin{align*}
VY_1^m=Y_1^mV+nY_1^{m-1}Y_2
\end{align*}
we have
\begin{align*}
   VY_1^mf= (Y_1^mV+mY_1^{m-1}Y_2)f=Y_1^mg+mY_1^{m-1}Y_2f,\qquad \forall\,m\in\NN,
\end{align*}
which implies that
\begin{align}\label{for:60}
\norm{Y_1^mf}&\overset{\text{(1)}}{\leq} C\norm{Y_1^mg+mY_1^{m-1}Y_2f}_2\notag\\
&\leq C\norm{g}_{m+2}+C_m\norm{f}_{m+2}\notag\\
&\overset{\text{(2)}}{=} C_m\norm{g}_{m+5}.
\end{align}
$(1)$ follows from Lemma \ref{le:9} and \ref{le:8} and $(2)$ holds because of \eqref{for:65}.

As an immediate consequence of \eqref{for:20}, \eqref{for:60} and Theorem \ref{th:4} we get
\begin{align*}
    \norm{f}_{s}\leq C_s\norm{g}_{s+6}\qquad \forall\,s\geq 0,
\end{align*}
which proves \eqref{for:61}.

\medskip
\emph{\textbf{Proof of \eqref{for:62}}}. By Lemma \ref{le:8}, there exists $f\in(\mathcal{H}_{t})_{G'}^\infty$ such that $Vf=Y_2g$.
For any $\omega\in \mathcal{H}_{t}^\infty$, by using the commutator relations $UV=X+VU$ and $VY_2=Y_2V$ we have
\begin{align}\label{for:41}
    &\langle Y_2^2(UY_2g-Xf),\omega\rangle\notag\\
    &=-\langle Y_2g,\,UY_2^2\omega\rangle+\langle f,\,XY_2^2\omega\rangle\notag\\
    &=-\langle Vf,\,UY_2^2\omega\rangle+\langle f,\,XY_2^2\omega\rangle\notag\\
    &=\langle f,\,(VU+X)Y_2^2\omega\rangle\notag\\
    &=\langle f,\,UVY_2^2\omega\rangle=\langle f,\,UY_2^2V\omega\rangle.
\end{align}
By relations in \eqref{for:107} we get
\begin{align}\label{for:45}
    UY_2^2&=-t\sqrt{-1}+VY_1^2-XY_2Y_1\qquad \text{and}\notag\\
    UY_2^2V&=-tV\sqrt{-1}+VY_1^2V-XY_2Y_1V.
\end{align}
Since $f\in(\mathcal{H}_{t})_{G'}^\infty$, using above expressions we have
\begin{align}\label{for:42}
    \langle f,\,UY_2^2V\omega\rangle=\langle V(t\sqrt{-1}+Y_1^2V-Y_1Y_2X)f,\,\omega\rangle.
\end{align}
 \eqref{for:41} and \eqref{for:42} imply that
 \begin{align}\label{for:44}
  V(t\sqrt{-1}+Y_1^2V-Y_1Y_2X)f=Y_2^2(UY_2g-Xf),
 \end{align}
which gives the relation
\begin{align*}
\int\bigl(UY_2g(x,y)-Xf(x,y)\bigl)dy=0\qquad\text{ for almost all }x\in\RR
\end{align*}
by using \eqref{for:109} Lemma \ref{le:9}. Then it follows from Lemma \ref{le:8} that the equation
\begin{align}\label{for:37}
Vh=Y_2(UY_2g-Xf).
\end{align}
has a solution $h\in (\mathcal{H}_t)^\infty_{G'}$. Comparing  \eqref{for:44} and \eqref{for:37} we find that
\begin{align*}
    f\cdot x^{-1}&=-t^{-1}(h+Y_1Xf-Y_1^2g).
    \end{align*}
Then we see that $f\cdot(x\sqrt{-1})^{-1}\in (\mathcal{H}_t)^\infty_{G'}$ and satisfies the equation
\begin{align*}
 V\bigl(f\cdot(x\sqrt{-1})^{-1}\bigl)=g,
\end{align*}
which  proves \eqref{for:62}.
\medskip

\emph{\textbf{Proof of \eqref{for:23}}}. From \eqref{for:62} we see that the cohomological equation $Vf=g$ has a solution $f\in(\mathcal{H}_t)^\infty_{G'}$. For any $\omega\in \mathcal{H}_t^\infty$, by using the commutator relations $VU=UV-X$ and $VY_2=Y_2V$  we have
\begin{align*}
    &\bigl\langle Y_2^2\bigl(Ug-Xf\bigl),\omega\bigl\rangle\\
    &=-\bigl\langle g,\,UY_2^2\omega\bigl\rangle+\bigl\langle f,\,XY_2^2\omega\bigl\rangle\\
    &=\langle f,\, (VU+X)Y_2^2\omega\rangle=\langle f,\,UY_2^2V\omega\rangle\notag\\
    &\overset{\text{(1)}}{=}\langle V(t\sqrt{-1}+Y_1^2V-Y_1Y_2X)f,\,\omega\rangle,
\end{align*}
where $(1)$ follows from \eqref{for:45}, which gives the equation
\begin{align*}
V(t\sqrt{-1}+Y_1^2V-Y_1Y_2X)f=Y_2^2\bigl(Ug-Xf\bigl).
\end{align*}
Then it follows from \eqref{for:109} Lemma \ref{le:9} that
\begin{align}\label{for:52}
    \int\bigl(Ug(x,y)-Xf(x,y)\bigl)dy=0,\qquad\text{ for almost all }x\in\RR,
\end{align}
which implies that the following equation
\begin{align*}
    Vh=Y_2\bigl(Ug-Xf\bigl)
\end{align*}
has a solution $h\in(\mathcal{H}_t)_{G'}^\infty$ by using Lemma \ref{le:8}.

For any $\omega\in \mathcal{H}_{t}^\infty$, by using commutator relations
\begin{align*}
 VU=UV-X,\quad UY_2-Y_2U=Y_1,\quad \text{ and }\quad VY_2=Y_2V
\end{align*}
we have
\begin{align*}
    -\langle h,V\omega\rangle&=\langle Y_2\bigl(Ug-Xf\bigl),\omega\rangle=\langle g,\,UY_2\omega\rangle-\langle f,\,XY_2\omega\rangle\\
    &=-\langle f,\,(VU+X)Y_2\omega\rangle=-\langle f,UY_2V\omega\rangle\\
    &=-\langle f,(Y_2U+Y_1)V\omega\rangle=\langle Y_2f,UV\omega\rangle+\langle Y_1f,V\omega\rangle.
    \end{align*}
Hence we have
\begin{align*}
    -\langle h+Y_1f,V\omega\rangle=\langle Y_2f,UV\omega\rangle.
    \end{align*}
This shows that
\begin{align*}
UY_2f=h+Y_1f
\end{align*}
by Remark \ref{re:2} and Lemma \ref{le:10}, which means that $UY_2f\in(\mathcal{H}_t)^\infty_{G'}$.

\medskip

\emph{\textbf{Proof of \eqref{for:38}}}. We take notations in \eqref{for:23}. Since $Xf$ satisfies the equation
\begin{align}\label{for:47}
    V(Xf)=Xg+2g,
\end{align}
\eqref{for:109} of Lemma \ref{le:9} and \eqref{for:23} show that  $UY_2Xf\in(\mathcal{H}_t)^\infty_{G'}$, which allows us to see that
the equation
\begin{align*}
    Vh_1=Ug-Xf
\end{align*}
has solution $h_1\in(\mathcal{H}_t)^\infty_{G'}$ by using \eqref{for:52}
and \eqref{for:62}.

Then for any $\omega\in \mathcal{H}_{t}^\infty$, by using the commutator relation $VU=UV-X$ we have
\begin{align*}
    -\langle h_1,V\omega\rangle&=\langle Ug-Xf,\,\omega\rangle=-\langle g,U\omega\rangle+\langle f,X\omega\rangle\\
    &=\langle f,\,(VU+X)\omega\rangle=\langle f,UV\omega\rangle.
    \end{align*}
This shows that $Uf=h_1$ by Lemma \ref{le:10} and Remark \ref{re:2}. Hence $Uf\in(\mathcal{H}_t)^\infty_{G'}$.

\medskip
\emph{\textbf{Proof of \eqref{for:63}}}. We will prove \eqref{for:63} inductively. Note that $Xf$ is the solution of \eqref{for:47},
\eqref{for:109} of Lemma \ref{le:9} and \eqref{for:38} show that $UXf\in(\mathcal{H}_t)^\infty_{G'}$.

For any $\omega\in \mathcal{H}^\infty$, by using the commutator relation
\begin{align*}
 VU^2=U^2V-2UX-2U\quad\text{ and }\quad VY_2=Y_2V
 \end{align*}
and \eqref{for:45} we have
\begin{align*}
    &\bigl\langle Y_2^2\bigl(U^{2}g-2UXf-2Uf\bigl),\,\omega\bigl\rangle\\
     &=\bigl\langle Y_2^2\bigl(U^{2}Vf-2UXf-2Uf\bigl),\,\omega\bigl\rangle\\
    &=-\langle f,U^{2}VY_2^2\omega\rangle=\langle Uf,\,UY_2^2V\omega\rangle\notag\\
    &=\langle V(t\sqrt{-1}+Y_1^2V-Y_1Y_2X)Uf,\,\omega\rangle.
\end{align*}
Then we get
\begin{align*}
   V(t\sqrt{-1}+Y_1^2V-Y_1Y_2X)Uf=Y_2^2\bigl(U^{2}g-2UXf-2Uf\bigl).
\end{align*}
Then by \eqref{for:109} of Lemma \ref{le:9} we have
\begin{align*}
    \int\bigl(U^{2}g(x,y)-2UXf(x,y)-2Uf(x,y)\bigl)dy=0,\quad\text{ for almost all }x\in\RR,
\end{align*}
which implies that the equation
\begin{align*}
    Vh_2=Y_2\bigl(U^{2}g-2UXf-2Uf\bigl).
\end{align*}
has a solution $h_2\in(\mathcal{H}_t)^\infty_{G'}$ by using Lemma \ref{le:8} again. For any $\omega\in \mathcal{H}_{t}^\infty$, by using the commutator relations
\begin{align*}
VU^2=U^2V-2UX-2U\quad\text{ and }\quad UY_2-Y_2U=Y_1
\end{align*}
we have
\begin{align*}
    -\langle h_2,V\omega\rangle&=\langle Y_2\bigl(U^{2}g-2UXf-2Uf\bigl),\omega\rangle\\
    &=\bigl\langle Y_2\bigl(U^{2}Vf-2UXf-2Uf\bigl),\omega\bigl\rangle\\
    &=\langle f,U^{2}VY_2\omega\rangle=-\langle Uf,UY_2V\omega\rangle\\
    &=-\langle Uf,(Y_2U+Y_1)V\omega\rangle\\
    &=\langle Y_2Uf,UV\omega\rangle+\langle Y_1Uf,V\omega\rangle.
    \end{align*}
Hence we have
\begin{align*}
    -\langle h_2+Y_1Uf,V\omega\rangle=\langle Y_2Uf,UV\omega\rangle
    \end{align*}
This shows that
\begin{align*}
UY_2Uf=h_2+Y_1Uf
\end{align*}
by Remark \ref{re:2} and Lemma \ref{le:10}. Hence $UY_2Uf\in (\mathcal{H}_t)_{G'}^\infty$.
Then we proved the case when $n=1$. Assume the result for $n=k$. Since $Xf$ is the solution of the equation \eqref{for:47}, by assumption $U^kXf$, $U^{k}f$, $UY_2U^kXf$ and $UY_2U^{k}f$ are all in $(\mathcal{H}_t)_{G'}^\infty$. Note that
\begin{align}\label{for:49}
    XU^k&=U^kX+e_kU^k, \qquad\qquad\text{ and }\notag\\
    U^{k+1}V&=VU^{k+1}+d_kU^kX+c_kU^{k},
    \end{align}
where $d_k,\,c_k$ are constants dependent on $k$. For any $\omega\in \mathcal{H}_{t}^\infty$, by using \eqref{for:45} and \eqref{for:49}  we have
\begin{align*}
    &\bigl\langle Y_2^2\bigl(U^{k+1}g-d_kU^kXf-c_kU^{k}f\bigl),\,\omega\bigl\rangle\\
     &=\bigl\langle Y_2^2\bigl(U^{k+1}Vf-d_kU^kXf-c_kU^{k}f\bigl),\,\omega\bigl\rangle\\
    &=(-1)^{k+2}\langle f,\,U^{k+1}VY_2^2\omega\rangle\notag\\
    &=\langle V(t\sqrt{-1}+Y_1^2V-Y_1Y_2X)U^kf,\,\omega\rangle.
\end{align*}
Hence we have
\begin{align*}
    V(t\sqrt{-1}+Y_1^2V-Y_1Y_2X)U^kf=Y_2^2\bigl(U^{k+1}Y_2g-d_kU^kXf-c_kU^{k}f\bigl).
\end{align*}
\eqref{for:109} of Lemma \ref{le:9} shows that
\begin{align*}
    \int\bigl(U^{k+1}g(x,y)-d_kU^kXf(x,y)-c_kU^{k}f(x,y)\bigl)dy=0
\end{align*}
for almost all $x\in\RR$, which implies that the equation
\begin{align}\label{for:50}
    Vh_3=U^{k+1}g-d_kU^kXf-c_kU^{k}f
\end{align}
has a solution $h_3\in (\mathcal{H}_t)_{G'}^\infty$ by using \eqref{for:62}. For any $\omega\in \mathcal{H}_{t}^\infty$, by using \eqref{for:49} we have
\begin{align*}
    -\langle h_3,V\omega\rangle&=\langle U^{k+1}g-d_kU^kXf-c_kU^{k}f,\omega\rangle\\
    &=\langle U^{k+1}Vf-d_kU^kXf-c_kU^{k}f,\omega\rangle\\
    &=(-1)^{k+2}\langle f,U^{k+1}V\omega\rangle\\
    &=\langle U^kf,UV\omega\rangle.
    \end{align*}
This shows that $U^{k+1}f=h_3$ by Remark \ref{re:2} and Lemma \ref{le:10}. Hence $U^{k+1}f\in (\mathcal{H}_t)_{G'}^\infty$. Since $Xf$ is the solution of the equation \eqref{for:47}, in the earlier argument, substituting $g$ and $f$ by $Xg+2g$ and $Xf$ respectively, we get $U^{k+1}Xf\in (\mathcal{H}_t)_{G'}^\infty$.

By \eqref{for:49} and commutator relations $VU=UV-X$ we have
\begin{align}\label{for:51}
    U^{k+2}V&=UVU^{k+1}+d_kU^{k+1}X+c_kU^{k+1}\notag\\
    &=(VU+X)U^{k+1}+d_kU^{k+1}X+c_kU^{k+1}\notag\\
    &=VU^{k+2}+d_kU^{k+1}X+c_kU^{k+1}+XU^{k+1}.
    \end{align}
For any $\omega\in \mathcal{H}^\infty$, by \eqref{for:45} and \eqref{for:51} we have
\begin{align*}
    &\bigl\langle Y_2^2\bigl(U^{k+2}g-d_kU^{k+1}Xf-c_kU^{k+1}f-XU^{k+1}f\bigl),\omega\bigl\rangle\\
     &=\bigl\langle Y_2^2\bigl(U^{k+2}Vf-d_kU^{k+1}Xf-c_kU^{k+1}f-XU^{k+1}f\bigl),\omega\bigl\rangle\\
    &=(-1)^{k+3}\langle f,\,U^{k+2}VY_2^2\omega\rangle=\langle U^{k+1}f,\,UY_2^2V\omega\rangle\notag\\
    &=\langle V(t\sqrt{-1}+Y_1^2V-Y_1Y_2X)U^{k+1}f,\,\omega\rangle.
\end{align*}
Then we get
\begin{align*}
 &V(t\sqrt{-1}+Y_1^2V-Y_1Y_2X)U^{k+1}f\\
 &=Y_2^2\bigl(U^{k+2}g-d_kU^{k+1}Xf-c_kU^{k+1}f-XU^{k+1}f\bigl).
\end{align*}
Then by Lemma \ref{le:9} we have
\begin{align*}
    \int\bigl(U^{k+2}g-d_kU^{k+1}Xf-c_kU^{k+1}f-XU^{k+1}f\bigl)dy=0.
\end{align*}
for almost all $x\in\RR$, which implies that the equation
\begin{align*}
    Vh_4=Y_2\bigl(U^{k+2}g-d_kU^{k+1}Xf-c_kU^{k+1}f-XU^{k+1}f\bigl)
\end{align*}
has a solution $h_4\in(\mathcal{H}_t)_{G'}^\infty$ by using Lemma \ref{le:8}. For any $\omega\in \mathcal{H}_{t}^\infty$, using the commutator relations  $UY_2-Y_2U=Y_1$ and $VY_2=Y_2V$ and \eqref{for:51} we have
\begin{align*}
    -\langle h_4,V\omega\rangle&=\bigl\langle Y_2\bigl(U^{k+2}g-d_kU^{k+1}Xf-c_kU^{k+1}f-XU^{k+1}f\bigl),\omega\bigl\rangle\\
    &=\bigl\langle Y_2\bigl(U^{k+2}Vf-d_kU^{k+1}Xf-c_kU^{k+1}f-XU^{k+1}f\bigl),\,\omega\bigl\rangle\\
    &=(-1)^{k+4}\langle f,\,U^{k+2}VY_2\omega\rangle=(-1)\langle U^{k+1}f,\,UY_2V\omega\rangle\\
    &=(-1)\langle U^{k+1}f,\,(Y_2U+Y_1)V\omega\rangle\\
    &=\langle Y_2U^{k+1}f,\,UV\omega\rangle+\langle Y_1U^{k+1}f,\,V\omega\rangle.
    \end{align*}
Hence we have
\begin{align*}
    -\langle h_4+Y_1U^{k+1}f,\,V\omega\rangle=\langle Y_2U^{k+1}f,\,UV\omega\rangle
    \end{align*}
This shows that
\begin{align*}
UY_2U^{k+1}f=h_4+Y_1U^{k+1}f
\end{align*}
by Remark \ref{re:2} and Lemma \ref{le:10}, which means $UY_2U^{k+1}f\in (\mathcal{H}_t)_{G'}^\infty$. Then we proved case when $n=k+1$. Hence we proved the claim completely.
\medskip

\emph{\textbf{Proof of \eqref{for:55}}}. From \eqref{for:62} and \eqref{for:63}, by using Theorem \ref{th:5} we get $f\in \mathcal{H}_t^\infty$; and the Sobolev estimates follow from \eqref{for:61} immediately.
\medskip

\emph{\textbf{Proof of \eqref{for:53}}}. When $t\neq 0$ it follows from \eqref{for:55} the equation $Vf'=g$ has a solution $f'\in (\mathcal{H}_t)^\infty$. Since $VY_2=Y_2V$, we see that $V(Y_2f')=Y_2g$, which shows that $f=V_2f'$ is a smooth solution to the equation $Vf=Y_2g$. The Sobolev estimates follow from \eqref{for:61} immediately.  When $t=0$, Lemma \ref{le:8} implies the equation $Vf=Y_2g$ has a solution $f\in (\mathcal{H}_t)^\infty_{G'}$.
  Note that $\rho_{0}\mid_{SL(2,\RR)}$ only contains principal series (see Remark \ref{re:6}). Lemma 4.7 in \cite{Forni} states  that for any unitary representation $(\varrho, \mathcal{L})$ of $SL(2,\RR)$ that only contains principal series, if the cohomological equation has a solution in $\mathcal{L}$, then the solution is a smooth vector in $\mathcal{L}$. Hence wee see that $f\in(\mathcal{H}_0)^\infty_{SL(2,\RR)}$. Since the Lie algebras of $G'$ and $SL(2,\RR)$ cover the Lie algebra of $SL(2,\RR)\ltimes\RR^2$, by using Theorem \ref{th:5}, we see that $f\in \mathcal{H}_{0}^\infty$.
 Also, the Sobolev estimates follow from \eqref{for:61} immediately.

\medskip

\emph{\textbf{Proof of \eqref{for:64}}}. For $t\neq 0$, \eqref{for:109} of Lemma \ref{le:9} and \eqref{for:55} implies the conclusion. For $t=0$, arguments in \eqref{for:53} show that $f\in(\mathcal{H}_0)^\infty_{SL(2,\RR)}$. \eqref{for:109} of Lemma \ref{le:9} and \eqref{for:53} show that the equation $Vh=Y_2g$ has a solution $h\in \mathcal{H}_{0}^\infty$. For any $\omega\in \mathcal{H}_0^\infty$, we have
\begin{align*}
    -\langle h,V\omega\rangle&=\langle Y_2g,\omega\rangle=-\langle Vf,Y_2\omega\rangle=\langle f,Y_2V\omega\rangle.
    \end{align*}
This shows that $Y_2f=h$ by Remark \ref{re:2} and Lemma \ref{le:10}. Then $Y_2f\in \mathcal{H}_0^\infty$. Note that $Xf$ is the solution of equation \eqref{for:47}. Substituting $g$ and $f$ with $Xg+2g$ and $Xf$ respectively, we see that $Y_2Xf\in \mathcal{H}_0^\infty$. Since $Uf$ is the solution of equation $VUf=Ug-Xf$, by \eqref{for:109} of Lemma \ref{le:9} we get
\begin{align*}
\int (Ug-Xf)dy=0\qquad\text{ for almost all }x\in\RR,
\end{align*}
which shows that the equation
\begin{align*}
 Vh_1=Y_2(Ug-Xf)
\end{align*}
has a solution
$h_1\in \mathcal{H}_0^\infty$ by using \eqref{for:53}. For any $\omega\in \mathcal{H}_0^\infty$, from the commutator relation $VU=UV-X$ we have
\begin{align*}
    -\langle h_1,V\omega\rangle&=\langle Y_2(Ug-Xf),\omega\rangle=-\langle (UVf-Xf),\,Y_2\omega\rangle\\
    &=-\langle f,UY_2V\omega\rangle=\langle Uf,Y_2V\omega\rangle.
    \end{align*}
This shows that $Y_2Uf=h_1$ by Remark \ref{re:2} and Lemma \ref{le:10}. Hence $Y_2Uf\in \mathcal{H}_0^\infty$.
Using the relation  $UY_2-Y_2U=Y_1$, for any $\omega\in \mathcal{H}_0^\infty$ we have
\begin{align*}
    -\langle f,Y_1\omega\rangle&=-\langle f,(UY_2-Y_2U)\omega\rangle=\langle (UY_2-Y_2U)f,\,\omega\rangle.
    \end{align*}
This shows that
\begin{align*}
Y_1f=(UY_2-Y_2U)f.
\end{align*}
Hence $Y_1f\in \mathcal{H}_0^\infty$. We already have showed that $Y_2f\in \mathcal{H}_0^\infty$ and $f\in(\mathcal{H}_0)^\infty_{SL(2,\RR)}$, which implies that $f\in \mathcal{H}_0^\infty$ by Theorem \ref{th:5}. Also, the Sobolev estimates follow from \eqref{for:61} immediately.
\subsection{Global coboundary for the unipotent flow in $G=SL(2,\RR)\ltimes\RR^2$}\label{sec:7}
Let $(\pi,\mathcal{H})$ be a unitary representation of $SL(2,\RR)\ltimes\RR^2$ without non-trivial $\RR^2$-invariant vectors. We now discuss how to obtain a global solution from the solution which exists in each irreducible component of $\mathcal{H}$. By general arguments in Section \ref{sec:3} there is a direct decomposition of
$\mathcal{H}=\int_Z \mathcal{H}_zd\mu(z)$ of irreducible unitary representations of $G$ for some measure space $(Z,\mu)$. If $\pi$ has no non-trivial $\RR^2$-invariant vectors, then for almost
all $z\in Z$, $\pi_z$ has non-trivial $\RR^2$-invariant vectors. Hence we can apply
Theorem \ref{po:2} to prove the following:
\begin{corollary}\label{cor:3}
Let $(\pi,\mathcal{H})$ be a unitary representation of $SL(2,\RR)\ltimes\RR^2$ without non-trivial $\RR^2$-invariant vectors. If $g\in \mathcal{H}^\infty$ and the cohomological equation $Vf=g$  has a solution $f\in \mathcal{H}$, then  $f\in \mathcal{H}^\infty$ and satisfies
\begin{align*}
  \norm{f}_{t}\leq C_{t}\norm{g}_{t+6},\qquad \forall\,t\geq 0.
\end{align*}
\end{corollary}
\begin{proof}
The cohomological equation $Vf=g$ has a decomposition $Vf_z=g_z$ with $g_z\in\mathcal{H}_z^\infty$ for almost all $z\in Z$. \eqref{for:64} of Theorem \ref{po:2} shows that
\begin{align*}
\norm{f_z}_t\leq C_t\norm{g_z}_{t+6}, \qquad \forall\,t\geq 0\text{ and }\forall\,z\in Z
\end{align*}
Noting that $C_t$ are constants only dependent on $t$, we get
\begin{align*}
    \norm{f}_t^2=\int_Z\norm{f_z}^2_td\mu(z)\leq C_t^2\int_Z\norm{g_z}^2_{t+6}d\mu(z)=C_t^2\norm{g}^2_{t+6},
\end{align*}
which proves the claim.
\end{proof}
\begin{remark}\label{re:4}
We remark at this point that the condition $\int gdy=0$ is not sufficient to guarantee the existence of a solution $f\in \mathcal{H}_t$ of the cohomological
equation $Vf=g$ for any irreducible component $(\rho_t,\,\mathcal{H}_t)$. The problem only arises at $t=0$. If $g=h(x)h(y)$ where $h$ is a smooth function on $\RR$
with compact support and satisfies: $h=1$ on $[-1,1]$ and $\int h(x)dx=0$. Obviously, $g\in \mathcal{H}_0^\infty$. If $f\in \mathcal{H}_0$ is a solution to the cohomological equation $Vf=g$, then we have $-\partial_yf\cdot x=g$ which has the form
\begin{align*}
-\hat{f}_\xi\cdot x\xi\sqrt{-1}=h(x)\hat{h}(\xi)
\end{align*}
by taking Fourier transformation on fact $y$
as in \eqref{for:2}. Then $\int gdy=0$ implies that $\hat{h}(\xi)\cdot \xi^{-1}\in L^2(\RR)$, but $h(x)\cdot x^{-1}\notin L^2(\RR)$ by assumption. Then we have a contradiction. Noting that \eqref{for:53} of Theorem \ref{po:2} shows that the equation $Vp=Y_2g$ has a solution $p\in\mathcal{H}_0^\infty$,  the example also means that the  cocycle equation $Vp=Y_2g$ fails to have a common solution in $\mathcal{H}_0$.
\end{remark}

\subsection{Coboundary for the unipotent flow of $G=SL(n,\RR)$, $n\geq 3$}\label{sec:14} Before proceeding further with the proof of Theorem \ref{th:6}, we prove
certain technical results which are very useful for the discussion.
\begin{definition}\label{de:2}
For $m\geq 3$, let $G_m$ be the closed subgroup generated by $U_{1,2}$, $U_{2,1}$, $U_{1,m}$, and $U_{2,m}$ and let $H_m$ be the subgroup generated by $U_{1,2}$, $U_{2,1}$, $U_{m,1}$, and $U_{m,2}$.
\end{definition}
Then $G_m$ and $H_m$ are isomorphic to $SL(2,\RR)\ltimes\RR^2$. Let $\mathcal{A}=\{\mathfrak{u}_{1,i}, \,\mathfrak{u}_{j,1}: j\geq 3,\,i\geq 2\}$. Next, we will prove:
\begin{lemma}\label{le:11}
Let $\Pi$ be a  unitary representation of $G=SL(n,\RR)$, $n\geq3$ without non-trivial $G$-fixed vectors. If $\mathfrak{u}_{1,2}f=g$ where $g\in \mathcal{H}^\infty$ has a solution $f\in \mathcal{H}$, then for any $n\in\NN$ and $Y_i\in \mathcal{A}$, $1\leq i\leq n$, $Y_1\cdots Y_nf\in \mathcal{H}^\infty_{G_m}$ and $Y_1\cdots Y_nf\in \mathcal{H}^\infty_{H_m}$, $m\geq 3$.
\end{lemma}
\begin{proof}
We will prove the lemma inductively. Since $\pi\mid_{G_m}$ and $\pi\mid_{H_m}$, $m\geq 3$ have no $\RR^2$-invariant vectors by Howe-Moore,
Applying Corollary \ref{cor:3} to the equation $\mathfrak{u}_{1,2}f=g$ in each $\pi\mid_{G_m}$ and $\pi\mid_{H_m}$, $m\geq 3$ we find that $f$ is in all $\mathcal{H}^\infty_{G_m}$ and $\mathcal{H}^\infty_{H_m}$, $m\geq 3$.

Using the relations $\mathfrak{u}_{1,2}\mathfrak{u}_{1,i}=\mathfrak{u}_{1,i}\mathfrak{u}_{1,2}$, $i\geq 2$, and $\mathfrak{u}_{1,2}\mathfrak{u}_{j,2}=\mathfrak{u}_{j,2}\mathfrak{u}_{1,2}$, $j\geq 3$ we see that $\mathfrak{u}_{1,i}f$ and $\mathfrak{u}_{j,2}f$ satisfy the following equations
\begin{align}
\mathfrak{u}_{1,2}\mathfrak{u}_{1,i}f&=\mathfrak{u}_{1,i}g\qquad\text{ and }\label{for:68}\\
\mathfrak{u}_{1,2}\mathfrak{u}_{j,2}f&=\mathfrak{u}_{j,2}g\label{for:57}.
\end{align}
Applying Corollary \ref{cor:3} to the above two equations in each $\pi\mid_{G_m}$ and $\pi\mid_{H_m}$, $m\geq 3$, we also get that $\mathfrak{u}_{1,i}f$ and $\mathfrak{u}_{j,2}f$ are in all $\mathcal{H}^\infty_{G_m}$ and $\mathcal{H}^\infty_{H_m}$, $m\geq 3$.

In $\pi\mid_{H_j}$, by using the relation $\mathfrak{u}_{1,2}\mathfrak{u}_{j,1}=\mathfrak{u}_{j,1}\mathfrak{u}_{1,2}-\mathfrak{u}_{j,2}$, we see $\mathfrak{u}_{j,1}f$ satisfies the equation
\begin{align}\label{for:54}
\mathfrak{u}_{1,2}\mathfrak{u}_{j,1}f=\mathfrak{u}_{j,1}g-\mathfrak{u}_{j,2}f.
\end{align}
Note the right side of the above equation is in $\mathcal{H}^\infty_{G_m}$ and $\mathcal{H}^\infty_{H_m}$, $m\geq 3$ by earlier arguments. Then the same arguments as above show that $\mathfrak{u}_{j,1}f$, $j\geq 3$ are in all $\mathcal{H}^\infty_{G_m}$ and $\mathcal{H}^\infty_{H_m}$, $m\geq 3$. Then we proved the claim when  $n=1$.

Assume the result for $n=k$. Since $\mathfrak{u}_{j,2}f$ is the solution of  \eqref{for:57} for each $j\geq 3$, by assumption we see that
\begin{align}\label{for:70}
Y_1\cdots Y_k\mathfrak{u}_{j,2}f\quad\text{ and }\quad Y_1\cdots Y_kf
\end{align}
where $Y_i\in\mathcal{A}$, $1\leq i\leq k$ are in all $\mathcal{H}^\infty_{H_m}$ and $\mathcal{H}^\infty_{G_m}$, $m\geq 3$. For any $Y_i\in \mathcal{A}$, $1\leq i\leq k+1$, we have relations
\begin{align}\label{for:69}
    \mathfrak{u}_{1,2}Y_2\cdots Y_{k+1}&=Y_2\cdots Y_{k+1}\mathfrak{u}_{1,2}+\sum_i d_iY_2\cdots Y_{i-1}Y_{i+1}\cdots Y_{k+1}\mathfrak{u}_{j(i),2}\notag\\
    &+\sum_{i,j} c_{i,j}Y_2\cdots Y_{i-1}Y_{i+1}\cdots Y_{j-1}\mathfrak{u}_{1,2}Y_{j+1}\cdots Y_{k+1},
\end{align}
where $d_i$ and $c_{i,j}$ are constants and $j(i)\geq 3$ satisfies $0\neq [Y_i,\,\mathfrak{u}_{1,2}]=\mathfrak{u}_{j(i),2}$.
Let
\begin{align*}
    f'&=Y_2\cdots Y_{k+1}g+\sum_i d_iY_2\cdots Y_{i-1}Y_{i+1}\cdots Y_{k+1}\mathfrak{u}_{j(i),2}f\notag\\
    &+\sum_i c_jY_2\cdots Y_{i-1}Y_{i+1}\cdots Y_{j-1}\mathfrak{u}_{1,2}Y_{j+1}\cdots Y_{k+1}f.
\end{align*}
From \eqref{for:70}, we see that $f'$ and $Yf'$ are in all $\mathcal{H}^\infty_{G_m}$ and $\mathcal{H}^\infty_{H_m}$, $m\geq 3$ for any $Y\in\mathcal{A}$.
For any $\omega\in \mathcal{H}^\infty$, using $\mathfrak{u}_{1,2}f=g$ and \eqref{for:69} we have
\begin{align*}
    \langle f',\,\omega\rangle&=(-1)^{k+1}\langle f,\,Y_{k+1}\cdots Y_2 \mathfrak{u}_{1,2}\omega\rangle=-\langle Y_2\cdots Y_{k+1}f, \,\mathfrak{u}_{1,2}\omega\rangle
   \end{align*}
This shows that
\begin{align*}
  \mathfrak{u}_{1,2}Y_2\cdots Y_{k+1}f=f'.
\end{align*}
Since $Y_2\cdots Y_{k+1}f$ and $f'$ are in $\mathcal{H}^\infty_{G_m}$ and $\mathcal{H}^\infty_{H_m}$, $m\geq 3$, an argument similar to that in obtaining \eqref{for:68} and \eqref{for:57} shows that
\begin{align*}
\mathfrak{u}_{1,2}\mathfrak{u}_{1,i}Y_2\cdots Y_{k+1}f&=\mathfrak{u}_{1,i}f'\qquad \forall\,i\geq 2,\qquad \text{ and }\\
\mathfrak{u}_{1,2}\mathfrak{u}_{j,2}Y_2\cdots Y_{k+1}f&=\mathfrak{u}_{j,2}f'\qquad \forall\,j\geq 3.
\end{align*}
Since $\mathfrak{u}_{1,i}f'$ and $\mathfrak{u}_{j,2}f'$ are in all $\mathcal{H}^\infty_{G_m}$ and $\mathcal{H}^\infty_{H_m}$, $m\geq 3$, in exactly the same manner
as before we see that for $i\geq 2$ and $j\geq 3$
\begin{align*}
\mathfrak{u}_{1,i}Y_2\cdots Y_{k+1}f\quad\text{ and }\quad \mathfrak{u}_{j,2}Y_2\cdots Y_{k+1}f
\end{align*}
are in all $\mathcal{H}^\infty_{G_m}$ and $\mathcal{H}^\infty_{H_m}$, $m\geq 3$. Also, an analogous argument to that in obtaining \eqref{for:54} shows that
\begin{align*}
 \mathfrak{u}_{1,2}\mathfrak{u}_{j,1}Y_2\cdots Y_{k+1}f=\mathfrak{u}_{j,1}f'-\mathfrak{u}_{j,2}Y_2\cdots Y_{k+1}f.
\end{align*}
Since the right side the above equation is in all $\mathcal{H}^\infty_{G_m}$ and $\mathcal{H}^\infty_{H_m}$, $m\geq 3$, by exactly the same argument we see that $\mathfrak{u}_{j,1}Y_2\cdots Y_{k+1}f$ are in all $\mathcal{H}^\infty_{G_m}$ and $\mathcal{H}^\infty_{H_m}$, $m\geq 3$. Then we proved the case when $n=k+1$ and thus finish the proof.
\end{proof}
\begin{remark}The purpose of the above lemma is a preparation to prove that $f\in \mathcal{H}^\infty$.
If $\Gamma$ is cocompact and $\mathcal{H}=L_0^2(G/\Gamma)$, the space of square integrable functions on $G/\Gamma$ with zero average, then the fact that $f$ is in all $\mathcal{H}^\infty_{G_m}$ and $\mathcal{H}^\infty_{H_m}$, $m\geq 3$ implies that $f$ is a smooth function. Since the Lie algebras of $G_m$ and $H_m$, $m\geq 3$ generate the whole Lie algebra of $G$, the result is a direct consequence of subelliptic regularity theorem on compact manifolds (see Theorem \ref{th:5}).
\end{remark}
\subsection{Proof of Theorem \ref{th:6}} Note that the Weyl group operates simply transitive on the set of Weyl chambers, we may assume that $\mathfrak{u}_{i,j}=\mathfrak{u}_{1,2}$. We take notations in Lemma \ref{le:11} if there is no confusion.

\noindent \textbf{\emph{Proof of \eqref{for:33}}}
Since for $i\geq 3$ and $j\geq 2$,
\begin{align*}
u_{i,j}^n&=(u_{i,1}u_{1,j}-u_{1,j}u_{i,1})^n,\\
X_{i,i+1}^n&=\bigl(u_{i,i+1}u_{i+1,i}-u_{i+1,i}u_{i,i+1}\bigl)^n,
\end{align*}
where $X_{k,l}=\diag(0,\cdots,\underset{k}{1},\cdots,\underset{l}{-1},\cdots,0)\in\mathfrak{sl}(n,\RR)$, it follows from Lemma \ref{le:11} that $u_{i,j}^nf\in \mathcal{H}$ and $X_{i,i+1}^nf\in \mathcal{H}$  for each $n$.

Let $\mathcal{B}=\{\mathfrak{u}_{j,2}, \,\mathfrak{u}_{2,j}, j\geq 3\}$. Along the same lines as that of the proof of Lemma \ref{le:11}, we can show that for any $n\in\NN$ and $Y_i\in \mathcal{B}$ $1\leq i\leq n$, $Y_1\cdots Y_nf\in \mathcal{H}$.
Since
\begin{align*}
X_{2,3}^n&=\bigl(u_{2,3}u_{3,2}-u_{3,2}u_{2,3}\bigl)^n,
\end{align*}
it follows that  $X_{2,3}^nf\in \mathcal{H}$  for each $n$. Since the Lie algebras of $G_m$ and $H_m$, $m\geq 3$, directions $u_{i,j}$, $i\geq 3$, $j\geq 2$ and $X_{i,i+1}$, $i\geq 2$ cover $\mathfrak{sl}(n,\RR)$, by using Theorem \ref{th:5} we see that $f\in \mathcal{H}^\infty$, which implies that $\mathcal{D}(g)=0$ for all $\mathcal{D}\in\mathcal{E}_{U_{1,2}}(\mathcal{H})$. Then it follows from Proposition \ref{le:6} that for any $m\in\NN$, $i,\,j\geq 3$ and $1\leq k\neq l\leq n$
\begin{align*}
 \norm{U_{i,j}^mf}\leq C_{m}\norm{g}_{m+2}\quad \text{ and }\quad\norm{X_{k,l}^mf}\leq C_{m}\norm{g}_{m+2}.
\end{align*}
Moreover, Corollary \ref{cor:3} shows that
\begin{align*}
 \norm{f}_{G_m,t}&\leq C_{t}\norm{g}_{G_m,t+6}\quad \text{ and }\quad\norm{f}_{H_m,t}\leq C_{t}\norm{g}_{H_m,t+6},\quad \forall\,t>0.
\end{align*}
Then the Sobolev estimates of $f$ follow from the above estimates and Theorem \ref{th:5}.

\noindent \textbf{\emph{Proof of \eqref{for:35}}} To apply Proposition \ref{le:6}, it suffices to prove that $\pi\mid_{G_{\mathfrak{u}_{1,2}}}$ has a spectral gap.
By Howe-moore, $\pi\mid_{G_3}$ has no non-trivial $\RR^2$-invariant vectors. Then Remark \ref{re:2} shows that $\pi\mid_{G_{\mathfrak{u}_{1,2}}}$ is outside a fixed neighborhood of trivial representation in the Fell topology, which proves the claim. Then the result follows immediately from Proposition \ref{le:6}.
\subsection{Weak cocycle rigidity in $G=SL(2,\RR)\ltimes\RR^2$} Remark \ref{re:4} shows that generally, the cocycle equation $Vf=Y_2g$ fails to have common solutions. However, we have a weak version of cocycle rigidity:
\begin{lemma}\label{le:7}
Let $(\pi,\mathcal{H})$ be a unitary representation of $SL(2,\RR)\ltimes\RR^2$ without non-trivial $\RR^2$-invariant vectors. If $f,\,g\in \mathcal{H}^\infty$ and $g=Y_2p$ for some $p\in \mathcal{H}^\infty$, then  the cocycle equation $Vf=Y_2g$  has a common solution $h\in \mathcal{H}^\infty$ with estimates
\begin{align*}
    \norm{h}_s\leq C_s\max\{\norm{g}_{s+6}, \,\norm{f}_{s+6}\}.
\end{align*}
\begin{proof}
The discussion in Section \ref{sec:7} allows us to reduce our analysis of cocycle equations to each irreducible component $(\rho_t,\,\mathcal{H}_t)$ that appears in $\mathcal{H}$. If $t\neq 0$ for the cocycle equation $Vf_t=Y_2g_t$ in $\mathcal{H}_t$, \eqref{for:109} of Lemma \ref{le:9} shows that $\int (g_t(x,y)\cdot x)dy$ for almost all $x\in\RR$. Then immediately, we get $\int g_tdy=0$ for almost all $x\in\RR$, which shows that the equation $Vh_t=g_t$ has a solution $h_t\in \mathcal{H}^\infty$ by using \eqref{for:55} of Theorem \ref{po:2}. Lemma \ref{le:14} shows that $h_t$ is a common solution.

If $t=0$, the assumption means that
\begin{align*}
-\partial_yf_0\cdot x=Vf_0=Y_2g_0=Y_2Y_2p_0=-p_0\cdot x^2,
\end{align*}
which implies that $\int p_tdy=0$ for almost all $x\in\RR$ by \eqref{for:109} of Lemma \ref{le:9}. Then the equation $Vh_0=Y_2p_0=g$ has a solution $h_0\in \mathcal{H}^\infty$ by using \eqref{for:53} of Theorem \ref{po:2}. Again, Lemma \ref{le:14} show that $h_0$ is a common solution.

Hence we showed the existence of the common solution $h_t$ in  all $\mathcal{H}_t$ that appear in $\mathcal{H}$. Further, \eqref{for:64} of Theorem \ref{po:2} gives the Sobolev estimates
\begin{align*}
    \norm{h_t}_s\leq C_s\max\{\norm{g_t}_{s+6}, \,\norm{f_t}_{s+6}\}
\end{align*}
for any such $h_t$. Since these constants $C_s$ are independent of $t$, we get a global common solution $h\in \mathcal{H}^\infty$ with Sobolev estimates
\begin{align*}
    \norm{h}_s\leq C_s\max\{\norm{g}_{s+6}, \,\norm{f}_{s+6}\},
\end{align*}
which proves the claim.
\end{proof}
\end{lemma}

\section{Cocycle rigidity for $(SL(2,\RR)\ltimes \RR^2)\ltimes\RR^3$}\label{sec:9}
In this part we assume notation in Section \ref{sec:10}. Recall that $L$ denotes the rank $2$ subgroup $\begin{pmatrix}v_1 \\
v_2\\
0\\
 \end{pmatrix}$ in $\RR^3$. Suppose the group $(SL(2,\RR)\ltimes \RR^2)\ltimes\RR^3$ is as described in Section \ref{sec:10}. The purpose of  this section is to prove the following:
\begin{proposition}\label{po:1}
For any unitary representation $(\Pi,\mathcal{H})$ of $(SL(2,\RR)\ltimes \RR^2)\ltimes\RR^3$ without non-trivial $L$-invariant vectors, if $f,\,g\in\mathcal{H}^\infty$,
\begin{enumerate}
  \item \label{for:5}the cocycle equation $U_1f=Y_2g$ has a common solution $p\in\mathcal{H}^\infty$ satisfying
\begin{align*}
 \norm{p}_s\leq C_s\max\{\norm{g}_{s+7}, \,\norm{f}_{s+7}\},
\end{align*}

\medskip
  \item \label{for:18}if $g=Y_1h$ for some $h\in \mathcal{H}^\infty$, the cocycle equation $U_1f=U_2g$ has a common solution $p\in\mathcal{H}$ satisfying
\begin{align*}
 \norm{p}_s\leq C_s\max\{\norm{g}_{s+7}, \,\norm{f}_{s+7}\}.
\end{align*}
\end{enumerate}

\end{proposition}
Next, we will prove some  technical results whose roles will be clear from the subsequent development.
\subsection{Unitary representations of $\mathbb{S}=SL(2,\RR)\times \RR$}\label{sec:8} Let $\mathbb{S}$ denote the following subgroup of $(SL(2,\RR)\ltimes \RR^2)\ltimes\RR^3$: $\begin{pmatrix}[ccc|c]
  a & b & 0& 0\\
c & d & 0& 0\\
0 & 0 & 1& v\\
 \end{pmatrix}$, where $\begin{pmatrix}
  a & b\\
c &  d\\
 \end{pmatrix}\in SL(2,\RR)$ and  $v\in\RR$. For any $z,\,t,\,r\in \RR$ we define a unitary representation of $\mathbb{S}$ as follows:
for each $z\in\RR$ the group action $\mathbb{S}$ is defined by:
\begin{gather*}
\tau_{(t,r,z)}: SL(2,\RR)\times\RR\rightarrow \mathcal{B}(\mathbb{V}_{(t,r,z)})\\
\tau_{(t,r,z)}(s)f(x,y)=e^{(p_1r+p_2t-zv)\sqrt{-1}}f(-cy+xa,yd-bx);
 \end{gather*}
 and
\begin{align*}
 \norm{f}_{\mathbb{V}_{(t,r,z)}}=\norm{f}_{L^2(\RR^2)},
\end{align*}
where $s=\begin{pmatrix}[ccc|c]
  a & b & 0& 0\\
c & d & 0& 0\\
0 & 0 & 1& v\\
 \end{pmatrix}$, $p_1=cx^{-1}(xa-cy)^{-1}$ and $p_2=zp_1$.

 The vector fields for $\tau_{(t,r,z)}$ on $\mathbb{V}_{(t,r,z)}$ are
 \begin{gather*}
 X=x\partial_x-y\partial_y,\qquad U_1=-x\partial_y,\qquad Y_3=-z \sqrt{-1},\notag\\
V_1=-y\partial_x+(r+tz)x^{-2}\sqrt{-1}.
  \end{gather*}
Compared with Lemma \ref{le:2}, we see  that for any $z,\,t,\,r\in \RR$,  $\tau_{(t,r,z)}$ are unitary representations of $\mathbb{S}$.
\subsection{Decomposition of $\Pi_{(t,r)}$ into direct integrals of representations on $\mathbb{S}$}\label{sec:18}
For any $t,\,r\in\RR$, set
\begin{align*}
\mathbb{V}_{(t,r)}=\int \mathbb{V}_{(t,r,z)}dz.
\end{align*}
From the discussion in Section \ref{sec:5}, we see that $\mathbb{H}_{(t,r)}=L^3(\RR^3,\CC)$ and $\mathbb{V}_{(t,r)}=L^2(\RR,L^2(\RR^2,\CC))$. Define a map
\begin{gather*}
L^3(\RR^3,\CC)\xrightarrow{\mathcal{F}} L^2(\RR,L^2(\RR^2,\CC)):\,\,\mathcal{F}(h)(z)(x,y)=h(x,y,z)
\end{gather*}
for any $h\in L^3(\RR^3,\CC)$. It is clear that $\mathcal{F}$ is injective and unitary. Then $\mathbb{H}_{(t,r)}$  can be unitarily imbedded in to $\mathbb{V}_{(t,r)}$. Conversely, for any $k\in L^2(\RR,L^2(\RR^2,\CC))$, the map
\begin{align*}
    h\rightarrow\int_{\RR}\bigl\langle k(z), \mathcal{F}(h)(z)\bigl\rangle_{L^2(\RR^2)} dz
\end{align*}
defines a bounded linear function on $L^3(\RR^3,\CC)$. Then by Riesz representation theorem, there exists a unique element $\bar{\mathcal{F}}(k)$ in $L^3(\RR^3,\CC)$ such that
\begin{align*}
    \bigl\langle \bar{\mathcal{F}}(k),\, h\bigl\rangle_{L^3(\RR^3,\CC)}=\int_{\RR}\bigl\langle k(z),\, \mathcal{F}(h)(z)\bigl\rangle_{L^2(\RR^2)} dz,
\end{align*}
for any $h\in L^3(\RR^3,\CC)$. It is easy to check that $\bar{\mathcal{F}}$ is unitary and surjective. Since $\bar{\mathcal{F}}\circ \mathcal{F}(h)=h$
for any $h\in L^3(\RR^3,\CC)$, $\mathcal{F}$ is a unitary isomorphism. Moreover, by definition, it is clear that
\begin{align*}
\mathcal{F}(\Pi_{(t,r)}(s)h)&=\int_{\RR}\tau_{(t,r,z)}(s)\bigl((\mathcal{F}h)(z)\bigl)dz
\end{align*}
for any $(t,r)\in\RR^2$, where $s\in \mathbb{S}$ and $h\in \mathbb{H}_{(t,r)}$. Then we have:
\begin{lemma}\label{le:3}
The map $L^3(\RR^3,\CC)\xrightarrow{\mathcal{F}} L^2(\RR,L^2(\RR^2,\CC))$:
\begin{align*}
    h(x,y,z)\rightarrow \int_\RR \mathcal{F}(h)(z)dz\quad\text{ where }\quad \mathcal{F}(h)(z)(x,y)=h(x,y,z)
    \end{align*}
establishes unitary equivalences between unitary representations  $(\Pi_{(t,r)},\mathbb{H}_{(t,r)})$ and $(\int_{\RR}\tau_{(t,r,z)}dz,\mathbb{V}_{(t,r)})$ over $\mathbb{S}$.
\end{lemma}
Hence we can write
\begin{align}\label{for:12}
  \mathbb{H}_{(t,r)}&=\int_\RR\mathbb{V}_{(t,r,z)}dz\quad\text{ and }\quad\Pi_{(t,r)}=\int_\RR\tau_{(t,r,z)}dz
  \end{align}
over $\mathbb{S}$.

\subsection{Cocycle rigidity in $(\tau_{(t,r,z)},\mathbb{V}_{(t,r,z)})$ }
\begin{lemma}\label{le:4}For any $t,\,r\in\RR$, there exists a full measure set $\Omega\subset\RR$ such that: for any $z\in\Omega$ and any $f,\,g\in \mathbb{V}_{(t,r,z)}^\infty$, the cocycle equation $U_1f=Y_3g$ has a solution $p\in \mathbb{V}_{(t,r,z)}$ satisfying $\norm{p}\leq C\norm{g}_{2}$.
\end{lemma}
\begin{proof}
Let $\mathbb{S}_1$ denote the following subgroup of $(SL(2,\RR)\ltimes \RR^2)\ltimes\RR^3$ which is isomorphic to $SL(2,\RR)\ltimes\RR^2$: $\begin{pmatrix}[ccc|c]
  a & b & 0& v_1\\
c & d & 0& v_2\\
0 & 0 & 1& 0\\
 \end{pmatrix}$, where $\begin{pmatrix}
  a & b\\
c &  d\\
 \end{pmatrix}\in SL(2,\RR)$ and  $(v_1,v_2)^\tau\in\RR^2$. Since $\Pi_{(t,r)}\mid_{\mathbb{S}_1}$ has no non-trivial $\RR^2$-invariant vectors for any $t,\,r\in\RR$ (see Section \ref{sec:2}), Corollary \ref{cor:2} shows that $\Pi_{(t,r)}\mid_{SL(2,\RR)}$ is tempered. Hence it follows from Theorem \ref{th:3} that there exists a full measure set $\Omega\subset\RR$ such that $\tau_{(t,s,z)}\mid_{SL(2,\RR)}$ is tempered for almost all $z\in\Omega$.

For any $z\in\Omega\backslash 0$, in $(\tau_{(t,r,z)},\mathbb{V}_{(t,r,z)})$ the cocycle equation has the expression
\begin{align*}
 -\partial_yf\cdot x=g\cdot(-z\sqrt{-1}).
\end{align*}
Let $p=f\cdot z^{-1}\sqrt{-1}$. It is clear that $p\in \mathbb{V}^\infty_{(t,r,z)}$ and is a common solution to the cocycle equation, that is $U_1p=g$ and $Y_3p=f$. Then it follows from Theorem \ref{th:2} that $\norm{p}\leq C\norm{g}_2$. Note that by Remark \ref{re:2} $C$ is independent of $t,\,r,\,z$.
\end{proof}
The discussion in Section \ref{sec:7} shows that the following result is a direct consequence of Lemma \ref{le:4}
and the decomposition in \eqref{for:12}:
\begin{lemma}\label{le:5}
In $(\mathbb{H}_{(t,r)},\Pi_{(t,r)})$, if $f,\,g\in \mathbb{H}_{(t,r)}^\infty$ and satisfy the cocycle equation $U_1f=Y_3g$, then the equation has a common solution $p\in \mathbb{H}_{(t,r)}$ with the estimate $\norm{p}\leq C\norm{g}_{2}.$
\end{lemma}
\subsection{Proof of Proposition \ref{po:1}} \textbf{\emph{Proof of \eqref{for:5}}}
Again, from the discussion in Section \ref{sec:7}, we get that the cocycle equation $U_1f=Y_3g$ has solution $p\in \mathcal{H}$ by using Lemma \ref{le:5}.
Since $\Pi\mid_{\mathbb{S}_1}$ has no-nontrivial $\RR^2$-invariant vectors ($\mathbb{S}_1$ is defined in proof of Lemma \ref{le:4}), by Corollary \ref{cor:3} we have
\begin{align*}
    \norm{p}_{\mathbb{S}_1,s}\leq C_s\norm{g}_{s+6},\qquad \forall s\geq 0.
\end{align*}
Since $Y_3p=f$, for any $n\in\NN$ we have
\begin{align*}
    \norm{Y_3^np}=\norm{Y_3^{n-1}f}\leq \norm{f}_{n-1}.
\end{align*}
Let $\mathbb{S}_2$ denote the subgroup $\begin{pmatrix}[ccc|c]
  a & b & u_1& 0\\
c & d & u_2& 0\\
0 & 0 & 1& 0\\
 \end{pmatrix}$, where $\begin{pmatrix}
  a & b\\
c &  d\\
 \end{pmatrix}\in SL(2,\RR)$ and $\begin{pmatrix}u_1 \\
u_2\\
 \end{pmatrix}\in\RR^2$. By Remark \ref{re:5} and Corollary \ref{cor:3} we have
 \begin{align*}
    \norm{p}_{\mathbb{S}_1,s}\leq C_s\norm{g}_{s+6},\qquad \forall s\geq 0.
 \end{align*}
Since the Lie algebras of $\mathbb{S}_1$, $\mathbb{S}_2$ and one-parameter subgroup $\{tY_2\}_{t\in\RR}$ cover the Lie algebra of $(SL(2,\RR)\ltimes \RR^2)\ltimes\RR^3$, by using Theorem \ref{th:5} we get the claim.

\medskip
\noindent\textbf{\emph{Proof of \eqref{for:18}}} The discussion in Section \ref{sec:7} allows us to reduce our analysis of cocycle equations to each irreducible component $(\Pi_{t,r},\,\mathbb{H}_{t,r})$ that appears in $\mathcal{H}$. Using relations in \eqref{for:7} we get
 \begin{align*}
-\partial_yf_{t,r}\cdot x=U_1f_{t,r}=U_2g_{t,r}=U_2Y_1h_{t,r}=-x^2\sqrt{-1}\partial_zh_{t,r}.
\end{align*}
Taking Fourier transformation on factor $z$ as in \eqref{for:109} of Lemma \ref{le:9} to each side of the above equation, we get
\begin{align}\label{for:10}
-\partial_y\hat{f}_{t,r}(x,y,\xi)\cdot x=-x^2\hat{h}_{t,r}(x,y,\xi)\cdot \xi.
\end{align}
Noting that
\begin{align*}
\hat{(Y_2h)}_{t,r}=-\hat{h}_{t,r}\cdot y\sqrt{-1}\quad\text{ and }\quad \hat{(Y^2_2h)}_{t,r}=-\hat{h}_{t,r}\cdot y^2
\end{align*}
are both in $L^2(\RR^3)$, along the same lines as that of \eqref{for:109} and \eqref{for:110} of Lemma \ref{le:9}, from \eqref{for:10} we have
\begin{itemize}
  \item $\int\hat{h}_{t,r}dy=0$ for almost all $(x,\xi)\in\RR^2$,

  \medskip
  \item $k(x,y,\xi)=\int_0^\infty\hat{h}_{t,r}(x,c+y,\xi)dc\in L^2(\RR^3)$.
\end{itemize}
Comparing relations \eqref{for:107} and \eqref{for:7}, analogous to \eqref{for:21} we have
\begin{align*}
    &\partial_x\partial_x\circ Y_1^4+\partial_y\partial_y\circ Y_1^4\notag\\
    &=-(Y_1X+Y_2V)^2-12Y_1^2+6Y_1(Y_1X+Y_2V)-V^2Y_1^2.
\end{align*}
Recall Definition \ref{de:1}.  Now in exactly the same manner as the proof in the first part of
Lemma \ref{le:8}, we have
\begin{enumerate}
  \item $(\partial_y\hat{h}_{t,r}\cdot x^5)_z\in C^0(\RR^2)$ for almost all $z\in\RR$,

  \medskip
  \item $\hat{h}_{t,r}\cdot y\rightarrow 0$ as $y\rightarrow \infty$ for almost all $(x,z)\in\RR^2$,
\end{enumerate}
which implies that
\begin{align*}
 U_1(k\sqrt{-1})=Y_1\hat{h}_{t,r}=\hat{g}_{t,r}.
\end{align*}
Taking inverse Fourier transformation, we get
\begin{align*}
  U_1p_{t,r}=g_{t,r},
\end{align*}
where $p_{t,r}=\check{k}\sqrt{-1}\in \mathbb{H}_{t,r}$. Using the cocycle equation we see that
\begin{align*}
  U_1f_{t,r}=U_2g_{t,r}=U_2U_1p_{t,r}=U_1U_2p_{t,r},
\end{align*}
which implies that $U_2p_{t,r}=f_{t,r}$ by Remark \ref{re:5}, Proposition \ref{cor:1} and Howe-moore. Then we get a common solution $p_{t,r}$. Since $\Pi\mid_{\mathbb{S}_2}$ has no-nontrivial $\RR^2$-invariant vectors, by Corollary \ref{cor:3} we have
\begin{align*}
    \norm{p_{t,r}}\leq C\norm{g}_{6}.
\end{align*}
In exactly the same manner as that of \eqref{for:5}, we get a global solution $p\in\mathcal{H}$ to the cocycle equation which satisfies the estimates
as claimed.
\section{Proof of Theorem \ref{th:8}}\label{sec:22}
\textbf{\emph{Proof of \eqref{for:79}}} Since the Weyl group operates simply transitive on the set of Weyl chambers, we may assume that the pair $\mathfrak{u}_{i,j}$ and are $\mathfrak{u}_{1,2}$ and $\mathfrak{u}_{3,4}$. Let $S$ be the closed subgroup generated by $U_{1,2}$, $U_{2,1}$, $U_{1,3}$, $U_{2,3}$, $U_{1,4}$ $U_{2,4}$ and $U_{3,4}$. Then $S$ is isomorphic to the group $(SL(2,\RR)\ltimes \RR^2)\ltimes\RR^3$, where the action of $SL(2,\RR)\ltimes \RR^2$ on $\RR^3$ is as defined in Section \ref{sec:10}. Thanks to Howe-moore again, we can apply \eqref{for:5} of Proposition \ref{po:1} to $\pi\mid_{S}$, which states that there exists a commom solution $h\in \mathcal{H}$ to the cocycle equation $\mathfrak{u}_{1,2}f=\mathfrak{u}_{3,4}g$. Finally, the Sobolev estimates of $p$ follow immediately from Theorem \ref{th:6}.

\medskip
\noindent\textbf{\emph{Proof of \eqref{for:26}}} If $\mathfrak{u}_{m,l}=\mathfrak{u}_{k,\ell}$, then we can assume the pair are $\mathfrak{u}_{1,2}$ and $\mathfrak{u}_{1,3}$, the cocycle equation is $\mathfrak{u}_{1,2}f=\mathfrak{u}_{1,3}g$ and $g=\mathfrak{u}_{1,3}p$. Smoothness of $p$ follows from Theorem \ref{th:6}. Recall Definition \ref{de:2}. Then
Howe-moore shows that we can apply Lemma \ref{le:7} to $\pi\mid_{G_3}$ which shows that there is a commom solution $h\in \mathcal{H}$ to the cocycle equation $\mathfrak{u}_{1,2}f=\mathfrak{u}_{1,3}g$.

If $\mathfrak{u}_{m,l}\neq\mathfrak{u}_{k,\ell}$, we can assume the triple are $\mathfrak{u}_{i,j}=\mathfrak{u}_{1,2}$, $\mathfrak{u}_{k,\ell}=\mathfrak{u}_{1,3}$ and $\mathfrak{u}_{m,l}=\mathfrak{u}_{1,4}$; or are $\mathfrak{u}_{i,j}=\mathfrak{u}_{1,2}$ and $\mathfrak{u}_{k,\ell}=\mathfrak{u}_{3,2}$ and $\mathfrak{u}_{m,l}=\mathfrak{u}_{4,2}$. For the former case, let $S_1$ be the closed subgroup generated by $U_{1,2}$, $U_{2,1}$, $U_{1,3}$, $U_{2,3}$, $U_{1,4}$ $U_{2,4}$ and $U_{3,4}$; for the latter one, let $S_1$ be the closed subgroup generated by $U_{1,2}$, $U_{2,1}$, $U_{3,2}$, $U_{3,1}$, $U_{4,1}$ $U_{4,2}$ and $U_{4,3}$. Then the remaining steps are exactly the same as that in \eqref{for:79} by changing $S$ to $S_1$ and changing \eqref{for:5} of Proposition \ref{po:1} to \eqref{for:18} of Proposition \ref{po:1}.

\section{Proof of Theorem \ref{th:9}}\label{sec:23}

\subsection{Dual representations of  $SL(2,\RR)\ltimes\RR^2$} In Lemma \ref{le:1}, by the change of variable $(x,\lambda)=(x,x^{-1}y) $ we have the models $\overline{\mathcal{H}}_{t}=L^2(\RR^2,\abs{x}dxd\lambda)$. The group action is defined by
\begin{gather*}
\bar{\rho_t}: SL(2,\RR)\ltimes\RR^2\rightarrow \mathcal{B}(\overline{\mathcal{H}}_t)\\
\bar{\rho_{t}}(v)f(x,y)=e^{i(v_2x-v_1\lambda x)}f(x,\lambda),\\
\bar{\rho_{t}}(g)f(x,y)=e^{\frac{ibt}{x(dx-b\lambda x)}}f\bigl(dx-b\lambda x,\frac{-cx+a\lambda x}{dx-b\lambda x}\bigl);
 \end{gather*}
where $(g,v)=\Big(\begin{pmatrix}a & b \\
c & d\\
 \end{pmatrix},\begin{pmatrix}v_1 \\
v_2\\
 \end{pmatrix}\Big)\in SL(2,\RR)\ltimes \RR^2$. Computing derived representations, we get
\begin{gather}\label{for:40}
 X=-x\partial_x+2\lambda\partial_\lambda, \qquad V=-\partial_\lambda, \qquad Y_1=-x\lambda \sqrt{-1}\notag\\
 Y_2=x\sqrt{-1}, \qquad U=tx^{-2}\sqrt{-1}-\lambda x\partial_x+\lambda^2\partial_\lambda .
 \end{gather}

Taking the Fourier transformation on factor $\lambda$, we get the dual the \emph{dual  models} $\widehat{\overline{\mathcal{H}}_{t}}=L^2(\RR^2,\abs{x}dxdy)$, and the group action is defined by
\begin{gather*}
\widehat{\bar{\rho_t}}: SL(2,\RR)\ltimes\RR^2\rightarrow \mathcal{B}(\widehat{\overline{\mathcal{H}}_t})\\
\widehat{\bar{\rho_{t}}}(v)f(x,y)=\frac{1}{\sqrt{2\pi}}\int_{\RR}\bigl(\bar{\rho_t}(v)\check{f}(x,\lambda)\bigl) e^{-iy\lambda}d\lambda,\\
\widehat{\bar{\rho_{t}}}(g)f(x,y)=\frac{1}{\sqrt{2\pi}}\int_{\RR}\bigl(\bar{\rho_t}(g)\check{f}(x,\lambda)\bigl)e^{-iy\lambda}d\lambda,
 \end{gather*}
 where $\check{f}(x,\lambda)=\frac{1}{\sqrt{2\pi}}\int_{\RR}f(x,y)e^{iy\lambda}dy$.

Computing derived representations, we get
\begin{gather}\label{for:39}
 X=-2I-x\partial_x-2y\partial_y, \qquad Y_1=x\partial_y ,\notag\\
 Y_2=x\sqrt{-1}, \qquad V=-y\sqrt{-1} \notag\\
 U=tx^{-2}\sqrt{-1}-x\partial_x\partial_y\sqrt{-1}-2\partial_y\sqrt{-1}-y\partial_y\partial_y\sqrt{-1}.
 \end{gather}
\subsection{Solvability in the dual models}Denote by $d\mu(x)=\abs{x}dx$. Recall Definition \ref{de:1}.
\begin{corollary}\label{cor:4}
 For any irreducible component  $(\bar{\rho_{t}},\,\overline{\mathcal{H}}_{t})$ of $SL(2,\RR)\ltimes \RR^2$ we have
 \begin{enumerate}
 \item \label{for:73}if $g\in \overline{\mathcal{H}}^\infty_{t}$ and $\int_{\RR}g(x,\lambda)d\lambda=0$ for almost all $x\in\RR$ (with respect to $\mu$), then the cohomological equation $Vf=Y_2g$ has a solution $f\in \overline{\mathcal{H}}_{t}^\infty$ satisfying
\begin{align*}
  \norm{f}_s\leq C_{s}\norm{g}_{s+7},\qquad \forall\,s\geq 0.
\end{align*}

\medskip
\item \label{for:74}Suppose $g\in \overline{\mathcal{H}}_{t}$ and $Y_1g\in \overline{\mathcal{H}}_{t}$.  If the cohomological equation $Vf=g$ has a solution $f\in \overline{\mathcal{H}}$, then $\int_{\RR}g(x,\lambda)d\lambda=0$ for almost all $x\in\RR$ (with respect to $\mu$).
\medskip

 \item \label{for:75}if $g\in (\widehat{\overline{\mathcal{H}}})_{t}^\infty$ and $\lim_{y\rightarrow 0}g(x,y)=0$ for almost all $x\in\RR$ (with respect to $\mu$), then the cohomological equation $Vf=Y_2^2g$ has a solution $f\in (\widehat{\overline{\mathcal{H}}})_{t}^\infty$ satisfying
\begin{align*}
  \norm{f}_s\leq C_{s}\norm{g}_{s+7},\qquad \forall\,s\geq 0.
\end{align*}

\medskip
  \item \label{for:76}Suppose $g\in (\widehat{\overline{\mathcal{H}}})_{t}$ and $Y_1g\in(\widehat{\overline{\mathcal{H}}})_{t}$.  If the cohomological equation $Vf=g$ has a solution $f\in \widehat{\overline{\mathcal{H}}}_{t}$, then $\lim_{y\rightarrow 0}g(x,y)=0$ for almost all $x\in\RR$ (with respect to $\mu$).
\end{enumerate}
\end{corollary}
\begin{proof} \eqref{for:73} and \eqref{for:74} is a direct consequence of \eqref{for:53} of Theorem \ref{po:1} and \eqref{for:109} of Lemma \ref{le:9}.
Let
\begin{align*}
\check{g}(x,\lambda)=\frac{1}{\sqrt{2\pi}}\int_{\RR}g(x,y)e^{iy\lambda}dy
\end{align*}
for any $g\in\widehat{\overline{\mathcal{H}}}_{t}$. \eqref{for:108} of Lemma \eqref{le:9} implies that if $Y_1g\in \widehat{\overline{\mathcal{H}}}_{t}$, then there exists a full measure set $\omega_g$ (with respect to $\mu$) such that $\check{g}_x\in L^1(\RR)$ for all $x\in\omega_g$. Hence for all $x\in\omega_g$, $g_x$ is a continuous function over $\RR$ and $g_x(y)=\int\check{g}_xd\lambda$; moreover, for any $x\in\omega_g$, $\int\check{g}_xd\lambda=0$ if and only if $g_x(0)=0$. Then the claim is a direct consequence of \eqref{for:73} and \eqref{for:74}.
\end{proof}
We now make a slight digression to prove an important lemma which is important for the sequel.
\begin{lemma}\label{le:13}
For any dual irreducible component $(\widehat{\bar{\rho_t}},\,\widehat{\overline{\mathcal{H}}}_{t})$ of $SL(2,\RR)\ltimes \RR^2$, if $g\in (\widehat{\overline{\mathcal{H}}}_t)^2$, then
\begin{enumerate}
\item \label{for:90}for almost all $x\in\RR$ (with respect to the measure $\mu$)
\begin{align*}
 \norm{(Y_2^2g)_x}_{C^0}\leq C\norm{(Y_2^2g)_x}_{L^2(\RR)}+C\norm{(Y_1Y_2g)_x}_{L^2(\RR)},
\end{align*}

\medskip
\item \label{for:91}$(Y_2^2g)_{y_1}\rightarrow (Y_2^2g)_{y}$ as $y_1\rightarrow y$ in $L^2(\RR, d\mu)$ for any $y\in\RR$ with respect to the norm topology.
 \end{enumerate}
\end{lemma}
\begin{proof} \textbf{\emph{Proof of \eqref{for:90}}} Using relations in \eqref{for:39} we get
\begin{align}\label{for:92}
\partial_y(Y_2^2g)=Y_2(Y_1g)\sqrt{-1}.
\end{align}
Then there exists a full measure set $\Omega_g\subset\RR$ (with respect to $d\mu$) such that $(Y_2^2g)_x\in L^2(\RR)$ and $\bigl(\partial_y(Y_2^2g)\bigl)_x\in L^2(\RR)$ for any $x\in \Omega_g$. Further,  by using Sobolev embedding theorem and \eqref{for:92}, for all $x\in \Omega_g$ we have
\begin{align}\label{for:93}
\norm{(Y_2^2g)_x}_{C^0}&\leq C\norm{(Y_2^2g)_x}_{L^2(\RR)}+C\norm{\bigl(\partial_y(Y_2^2g)\bigl)_x}_{L^2(\RR)}\notag\\
&=C\norm{(Y_2^2g)_x}_{L^2(\RR)}+C\norm{(Y_1Y_2g)_x}_{L^2(\RR)},
\end{align}
which proves the claim.

\medskip

\noindent\textbf{\emph{Proof of \eqref{for:91}}} It follows from \eqref{for:93} that
\begin{align}\label{for:98}
&\norm{(Y_2^2g)_{y}(x)}^2_{L^2(\RR,d\mu(x))}\notag\\
&=\int_{\RR} |(Y_2^2g)(x,y)|^2d\mu(x)\leq \int_{\RR} \norm{(Y_2^2g)_x}_{C^0}^2d\mu(x)\notag\\
&\overset{\text{(*)}}{\leq} 2C\int_{\RR} |(Y_2^2g)(x,y)|^2dyd\mu(x)+2C\int_{\RR} |(Y_2Y_1g)(x,y)|^2dyd\mu(x)\notag\\
&\overset{\text{(**)}}{=}2C\norm{Y_2^2f}^2+2C\norm{Y_2Y_1f}^2.
\end{align}
\eqref{for:93} shows that
\begin{align*}
(Y_2^2f)_{y_1}(x)\rightarrow (Y_2^2f)_{y}(x),\quad \text{ as }y_1\rightarrow y
\end{align*}
for any $x\in \Omega_g$. Let
\begin{align*}
 h(x)=C\norm{(Y_2^2g)_x}_{L^2(\RR)}+C\norm{(Y_1Y_2g)_x}_{L^2(\RR)}.
\end{align*}
Then $(*)$ and $(**)$ in \eqref{for:98} show that $h\in L^2(\RR,\mu)$. Hence the conclusion follows directly from dominant convergence theorem.
\end{proof}

\subsection{Direct integrals with respect to $\{\exp(tV)\}_{t\in\RR}$}\label{sec:13} For any unitary representation $(\pi,\mathcal{H})$ of $SL(2,\RR)\ltimes\RR^2$ without non-trivial $\RR^2$-invariant vectors, we have a direct direct integral decomposition: $v=\int_{\RR} v_t d\nu(t)$ for any $v\in\mathcal{H}$, where $v_t\in \widehat{\overline{\mathcal{H}}_{t}}$ and $\nu$ is a Borel measure on $\RR$. Note that $\widehat{\overline{\mathcal{H}}_{t}}=L^2(\RR^2,d\mu(x)dy)$ for each $t\in\RR$ and $d\mu(x)=\abs{x}dx$. From the discussion in Section \ref{sec:5}, we see that
\begin{align*}
 \mathcal{H}=L^2\bigl(\RR,L^2(\RR^2,d\mu(x)dy),d\nu\bigl).
\end{align*}
Let
\begin{align*}
  \widetilde{\mathcal{H}}=L^2\bigl(\RR^3,\CC, d\mu dyd\nu\bigl)\quad\text{ and }\quad\bar{\mathcal{H}}=L^2\bigl(\RR,L^2(\RR^2,\CC,d\mu d\nu),dy\bigl).
\end{align*}
We define a unitary representations $(\bar{\pi},\bar{\mathcal{H}})$ and $(\widetilde{\pi},\widetilde{\mathcal{H}})$ of $\{\exp(tV)\}_{t\in\RR}$ as follows:
\begin{align*}
    \bar{\pi}(\exp(tV))\bigl(\int  f_{y}dy\bigl)&=\int f_{y}\cdot e^{-y\sqrt{-1}}dy\qquad\text{ and }\\
    \widetilde{\pi}(\exp(tV))g(x,y,t)&=g(x,y,t)\cdot e^{-y\sqrt{-1}}
\end{align*}
An argument similar to the one in Section \ref{sec:18} shows that the maps
\begin{align*}
L^3(\RR^3,\CC,d\mu dyd\nu)&\xrightarrow{\mathcal{F}} L^2\bigl(\RR,L^2(\RR^2,\CC,d\mu dy),d\nu\bigl)\\
h(x,y,t)&\rightarrow \int_\RR \mathcal{F}(t)dt\
\end{align*}
 where  $\mathcal{F}(h)(t)(x,y)=h(x,y,t)$, and
 \begin{align*}
L^3(\RR^3,\CC,d\mu dyd\nu)&\xrightarrow{\mathcal{F}_1} L^2\bigl(\RR,L^2(\RR^2,\CC,d\mu d\nu),dy\bigl)\\
h(x,y,t)&\rightarrow \int_\RR \mathcal{F}_1(y)dy\
\end{align*}
 where $\mathcal{F}_1(h)(y)(x,t)=h(x,y,t)$ for any $h\in L^3(\RR^3,\CC,d\mu dyd\nu)$ are  unitary isomorphisms over $\{\exp(tV)\}_{t\in\RR}$. Hence  $(\pi,\mathcal{H})$ and $(\bar{\pi},\bar{\mathcal{H}})$ are unitarily equivalent over $\{\exp(tV)\}_{t\in\RR}$.

 \bigskip
We are now in a position to proceed with the proof of Theorem \ref{th:9}.
\subsection{Proof of Theorem \ref{th:9}} In the proof we assume that $\mathfrak{u}_{i,j}=\mathfrak{u}_{1,2}$ since the Weyl group acts simply transitive on all the one parameter subgroups $U_{i,j}$.

\medskip
\noindent\textbf{\emph{Proof of \eqref{for:94}}} Note that $G_3$  is isomorphic to $SL(2,\RR)\ltimes\RR^2$ and $\pi\mid_{G_3}$ has no $\RR^2$-invariant vectors by Howe-moore.  Then the  claim is a direct consequence of the discussion in Section \ref{sec:13} and Remark \ref{re:3}.
\medskip

\noindent\textbf{\emph{Proof of \eqref{for:95}}} Note that  $\mathfrak{u}_{k,\ell}\in E_{1,2}$ if and only if there exists $3\leq m\leq n$ such that
$\mathfrak{u}_{k,\ell}=\mathfrak{u}_{1,m}$ or $\mathfrak{u}_{k,\ell}=\mathfrak{u}_{m,2}$. Without loss of generality, we assume that $\mathfrak{u}_{k,\ell}=\mathfrak{u}_{1,3}$. We restrict
ourselves to the subrepresentation $(\pi\mid_{G_3},\,\mathcal{H})$. Recall Definition \ref{de:2}. Using arguments in Section \ref{sec:13} for any $h\in \mathcal{H}$ we can write $h=\int h_ydy$ where $h_y(x,t)=\bar{h}(x,y,t)$ for some $\bar{h}\in L^2(\RR^3, \CC, d\mu(x)dyd\nu(t))$.

For almost all $(x,t)\in\RR^2$ (with respect to the measure $d\mu(x)d\nu(t)$), it follows from \eqref{for:90} of Lemma \ref{le:13} that
\begin{align}\label{for:99}
    &\abs{(\mathfrak{u}_{1,3}^2g)_{y_1}(x,t)}\leq\norm{(\overline{\mathfrak{u}_{1,3}^2g})_{t}(x,\cdot)}_{C^0}\notag\\
    &\leq C\norm{(\overline{\mathfrak{u}_{1,3}^2g})_{t}(x,\cdot)}_{L^2(\RR)}+C\norm{(\overline{\mathfrak{u}_{2,3}\mathfrak{u}_{1,3}g})_{t}(x,\cdot)}_{L^2(\RR)}.
\end{align}
Set $q(x,t)=\norm{(\overline{\mathfrak{u}_{1,3}^2g})_{t}(x,\cdot)}_{L^2(\RR)}+\norm{(\overline{\mathfrak{u}_{2,3}\mathfrak{u}_{1,3}g})_{t}(x,\cdot)}_{L^2(\RR)}$. Then
\begin{align*}
\int_{\RR^2}q(x,t)^2d\mu(x)d\nu(t)&\leq 2\int_{\RR^2}\bigl(\int_{\RR}|(\overline{\mathfrak{u}_{1,3}^2g})(x,y,t)|^2dy\bigl)d\mu(x)d\nu(t)\\
&+2\int_{\RR^2}\bigl(\int_{\RR}|(\overline{\mathfrak{u}_{2,3}
\mathfrak{u}_{1,3}g})(x,y,t)|^2dy\bigl)d\mu(x)d\nu(t)\\
&\leq 2\norm{\overline{\mathfrak{u}_{1,3}^2g}}_{L^2(\RR^3, \CC,d\kappa)}+2\norm{\overline{\mathfrak{u}_{2,3}\mathfrak{u}_{1,3}g}}_{L^2(\RR^3, \CC,d\kappa)}\\
&=2(\norm{\mathfrak{u}_{1,3}^2g}^2+\norm{\mathfrak{u}_{2,3}\mathfrak{u}_{1,3}g}^2)\\
&\leq 2\norm{g}_2^2
\end{align*}
where $d\kappa=d\mu(x)dyd\nu(t)$. Hence $q(x,t)\in L^2(\RR^2,\,\CC,\,d\mu(x)d\nu(t))$. Now noting that \eqref{for:91} of Lemma \ref{le:13} implies that for almost all $(x,t)\in\RR^2$ (with respect to the measure $d\mu(x)d\nu(t)$),
\begin{align*}
(\mathfrak{u}_{1,3}^2g)_{y_1}(x,t)=(\overline{\mathfrak{u}_{1,3}^2g})_{t}(x,y_1)\rightarrow (\overline{\mathfrak{u}_{1,3}^2g})_{t}(x,y)
=(\mathfrak{u}_{1,3}^2g)_{y}(x,t)
\end{align*}
as $y_1\rightarrow y$. By using dominant convergence theorem, we see that
\begin{align*}
(\mathfrak{u}_{1,3}^2g)_{y_1}(x,t)\rightarrow (\mathfrak{u}_{1,3}^2g)_{y}(x,t),
\end{align*}
as $y_1\rightarrow y$ in  $L^2(\RR^2, \CC,d\mu(x) d\nu)$. Immediately, we find that the function $y\rightarrow \norm{g_y}^2_{L^2(\RR^2, d\mu(x) d\nu)}$ is continuous on $\RR$. Then the claim follows directly from Lemma \ref{le:10} and Remark \ref{re:3}.

\medskip
\noindent\textbf{\emph{Proof of \eqref{for:96}}} Arguments in the previous part shows that without loss of generality, we just need to prove that $\lim_{\chi\rightarrow 0}D_{1,2}(\mathfrak{u}_{1,3}^2g)(\chi)=0$. Again, we will focus on the subrepresentation $(\pi\mid_{G_3},\,\mathcal{H})$.
Since $f$ is the solution of the cohomological equation $\mathfrak{u}_{1,2}f=g$, $f$ is smooth from \eqref{for:33} of Theorem \ref{th:6}. Then $\mathfrak{u}_{1,3}^2f$ satisfies the equation $\mathfrak{u}_{1,2}\mathfrak{u}_{1,3}^2f=\mathfrak{u}_{1,3}^2g$. Similar to the previous case, we write $\mathfrak{u}_{1,3}^2g=\int (\mathfrak{u}_{1,3}^2g)_ydy$ where $(\mathfrak{u}_{1,3}^2g)_y(x,t)=(\overline{\mathfrak{u}_{1,3}^2g})(x,y,t)$ and $(\overline{\mathfrak{u}_{1,3}^2g})\in L^2(\RR^3, \CC, d\mu(x)dyd\nu(t))$. Then \eqref{for:76} of Corollary \ref{cor:4} shows that for almost all $(x,t)\in\RR^2$ (with respect to the measure $d\mu(x)d\nu(t)$)
\begin{align}\label{for:8}
 \lim_{y\rightarrow 0}(\mathfrak{u}_{1,3}^2g)_y(x,t)=\lim_{y\rightarrow 0}(\overline{\mathfrak{u}_{1,3}^2g})_t(x,y)=0
\end{align}
By earlier arguments in \eqref{for:95}, we see that $(\mathfrak{u}_{1,3}^2g)_{y}(x,t)\rightarrow 0$ as $y\rightarrow 0$ in $L^2(\RR^2, \CC,d\mu(x) d\nu)$, which proves the claim.
\medskip

\noindent\textbf{\emph{Proof of \eqref{for:97}}}  Without loss of generality, we just need to prove the claim for the pair $\mathfrak{u}_{1,3}$ and $\mathfrak{u}_{2,4}$. We also assume notations in \eqref{for:96}.
Using subrepresentation $(\pi\mid_{G_3},\,\mathcal{H})$, the assumption means that for almost all $(x,t)\in\RR^2$ (with respect to the measure $d\mu(x)d\nu(t)$)
$\lim_{y\rightarrow 0}(\mathfrak{u}_{1,3}^2g)_y(x,t)=0$. By using \eqref{for:8} and \eqref{for:75} of Corollary \ref{cor:4} we see that the equation
$\mathfrak{u}_{1,2}f=\mathfrak{u}_{1,3}g$ has a solution $f\in \mathcal{H}$. Similarly, by using subrepresentation $(\pi\mid_{H_4},\,\mathcal{H})$, we get a solution
$k\in \mathcal{H}$ for the equation $\mathfrak{u}_{1,2}k=\mathfrak{u}_{2,4}g$. The smoothness of $k$ and $f$ follow from Theorem \ref{th:6}. Then
\begin{align*}
 \mathfrak{u}_{1,3}\mathfrak{u}_{1,2}k=\mathfrak{u}_{1,3}\mathfrak{u}_{2,4}g=\mathfrak{u}_{2,4}\mathfrak{u}_{1,3}g=\mathfrak{u}_{2,4}\mathfrak{u}_{1,2}f,
\end{align*}
which implies $k$ and $f$ satisfy the cocycle equation
\begin{align}\label{for:11}
\mathfrak{u}_{1,3}k=\mathfrak{u}_{2,4}f
\end{align}
thanks to Howe-moore. By Theorem \ref{th:8}, there exists $h\in \mathcal{H}^\infty$
satisfying
\begin{align*}
\mathfrak{u}_{2,4}h=k\quad\text{ and }\quad\mathfrak{u}_{1,3}h=f.
\end{align*}
Substituting the relations into \eqref{for:11}, we get the equation
\begin{align*}
\mathfrak{u}_{1,2}\mathfrak{u}_{2,4}h=\mathfrak{u}_{2,4}g.
\end{align*}
By Howe-moore again, we find that $\mathfrak{u}_{1,2}h=g$. Thus we finish the proof.
\section{Proof of Theorem \ref{th:10}}
The proof is standard and similar arguments appeared in \cite{Spatzier1}, \cite{Mieczkowski} and \cite{Ramirez}. Let $\beta$ be a cocycle over the $V$-action on $G/\Gamma$. Restricted to the $U$-action on $G/\Gamma$, $\beta$ is also a cocycle. Then it follows from the result in \eqref{for:79} that there is a smooth transfer function $p$ that satisfies
\begin{align*}
    \beta(u,x)=p(u\cdot x)+c(u)-p(x)
\end{align*}
for any $u\in U$ and $x\in G/\Gamma$, where $c:U\rightarrow \CC$ is a constant cocycle. For any $v\in V$, let
\begin{align*}
    \beta^*(v,x)=\beta(v,x)-p(v\cdot x)+p(x).
\end{align*}
Using the definition of cocycle, we see that $\beta^*$ is also a cocycle over $V$-action. Then
\begin{align*}
 \beta^*(v,x)&= \beta^*(uv,x)-\beta^*(u,v\cdot x)=\beta^*(vu,x)-\beta^*(u,v\cdot x)\\
 &=\beta^*(v,u\cdot x)+\beta^*(u,x)-\beta^*(u,v\cdot x)\\
 &=\beta^*(v,u\cdot x)
\end{align*}
is a $U$-invariant smooth function on $G/\Gamma$ for every $v\in V$. By Howe-moore, it is constant. Therefore, setting
$c'(v)=\beta(v,x)-p(v\cdot x)+p(x)$, we have shown that $p$ satisfies
\begin{align*}
\beta(v,x)-p(v\cdot x)+p(x)=c'(v)
\end{align*}
for all $v\in V$ and $x\in G/\Gamma$. It is clear that $c'=c$ on $U$.

\end{document}